\documentclass[12pt,article]{amsart}
\usepackage{mathrsfs}
\usepackage{amssymb}
\usepackage{amsfonts}
\usepackage{amsbsy}
\usepackage{latexsym}
\usepackage{amssymb,latexsym,amsmath,amsthm}
\usepackage{framed}
\usepackage{diagbox}
\usepackage{tcolorbox}
\usepackage[colorlinks,linkcolor=blue]{hyperref}
\usepackage{graphicx}
\usepackage{xcolor}
\usepackage{epstopdf}
\usepackage{bm}
\usepackage{graphicx}
\theoremstyle{plain}

 \theoremstyle{remark} 

\newtheorem {theo} {\bf Theorem} [section]
\newtheorem {prop} [theo] {\bf Proposition}

\newtheorem {lem} [theo] {\bf Lemma}
\newtheorem {note} [theo] {\bf Note}
\newtheorem {defi} {\bf Definition}[section]
\newtheorem{exam} {\bf Example}[section]

\newtheorem{rem}{\bf Remark}[section]

\newtheorem{appr} {\bf Approach}[section]

\newtheorem{sche} {\bf Scheme}[section]

\numberwithin{equation}{section}
\begin{document}
\title[Single-shot  phase retrieval]{Single-shot     phase retrieval: a holography-driven problem in Sobolev  space}
\author{Youfa Li}
\address{College of Mathematics and Information Science\\
Guangxi University,  Nanning, China }
\email{youfalee@hotmail.com}
\author{Shengli Fan}
\address{CREOL College of Optics \& Photonics\\
University of Central Florida,  Orlando, FL 32816}
\email{shengli.fan@knights.ucf.edu}
\author{Deguang Han}
\address{Department of Mathematics,
University of Central Florida,  Orlando, FL 32816}
\email{Deguang.Han@ucf.edu}
%\author{Wenchang  Sun}
%\address{College  of Mathematics,
%Nan Kai University, Tianjin, China}
%\email{sunwch@nankai.edu.cn}
%\email{shengli.fan@knights.ucf.edu}
%\author{Deguang Han}
%\address{Department of Mathematics,
%University of Central Florida,  Orlando, FL 32816}
%\email{Deguang.Han@ucf.edu}
%\author{Yanfen Huang}
%\address{College of Mathematics and Information Science\\
%Guangxi University,  Nanning, China }
%\email{hyfbqy@163.com}
\thanks{Youfa Li is partially supported by Natural Science Foundation of China (Nos: 61961003, 61561006, 11501132),  Natural Science Foundation of Guangxi (Nos: 2019GXNSFAA185035, 2016GXNSFAA380049) and  the talent project of  Education Department of Guangxi Government  for Young-Middle-Aged backbone teachers.
Deguang Han  is  supported by the NSF grant   DMS-2105038.
%Wenchang Sun
%was  supported by the
%National Natural Science Foundation of China (11525104, 11531013 and 11761131002)
}
\keywords{phase retrieval, holography, admissible  reference wave,  Sobolev space, single-shot intensity, approximation to  quasi-interference intensity}
\subjclass[2010]{Primary 42C40;  94A12}

\date{\today}
%\keywordsÒ°

\begin{abstract}
The  phase-shifting digital holography (PSDH) is a widely used   approach for recovering target  signals  by their
interference (with reference signals) intensity measurements. Such reference signals  are   traditionally
from multiple shots (or multiple phase-shiftings) of the  reference wave. However, the imaging of dynamic target  signals  requires a single-shot PSDH approach, namely, such an approach
depends only  on the intensity measurements   from the interference with
the single reference signal  (being from the single phase-shifting of the reference wave).
 In this paper, based on
the uniform  admissibility of  plane (or spherical) reference wave and  the interference intensity-based   approximation to
quasi-interference intensity,
the nonnegative refinable function is applied  to  establish the  single-shot PSDH  in Sobolev space. Our approach is conducted by
the intensity measurements   from the interference of the target  signal  with  a single  reference signal. The  main results
imply that  the approximation version from such a  single-shot  approach converges    exponentially to the signal    as the  level   increases.
Moreover, like the transport of intensity equation (TIE),  our results  can be   interpreted from the perspective of intensity difference.
\end{abstract}
\maketitle

\section{Introduction}\label{introductionsection}
Phase retrieval (PR) is a nonlinear   problem  that seeks to
reconstruct   a target  signal   $f\in \mathcal{H}$, \emph{up to  the potential ambiguity},  from   the intensity measuements
\begin{align}\notag \begin{array}{lll}  |\langle f, \textbf{a}_{k}\rangle|^{2}, \ k\in \Gamma, \end{array}\end{align}
where $\mathcal{H}$ is the signal space and  $\textbf{a}_{k}$ (an operator) is called  the measurement vector (c.f. \cite{Ba1,Ba2,Xu2}). PR has been widely investigated  in function spaces   (e.g. \cite{Alaifari,QiyuPR,QiyuPR1,Gaussian,Lai,Pohl,Iserles}) and many optical
problems  such as   coherent diffraction imaging (\cite{crystallography,reasonforPR}),   quantum tomography (\cite{Heinosaarri}), ptychography (\cite{Fannjiang}) and holography (\cite{holographic}). The most classical PR (c.f. \cite{Fienup1,Fienup2}) is achieved by the Fourier intensity measuements, namely, the measurement vector is chosen as the Fourier sampling operator. Our PR problem  here   is derived from the  phase-shifting digital holography (PSDH) introduced by Yamaguchi and Zhang \cite{Zhang}. Before introducing our model and main results, the background on PSDH is necessary.

\subsection{Background on  phase-shifting digital holography}
Holography was invented by Dennis Gabor in 1948 to improve the resolution of the electron microscope (c.f.\cite{D.Gabor}).
From then on, digital  holography (DH) has been blooming  into a very broad field with numerous applications including  bio-medicine, microscopy,
cyber-physical security and so on. Readers can refer to the most recent  roadmap paper \cite{Roadmap} for an overview of such applications.
These applications depend on effective DH configurations such as off-axis Fresnel holography (\cite[Section 5.1]{holographic}), Fourier holography (\cite[Section 5.2]{holographic}) and in-line holography (\cite[Section 5.4]{holographic}). Among others, PSDH
is an effective  holography configuration  since it  removes  cross terms.
It   has gained much  attention over the last two decades (c.f. \cite{Roadmap,holographic}).
By  \cite[Section 5.6]{holographic},
PSDH is mathematically  modeled as: to reconstruct a target signal  (or an object) $f$ by its interference (with the reference signal  $e^{\textbf{i}\theta}g$) intensity measurements:
\begin{align}\label{jh} I^{\theta}(\emph{\textbf{x}})=|f(\emph{\textbf{x}})+e^{\textbf{i}\theta}g(\emph{\textbf{x}})|^{2},\end{align}
where $g$ is termed as the reference wave.
Choosing $\theta=0, \pi/2, \pi, 3\pi/2$,  the four-shot (or four-step) PSDH (c.f. \cite{Yamaguchi0,WYZ}) admits the recovery formula
\begin{align}\label{fourshot}
f(\emph{\textbf{x}})=\frac{1}{4\bar{g}(\emph{\textbf{x}})}\big(I^{0}(\emph{\textbf{x}})-I^{\pi}(\emph{\textbf{x}})+\textbf{i}(I^{\pi/2}(\emph{\textbf{x}})-I^{3\pi/2}(\emph{\textbf{x}}))\big),
\end{align}
where $\bar{g}(\emph{\textbf{x}})$ is the complex conjugate of $g(\emph{\textbf{x}})$.
Note that \eqref{fourshot}  is not involved with  the cross terms: $f(\emph{\textbf{x}})\bar{g}(\emph{\textbf{x}})$ and $\bar{f}(\emph{\textbf{x}})g(\emph{\textbf{x}})$. Through the three-shot ($\theta=0, \pi/2, \pi$; c.f. \cite{Yamaguchi0} or \cite[(39)]{holographic}) or  two-shot (requiring additional conditions on intensities; c.f. \cite{Yamaguchi1}) approaches,
the recovery also can be achieved. Compared with phaseless sampling (PS), the following states that PSDH enjoys some nice properties.
%Note that the two-shot approach requires additional conditions such as the separate intensity exposures $|f|$ and $|g|$ of $f$ and $g$.

If  $e^{\textbf{i}\theta}g(\emph{\textbf{x}})$ is removed from \eqref{jh}, then  PSDH degenerates to the  PS problem (c.f. \cite{QiyuPR,Gaussian,Lai,LISUN,Wenchangsun, Iserles}) which requires to reconstruct $f$
by its  phaseless samples (non-interference intensity measurements) $\{|f(\emph{\textbf{x}})|^{2}: \emph{\textbf{x}}\in \Lambda\}$.
Recall that the PS results in the literature  hold for special signals  (e.g.
\cite{QiyuPR,Wenchangsun}), and  the recovery is achieved up to the potential  ambiguity
such as a unimodular scalar (e.g. \cite{QiyuPR,LISUN,Wenchangsun,Iserles}) and conjugation ambiguity (e.g.  \cite{Lai}).
Contrary to PS, PSDH  holds for any signal  and its recovery
result is not involved with any  ambiguity.
Some  PR problems (e.g. \cite{Robert1,Xuzhiqiang,Huangmeng}) in the field of  applied harmonic analysis is related  to
PSDH.

%If the signal  $f$ sits in a function space,
%and the recovery is achieved by  (the measurement vector is chosen as the Dirac distribution or its shift), then the corresponding PR is the so called phaseless sampling. Most recently, researchers are interested in such a PR problem in \textcolor[rgb]{0.00,0.07,1.00}{shift-invariant spaces}

\subsection{ The imaging of a dynamic object requiring  a single-shot  configuration and the recent optical  single-shot configurations}
Recall that the traditional configurations (including the above mentioned  four-, three-, two-shot configurations) for   PSDH
is time-consuming since one needs to switch the phase  plate (or mathematically choosing different $\theta$ in  \eqref{jh}).
 Such
configurations are ideally suited for static objects but not for dynamic ones especially not  for the vivo having very short life span  (c.f. \cite{Liguifang}).
Correspondingly, a single-shot (requiring  only one  fixed $\theta$ in \eqref{jh}) intensity  based reconstruction approach
is required for  the  dynamic imaging. In optics,  some recent single-shot  configurations are introduced
to  conduct PSDH.
Although the single-shot intensity measurements   are used indeed in   these configurations,  other additional  information are also necessary
or extra  optical devices are required to generate additional  phase-shifting  holograms.
The configurations in Nobukawa et.al \cite{Checkboard} and Zhang et.al \cite{Liguifang}  are such two typical examples.
In \cite[Fig. 1]{Liguifang},   a configuration for single-shot holography was   established   by the polarization-encoded information (obtained by
two polarization beam splitters (PBS) and four cameras, c.f. \cite[(6)]{Liguifang})  of both the  object and reference beams.
%Specifically, the components of the polarization-based (approximative) representations of the  two beams are extracted by
%two polarization beam splitters (PBS) and  four  charge coupled devices (CCD)  (c.f. \cite[(6)]{Liguifang}).
From the mathematical perspective, such a configuration indeed depends on  increasing  the dimensionality of the
information  space.  In \cite{Checkboard},   gratings
are applied to a single  reference beam  to generate multiple  phase-shifting  reference beams (or equivalently to generate
different $\theta$ in \eqref{jh}) and therefore to generate
multiple phase-shifting  holograms of multiple shots. From the mathematical point of view, such a method (c.f. \cite[(4)]{Checkboard})  essentially does not take too  much difference  from
the traditional holography \eqref{fourshot}.

\subsection{The single-shot  phase retrieval problem and the main result}\label{section3.1}
A dynamic (w.r.t time) object is commonly modeled as a function   $h(\emph{\textbf{z}}; t)$ where $\emph{\textbf{z}}$ and  $t$ represent  the location and  time, respectively  (e.g. \cite{Dynamicboject1, Dynamicboject}). In this paper, it is expected that the recovery approach holds for every time $t$.
Therefore,
from the mathematical perspective
a natural PR problem is, can one reconstruct a function  $f(\emph{\textbf{x}})=|f(\emph{\textbf{x}})|e^{\textbf{i}\theta(f(\emph{\textbf{x}}))}$ just by the single-shot   interference intensity
\begin{align}\label{intensitygeneral}  I(\emph{\textbf{x}})=|f(\emph{\textbf{x}})+g(\emph{\textbf{x}})|^{2} ?\end{align}
Summarizing what has been addressed above, the single-shot  model \eqref{intensitygeneral}
enjoys the following advantages:\\
(1) Compared with the traditional multi-shot methods such as  \eqref{fourshot},
the model \eqref{intensitygeneral} just requires the single-shot intensities and consequently it meets the requirements of
the imaging  of   dynamic objects.  \\
(2) The    holography in \cite{Liguifang}
requires two PBS and four cameras.  Unlike \cite{Liguifang},
\eqref{intensitygeneral} does not require the  polarization-encoded information.
Consequently, the PBS and additional cameras  are not necessary. One just  needs    a  camera to record the intensity measurements.\\
(3) Unlike  the    holography in
\cite{Checkboard}, the measurements in  \eqref{intensitygeneral}  are directly from single-shot   interference
and consequently one  does not need the     gratings (to generate
multiple phase-shifting  holograms of multiple shots).

Low cost  is an important requirement  for  evaluating a holography  setup (c.f. \cite{FAAlight}). The above items (2, 3)
imply that model \eqref{intensitygeneral} can cut the cost of the setup of imaging in terms of devices.
%Clearly,
%Therefore,  it is essentially different from \cite{Checkboard} and \cite{Liguifang} and will simplify the configuration for PSDH.
%Moreover,  such a model does not need
%the expensive optical devices: PBS and gratings which are the inevitable  expenses for the configurations in \cite{Checkboard,Liguifang}.
The following mathematically  states   that  \eqref{intensitygeneral} is absolutely  not trivial and  essentially different from  the traditional (multi-shot) PSDH problem.
%It is also quite different from the PR problems in \cite{Robert1,Robert2,Xuzhiqiang}.

\begin{rem}\label{ffffremark}  By \eqref{intensitygeneral} we have
\begin{align}\label{jhk0} I(\emph{\textbf{x}})=|f(\emph{\textbf{x}})|^{2}+f(\emph{\textbf{x}})\bar{g}(\emph{\textbf{x}})+\bar{f}(\emph{\textbf{x}})g(\emph{\textbf{x}})+|g(\emph{\textbf{x}})|^{2}.\end{align}
The cross terms: $f(\emph{\textbf{x}})\bar{g}(\emph{\textbf{x}})$
and $\bar{f}(\emph{\textbf{x}})g(\emph{\textbf{x}})$ are contained in   \eqref{jhk0}.
For any $\emph{\textbf{x}}\in \mathbb{R}^{d}$, it is impossible
to determine    $|f(\emph{\textbf{x}})|$ and the phase $\theta(f(\emph{\textbf{x}}))$ of $f(\emph{\textbf{x}})$ directly from the single-shot intensity $I(\emph{\textbf{x}})$.
Even though $|f(\emph{\textbf{x}})|$ is known, the most greatest difficulty is that   the term $f(\emph{\textbf{x}})\bar{g}(\emph{\textbf{x}})+\bar{f}(\emph{\textbf{x}})g(\emph{\textbf{x}})$ just
provides  the information $\cos(\theta(f(\emph{\textbf{x}}))-\theta(g(\emph{\textbf{x}})))$ but not  $\theta(f(\emph{\textbf{x}}))$.
Consequently,
the   phases  $\theta(f(\emph{\textbf{x}}_{1}))$ and $\theta(f(\emph{\textbf{x}}_{2}))$ are very difficult  to match together for any  two points $\emph{\textbf{x}}_{1}$ and $\emph{\textbf{x}}_{2}$.
 Such an obstacle does not exist in the traditional PSDH problem (e.g., see  \eqref{fourshot}) since the cross terms  are   removed therein through the multi-shot information.
% (2) Although  the PR problems addressed  in Beinert and  Plonka \cite{Robert1,Robert2} and
%  Gao, Sun, Wang and  Xu \cite{Xuzhiqiang} formally take some similarity as \eqref{intensitygeneral}, they are quite different indeed
%  from  \eqref{intensitygeneral}. Firstly, the signal to be reconstructed in \cite{Robert1,Xuzhiqiang}
%  is a vector in $\mathbb{C}^{N}$ which  is  essentially
%considered a linear function. Secondly, the measurements used  in  \cite[section 4.2]{Robert2}
%  are  the Fourier interference and noninterference  intensities. However,  the function $f$ in \eqref{intensitygeneral}
%  can be nonlinear and the required  measurements
%are  the (non-Fourier) single-shot interference intensities.
\end{rem}

The  function theory  of Sobolev space has many applications  in  optics including   holography (e.g. \cite{Enciso,Maxfield}).
Motivated   by the wavelet multilevel approximation (c.f. \cite{waveletBOOK2,waveletBOOK3,waveletBOOK1,Matla}),
we will  establish a mathematical theory on  the single-shot  PR  problem for   the functions in Sobolev space.

%Specifically, on a  $\Omega\subseteq \mathbb{R}^{d}$ we design uniformly  admissible plane (or spherical) reference waves,  denoted by $\{g_{N}\}^{\infty}_{N=1}$. Then at any  fixed  level $N$ an approximation to $f|_{\Omega}$ (the restriction of $f$ on $\Omega$) is  constructed. Such an approximation depends only on
%the single-shot intensity:
%\begin{align}\label{usedintensity}  I_{N}(\emph{\textbf{x}})=|f(\emph{\textbf{x}})+g_{N}(\emph{\textbf{x}})|^{2}.\end{align}

Before introducing  the main result    of this paper,
some definitions and  denotations are necessary. Throughout the paper,
the region of interest (ROI) on $\mathbb{R}^{d}$ is denoted by $\Omega$.
For a function $f$ on $\mathbb{R}^{d}$, its restriction on $\Omega$
is denoted by $f|_{\Omega}$.  The Sobolev space $H^s(\mathbb{R}^d), s\in \mathbb{R}$ consists of functions $f$
such that $\int_{\mathbb{R}^{d}}|\widehat{f}(\xi)|^{2} (1+\|\xi\|_{2}^{2})^{s}d\xi<\infty,$
where $\widehat{f}(\xi)=\int_{\mathbb{R}^{d}}f(\emph{\textbf{x}})e^{-\textbf{i}\emph{\textbf{x}}\cdot \xi}d\emph{\textbf{x}}$ is
the Fourier transform of $f$. The Sobolev smoothness exponent  of $f$ is defined to be $\nu_{2}(f)=\sup\{s\in \mathbb{R}: f\in H^{s}(\mathbb{R}^d)\}.$
For a function $\phi: \mathbb{R}^d\longrightarrow \mathbb{C}$, we say that it is refinable
if $\widehat{\phi}(2\xi)=\widehat{a}(\xi)\widehat{\phi}(\xi)$ where $\widehat{a}(\xi)$ is a $2\pi \mathbb{Z}^{d}$-periodic
function and it is  referred to as  the mask symbol of $\phi.$ We say that $\widehat{a}(\xi)$
has  $\kappa+1$ \textbf{sum rules} (c.f. \cite{waveletBOOK1}) if
\begin{align}\label{12345mmbcxx}\widehat{a}(\xi+2\pi\gamma)=\hbox{O}(\|\xi\|^{\kappa+1}_{2}), \xi\rightarrow \textbf{0},
%\quad \widehat{a}(\xi)=1+\hbox{O}(\|\xi\|^{\kappa+1}_{2}) \ \hbox{as} \
\end{align}
where $\gamma\neq\textbf{0}$ is any representative of the  cosets   $[2^{-1}\mathbb{Z}^{d}]/\mathbb{Z}^{d}$.
If the mask symbol  $\widehat{a}(\xi)$
has  $\kappa+1$ sum rules,   we also  say that the sum rule order $\hbox{sr}_{\phi}$ of $\phi$ is  $\kappa+1$.
%We denote by $\hbox{sr}_{\phi}$ the   sum rules that its mask symbol has.

The plane and spherical reference waves are commonly utilized in holography (c.f. \cite{Fourieroptics}). They are defined as
the complex-valued  functions  $g(\emph{\textbf{x}})=ae^{\textbf{i}\mathcal{K}\cdot \emph{\textbf{x}}}$ and   $g(\emph{\textbf{x}})=\frac{b}{\|\emph{\textbf{x}}\|_{2}}e^{\textbf{i}\nu\|\emph{\textbf{x}}\|_{2}}$ ($\mathcal{K}, \emph{\textbf{x}}\in \mathbb{R}^{d},  \nu\in \mathbb{R} , a, b>0 $), respectively. We will design  a class of  plane and spherical reference waves
(to be called uniformly admissible reference waves) such that the recovery is stable as the level increases.
The recovery results will be established in  Theorems \ref{firstmaintheorem} and \ref{secondmaintheorem}.
They are summarized as follows.

\begin{theo}\label{jielunzongjie}
Suppose that $\Omega\subseteq \mathbb{R}^{d}$ is a bounded ROI,   $\{g_{N}\}^{\infty}_{N=0}$  are the uniformly    admissible plane (or spherical) reference waves where   $N$ denotes  the level. Moreover, $\phi$ sitting in the Sobolev space  $ H^{s}(\mathbb{R}^{d})$ (with $s>d/2$)  is a nonnegative refinable  function such that  its  smoothness exponent  $\nu_{2}(\phi)>\varsigma>s$ and sum rule order  $\hbox{sr}_{\phi}>\varsigma$.
%its mask symbol has $\kappa+1 (>\varsigma)$   sum rules.
%
%$\phi$ is a nonnegative refinable  function in the Sobolev space $H^{\varsigma}(\mathbb{R}^{D})$   with $\varsigma>s>D/2$,  and   its mask symbol has $\kappa+1$ sum rules such that $\kappa+1>\varsigma$.
Then for any $f\in H^{\varsigma}(\mathbb{R}^{d})$, it holds that
\begin{align}\label{mainresult} \begin{array}{lll}
\Big\|f|_{\Omega}-S_{N}f|_{\Omega}\Big\|_{L^2(\mathbb{R}^d)}
=\hbox{O}(2^{-N\alpha/2}),
\end{array}\end{align}
where $\alpha=\min\{1, \varsigma-s\}$, and the  recovery version  $S_{N}f|_{\Omega}$ is derived from $\phi$
and the measurements of the single-shot  intensity  $I_{N}(\emph{\textbf{x}})=|f(\emph{\textbf{x}})+g_{N}(\emph{\textbf{x}})|^{2}$ on a finite   set.
\end{theo}

\begin{rem}
(1) As the level $N$ increases, $S_{N}f|_{\Omega}$ exponentially converges to $f|_{\Omega}$.
(2) As        the multi-shot holography, our recovery result is not involved with the ambiguity or  the phase matching in Remark \ref{ffffremark}.
\end{rem}

\begin{rem}\label{remmarkdd}
(1) Note that  $H^{\varsigma_{2}}(\mathbb{R}^{d})\subseteq H^{\varsigma_{1}}(\mathbb{R}^{d})$ for any $\varsigma_{1}\leq \varsigma_{2}$.
From the perspective of the  applicable  scope, the Sobolev  smoothness exponent $\varsigma(>d/2)$ in Theorem \ref{jielunzongjie}
should be as   close to $d/2$ as possible  such that more signals   can be recovered. (2)
It is required in Theorem \ref{jielunzongjie} that
\begin{align}\label{yaoqiu1234}\hbox{(i)} \  \min\{\nu_{2}(\phi), \hbox{sr}_{\phi}\}>d/2;
\ \hbox{(ii)} \ \phi \ \hbox{being nonnegative}. \end{align}
On the other hand,
as long as    $\phi$ satisfies \eqref{yaoqiu1234}, then Theorem \ref{jielunzongjie} holds for any $\varsigma$ and $s$ such that
$\min\{\nu_{2}(\phi), \hbox{sr}_{\phi}\}>\varsigma>s>d/2$.
\end{rem}

\subsection{On the construction  of   the refinable function $\phi$ in Theorem \ref{jielunzongjie}}\label{gouzaojiaxihanshu}
It follows from Remark \ref{remmarkdd} (2) that the refinable function $\phi$   is required to satisfy \eqref{yaoqiu1234}.
The purpose of this subsection is to explain that \eqref{yaoqiu1234} is  weak and the construction of   $\phi$  can be achievable. The cardinal B-splines are helpful for our discussion.

The cardinal B-splines $\{B_{m}\}_{m\in \mathbb{N}}$ on $\mathbb{R}$ are an interesting class of nonnegative  refinable functions  (c.f. \cite{waveletBOOK2,waveletBOOK1,Wenchangsun}), where
$B_{m}$ is the $m$th cardinal B-spline
defined by
$ B_{m}:=\overbrace{\chi_{(0,1]}\star\ldots\star\chi_{(0,1]}}^{m\ \tiny{\hbox{copies}}}$
with $\chi_{(0,1]}$    the characteristic function of the interval  $(0,1]$
and $\star$  the convolution.
Through the  direct  calculation we have $\hbox{supp}(B_{m})=[0, m]$  and
  $\widehat{B_{m}}(\xi)=e^{-\textbf{i} \frac{m\xi}{2}}(\frac{\sin\xi/2}{\xi/2})^{m}. $ From this,   $B_{m}$ is  refinable such that
\begin{align}\label{lithography} \widehat{B_{m}}(2\xi)=(\frac{1+e^{-\textbf{i}\xi}}{2})^{m}\widehat{B_{m}}(\xi).\end{align}
Moreover, $\hbox{sr}_{B_{m}}=m$ and $\nu_{2}(B_{m})=m-1/2$
(c.f. \cite[section 6.1.1]{waveletBOOK1}).

\subsubsection{\textbf{The Sobolev smoothness exponent  $\min\{\nu_{2}(\phi), \hbox{sr}_{\phi}\}>d/2$ does not necessarily implies that $\phi$ is  smooth}}
 If $\varsigma>d/2$ then the functions in $H^{\varsigma}(\mathbb{R}^d)$ are continuous (c.f. \cite[Chapter 9.1]{Matla}, \cite[section 1]{LiSCI}).
Therefore,  $\phi$  in Theorem \ref{jielunzongjie} is  continuous.
But    $\phi$ is not necessarily   smooth.
The following explains    this for the cases of $d=1, 2$.

\begin{rem}\label{ttensorr}
%Therefore,  both   $\phi$ and   $f$  in  Theorem \ref{jielunzongjie}   are
%not required to be smooth.
%We  explain  this for the cases of $d=1, 2$.
(1) Direct calculation gives that
$B_{2}(x)$ takes $x, 2-x$ and $0$ for $x\in  [0, 1], (1, 2]$ and $(-\infty, 0)\cup (2, +\infty)$, respectively.
Additionally,  $\hbox{sr}_{B_{2}}=2>1/2$ and $\nu_{2}(B_{2})=3/2>1/2$.
%Then  \eqref{yaoqiu1234} holds for choosing $\phi=B_{2}.$
Choosing $\phi=B_{2}$, then \eqref{yaoqiu1234} holds and consequently  it satisfies the requirement for
Theorem \ref{jielunzongjie}. But it is clear that  $B_{2}$ is not smooth.
%
%it follows from Remark \ref{remmarkdd} (2) that  Theorem \ref{jielunzongjie} holds for
%any $3/2>\varsigma>s>1/2.$
%Besides $\phi$, there are also many other   signals in $H^{\varsigma}(\mathbb{R})$ such  that they are  not smooth.
%Typical examples are  $f(x)=e^{\textbf{i}\beta x}\phi(x), 0\neq\beta\in \mathbb{R}$.
(2) For the case of $d=2$, through the tensor product  we define $\phi(x_{1}, x_{2})=B_{2}(x_{1})B_{2}(x_{2}).$
Then  $\hbox{sr}_{\phi}=2>1$.
Clearly, $\phi$ is nonnegative.
Moreover,  $\phi$ is refinable such that
$\widehat{\phi}(2\xi_{1}, 2\xi_{2})
=(\frac{1+e^{-\textbf{i}\xi_{1}}}{2})^{2}(\frac{1+e^{-\textbf{i}\xi_{2}}}{2})^{2}$
$\widehat{\phi}(\xi_{1}, \xi_{2}).$
By $\widehat{\phi}(\xi_{1}, \xi_{2})=e^{-\textbf{i} (\xi_{1}+\xi_{2})}(\frac{\sin\xi_{1}/2}{\xi/2})^{2}(\frac{\sin\xi_{2}/2}{\xi/2})^{2}$,
direct calculation gives us that $\nu_{2}(\phi)=3/2>1$.
Then \eqref{yaoqiu1234} holds and consequently  $\phi$ satisfies the requirement for
Theorem \ref{jielunzongjie}. But it is clear that  $\phi$ is  not smooth.
\end{rem}

%\begin{rem}
%% But they are not   required to be  smooth.
%%Then requirement  $s>d/2$ is in Theorem \ref{jielunzongjie} is weak.
%Through the direct calculation, we have that
%the Fourier transform of the characteristic function $\chi_{(0, 1]}$ is $e^{-\textbf{i} \xi/2}\frac{\sin\xi/2}{\xi/2}.$
%Define the $2$nd cardinal B-spline $\phi:=\chi_{(0, 1]}\star\chi_{(0, 1]}$ where $\star$ is the convolution.
%On the other hand, $\widehat{\phi}(\xi)=e^{-\textbf{i} \xi}(\frac{\sin\xi/2}{\xi/2})^{2}.$
%From this, $\phi$ is refinable with the mask symbol $\widehat{a}(\xi)=(\frac{1+e^{-\textbf{i}\xi}}{2})^{2}$ and its  sum rule   is $2.$
%By the definition of  Sobolev smoothness exponent, we can calculate that $\nu_{2}(\phi)=1>1/2.$
%Then choosing arbitrary $1/2<s<\varsigma<1,$
%Theorem \ref{jielunzongjie} \eqref{mainresult}  holds  for  all the functions in $H^{\varsigma}(\mathbb{R})$.
%
%%Such an example implies that the target function $f$ in Theorem \ref{jielunzongjie} is not required to be smooth.
%\end{rem}

%\begin{rem}
%
%For a nonnegative  refinable function $\phi$ on $\mathbb{R}^{d}$, if its smoothness exponent  $\nu_{2}(\phi)$
%and sum rule $\kappa+1$ are both larger than $d/2,$ then Theorem \ref{jielunzongjie} holds
%with
%\end{rem}

\subsubsection{\textbf{Construction of the   refinable function}  $\phi$ \textbf{required in Theorem} \ref{jielunzongjie}}
Since  $\hbox{sr}_{B_{m}}=m$ and $\nu_{2}(B_{m})=m-1/2$, then
\eqref{yaoqiu1234} holds when  choosing $m\geq2$ and  $\phi=B_{m}$.
The purpose of this subsection is to construct more   refinable functions that  satisfy
\eqref{yaoqiu1234} and consequently meet  the requirements for Theorem \ref{jielunzongjie}.

We start with the nonnegative property \eqref{yaoqiu1234} (ii). For the case of  $d=1$,
there are many results (e.g. \cite{Micchelli,Yangwang,Zhouxinlong}) on the mask symbol $\widehat{a}(\xi):=\sum^{n}_{k=0}a_{k}e^{-\textbf{i}k\xi}$ such that the corresponding compactly supported    refinable function
 $\phi$ is nonnegative.
 In particular, it was proved in \cite[pp.80-83]{Micchelli} that  for $n\geq2$   if the following holds:
 \begin{align}\label{feifuxing}
 a_{k}>0 \ \hbox{for any} \ k\in \{0, 1, \ldots, n\} \ \  (\ref{feifuxing} A), \sum_{j}a_{2j}=\sum_{j}a_{2j+1}=1/2 \ \  (\ref{feifuxing} B),
 \end{align}
 then $\phi(x)>0$ and $\phi(x)=0$  for    $x\in (0, n)$ and $x\notin (0, n),$ respectively.
 Note that
 (\ref{feifuxing}$B$) implies   $\hbox{sr}_{\phi}\geq1.$
 By \eqref{feifuxing} one can construct many nonnegative refinable functions. A standard construction
 formula is to set  that $\widehat{a}(\xi)=(\frac{1+e^{-\textbf{i}\xi}}{2})^{L}Q(\xi)$
 where $Q(\xi)=\sum^{n-L}_{k=0}q_{k}e^{-\textbf{i}k\xi}$ such that every  $q_{k}>0$.  
 The refinable function $\phi$ corresponding to  such a mask symbol  $\widehat{a}(\xi)$  has at least 
 $L$ sum rules (c.f. \cite[Theorem 1.2.5]{waveletBOOK1}). 
%It follows from \eqref{lithography} that
%$\{B_{m}\}^{\infty}_{m=1}$ satisfy \eqref{feifuxing} and  they    are nonnegative.
Additionally, the sum rule order  can be lifted through the convolution. Particularly, suppose that  $\phi_{1}$ and
$\phi_{2}$ are two nonnegative refinable functions on $\mathbb{R}$ such that
$\widehat{\phi_{1}}(2\xi)=\widehat{a_{1}}(\xi)\widehat{\phi_{1}}(\xi)$,
$\widehat{\phi_{2}}(2\xi)=\widehat{a_{2}}(\xi)\widehat{\phi_{2}}(\xi)$
and $\hbox{sr}_{\phi_{1}}, \hbox{sr}_{\phi_{2}}\geq1$. Then
$\phi:=\phi_{1}\star\phi_{2}$ is also nonnegative and  refinable such that $\widehat{\phi}(2\xi)=\widehat{a_{1}}(\xi)\widehat{a_{2}}(\xi)
\widehat{\phi}(\xi).$
Moreover,
\begin{align}\label{harbor}\hbox{sr}_{\phi}=\hbox{sr}_{\phi_{1}}+\hbox{sr}_{\phi_{2}}.\end{align}
Then the sum rule orders of $\phi_{1}$ and $\phi_{2}$ are both lifted.
As in Remark \ref{remmarkdd} (1), we just require that $\nu_{2}(\phi)>1/2.$
Such a requirement can be achieved by the above  convolution. Particularly,
the Sobolev smoothness exponent
\begin{align}\label{fariy}\nu_{2}(\phi)\geq\nu_{2}(\phi_{1})+\mu\end{align}
if $|\widehat{\phi_{2}}(\xi)|^{2}\leq\frac{C}{(1+\xi^{2})^{\mu}}$
for a constant $C>0$ and  any $\xi\in \mathbb{R}.$
That is, $\nu_{2}(\phi_{1})$
is lifted  if $\mu>0.$ As an example, such a  $\phi_{2}$ can be chosen as $B_{m}$ with $m\geq2.$
Summarizing \eqref{feifuxing}, \eqref{harbor} and \eqref{fariy} we conclude  that
for any $\gamma>1/2$, a nonnegative  refinable function $\phi$ on $\mathbb{R}$ can be constructed such that $\min\{\nu_{2}(\phi), \hbox{sr}_{\phi}\}>\gamma.$ Naturally, \eqref{yaoqiu1234} holds.

Motivated by Remark \ref{ttensorr} (2), from $\{B_{m}\}$   one can use the tensor product to construct
nonnegative refinable functions on $\mathbb{R}^{d}$ such that they
satisfy \eqref{yaoqiu1234}.   Particularly, define
\begin{align}
\phi(x_{1}, \ldots, x_{d})=B_{m_{1}}(x_{1})B_{m_{2}}(x_{2})\cdots B_{m_{d}}(x_{d}),
\end{align}
where every  $m_{i}\geq2.$ Then  $\nu_{2}(\phi)=\min\{m_{1}, \ldots, m_{d}\}-1/2$
(c.f. Han and Shen  \cite[Page 375]{Hanbin1}). Moreover, $\hbox{sr}_{\phi}=\min\{m_{1}, \ldots, m_{d}\}.$
Choosing appropriate $m_{1}, \ldots, m_{d}$, the refinable function  $\phi$ satisfies \eqref{yaoqiu1234}.
Besides  the tensor product-based construction, one can also choose the  refinable functions from the refinable box-splines
 (c.f. \cite{Hanbin1}) such that they satisfy  \eqref{yaoqiu1234}.

%Besides the cardinal B-splines,
%Based on these results one can construct nonnegative refinable functions.

%\begin{rem}\label{remark0} (1) There are two key ingredients for our recovery. One is
%the interference intensity-based   approximation to quasi-interference intensity (see \eqref{qiujie7gh15} and \eqref{HJJ}). The other is the uniform  admissibility of the reference wave.
% (2)
%\end{rem}

\subsection{Outline of the paper}
In subsection \ref{samplingset}, based on the wavelet multilevel approximation,  a two-step scheme is established for the
single-shot PR. The error of the scheme is analyzed in subsection
\ref{erroranalysis}. Motivated by the error estimation, the  admissibility
of reference waves is introduced in subsection \ref{yixiedingyi} such that, as the level increases  the recovery from the single-shot scheme
converges stably to the target function on the region of interest. The plane and spherical waves
are two types of reference waves commonly applied in holography. In subsection \ref{JKKKK12345},
classes of uniformly  admissible  plane and spherical reference waves are designed
such that  they enjoy some nice properties such as the admissibility  being independent
of  sampling set.
%
%
%In subsection \ref{yixiedingyi} we introduce  the uniform
%for plane   and spherical reference  waves.
%It will be clear in Remark \ref{butong} that the uniform admissibility of plane wave and that of spherical wave are  quite  different from each other.

We establish the single-shot PR
in Approaches  \ref{ppp4} and \ref{ppp55} for the plane and spherical cases, respectively.
The recovery error estimation  for the two approaches relies on the interference intensity-based approximation  to the quasi-interference
intensities that is established in Theorem \ref{miyuzhong} and \eqref{HJJ}.
The  two main approximation  results are organized in Theorems \ref{firstmaintheorem} and \ref{secondmaintheorem} for   the
above two cases, respectively.
As implied   in \eqref{mainresult}, the approximation
version converges  exponentially to $f$  as the level  $N$ increases.
Inspired by the transport of intensity equation (a classical  noninterference method for PR, c.f. \cite{TIE1,TIE2,TIE3}), Theorems \ref{firstmaintheorem} and \ref{secondmaintheorem}
are interpreted in subsection \ref{interpretation} from the perspective of  intensity difference.

%As mentioned in \cite{reasonforPR}, on the other hand,
%the reason for PR is that measuring the phase distributions   of   highly  oscillatory signals
%is very  difficult for electronic measurement
%devices. Recall that the chirp signal is  ubiquitous in nature and  is also a typical class of
%oscillatory signals (c.f. \cite{CHEN}).
%In Section \ref{simulation}, numerical simulations
%are conducted for  the chirp signals (derived from Cohen \cite{Cohen}) to check our results.

%Through the three-shot ($\theta=0, \pi/2, \pi$; c.f. \cite{Yamaguchi0} or \cite[(39)]{holographic}) or  two-shot (c.f. \cite{Yamaguchi1}) approaches,
%the recovery can be also achieved. Note that the two-shot approach requires additional conditions such as the separate intensity exposures $|f|$ and $|g|$ of $f$ and $g$.

\section{Preliminary}\label{07654}
\subsection{Denotations}
The sets of real, complex numbers and    integers are denoted by $\mathbb{R}$, $\mathbb{C}$ and $\mathbb{Z}$,
respectively. Additionally, the set of  positive integers  is denoted by $\mathbb{N}$, and let $\mathbb{N}_{0}:=\mathbb{N}\cup\{0\}$.
For any  $a\in \mathbb{C}$, its  real and imaginary parts
are denoted by $\Re(a)$ and $\Im(a)$, respectively, such that  $a=\Re(a)+\textbf{i}\Im(a)$ where $\textbf{i}$ is the imaginary unit.
The natural basis for $\mathbb{R}^{d}$ is denoted by $\{\emph{\textbf{e}}_{k}\}^{d}_{k=1}$, where the $n$th element of  $\emph{\textbf{e}}_{k}$
takes   $1$ and $0$ for $n=k$ and $n\neq k,$ respectively.
The  Lebesgue measure of  a measurable set $\Xi\subseteq\mathbb{R}^{d}$  is denoted by $\hbox{Vol}(\Xi)$, and the cardinality
of a countable set $\mathcal{A}\subseteq\mathbb{R}^{d}$ is denoted by $\#\mathcal{A}$.
Throughout the paper, the region of interest (ROI) is denoted by $\Omega$ and the restriction
of a function $f$ on $\Omega$ is denoted by $f|_{\Omega}$.
%A letter in boldface (e.g. $\emph{\textbf{x}}$)
%is used to denote a vector in   $\mathbb{C}^N.$

 For $\mu=(\mu_{1}, \ldots, \mu_{d})\in \mathbb{N}^{d}_{0}$, let $p_{\mu}$ denote a $d$-variate  monomial given by
$p_{\mu}(x_{1}, \ldots, x_{d})$
$=x^{\mu_{1}}_{1}\cdots x^{\mu_{d}}_{d}.$
The total degree  of $p_{\mu}$ is $\hbox{deg}(p_{\mu})=\|\mu\|_{1}$. For $\kappa\in \mathbb{N}_{0}$, denote by $\pi_{\kappa}$ the linear  span
of all $d$-variate  monomials $p_{\mu}$ such that $\hbox{deg}(p_{\mu})\leq \kappa$.
 For the above  $\mu$, its associated
differential operator $D^{\mu}:=D^{\mu_{1}}_{1}\cdots D^{\mu_{d}}_{d}$ where $D^{\mu_{k}}_{k}$
is the partial derivative with respect to the $k$th coordinate.

\subsection{Preliminary of  Sobolev space}\label{positivemd}
For any $\varsigma\in\mathbb{R}$,  the Sobolev space $H^{\varsigma}(\mathbb{R}^{d})$
is defined as
\begin{align}\begin{array}{lllll}\label{defi_s}\displaystyle  H^{\varsigma}(\mathbb{R}^{d}):=\big\{f: \mathbb{R}^d\longrightarrow \mathbb{C}
\ \hbox{such that} \ \int_{\mathbb{R}^{d}}|\widehat{f}(\xi)|^{2} (1+\|\xi\|_{2}^{2})^{\varsigma}d\xi<\infty\big\},\end{array}\end{align}
 where
$\widehat{f}(\xi)=\int_{\mathbb{R}^{d}}f(\emph{\textbf{x}})e^{-\textbf{i}\emph{\textbf{x}}\cdot \xi}d\emph{\textbf{x}}$ is
the Fourier transform of $f$.
For any $f\in H^{\varsigma}(\mathbb{R}^{d})$  the deduced norm $\|f\|_{H^{\varsigma}(\mathbb{R}^{d})}:=\frac{1}{(2\pi)^{d/2}}(\int_{\mathbb{R}^{d}}|\widehat{f}(\xi)|^{2} (1+\|\xi\|_{2}^{2})^{\varsigma}d\xi)^{1/2}$.
Clearly, $H^{0}(\mathbb{R}^{d})=L^{2}(\mathbb{R}^{d})$ and
$H^{\varsigma_{1}}(\mathbb{R}^{d})\subseteq H^{\varsigma_{2}}(\mathbb{R}^{d})$ for  $\varsigma_{1}>\varsigma_{2}$.
The Sobolev smoothness exponent  of $f$ is defined as $\nu_{2}(f)=\sup\{\varsigma\in \mathbb{R}: f\in H^{\varsigma}(\mathbb{R}^{d})\}$.
If $\varsigma>d/2$ then
the functions in  $H^{\varsigma}(\mathbb{R}^{d})$ are continuous (c.f. \cite[Chapter 9.1]{Matla}
or \cite[Section 1.1]{LiSCI}).

For any $f\in H^{\varsigma}(\mathbb{R}^d)$ and $g\in H^{-\varsigma}(\mathbb{R}^d)$, define
$ \displaystyle\langle f, g
\rangle:=\frac{1}{(2\pi)^{d}}\int_{\mathbb{R}^{d}}\widehat{f}(\xi)\overline{\widehat{g}}(\xi)d\xi,$
where $\overline{\widehat{g}}(\xi)$ is the complex conjugate  of $\widehat{g}(\xi)$.
By the Cauchy-Schwarz inequality we have
\begin{align}\label{neijiyouyiyi}\begin{array}{lllll}
\displaystyle|\langle f, g
\rangle|&\displaystyle=\frac{1}{(2\pi)^{d}}|\int_{\mathbb{R}^{d}}\widehat{f}(\xi)(1+\|\xi\|_{2}^{2})^{\varsigma/2}\overline{\widehat{g}}(\xi)
(1+\|\xi\|_{2}^{2})^{-\varsigma/2}d\xi|\\
&\displaystyle\leq\frac{1}{(2\pi)^{d}}\Big(\int_{\mathbb{R}^{d}}|\widehat{f}(\xi)|^{2}(1+\|\xi\|_{2}^{2})^{\varsigma}d\xi\Big)^{1/2}
\Big(\int_{\mathbb{R}^d}|\overline{\widehat{g}}(\xi)|^{2}(1+\|\xi\|_{2}^{2})^{-\varsigma}d\xi\Big)^{1/2}\\
&=\|f\|_{H^{\varsigma}(\mathbb{R}^d)}\|g\|_{H^{-\varsigma}(\mathbb{R}^d)}\\
&<\infty.
\end{array}\end{align}
In this sense, the pair $(H^{\varsigma}(\mathbb{R}^d), H^{-\varsigma}(\mathbb{R}^d))$
is   referred to as a pair of dual Sobolev  spaces, and $\langle f, g
\rangle$ is considered as the inner product of $f$ and $ g$.

\subsection{Preliminary of wavelet in Sobolev space}\label{xiaobosobolev}

In what follows  we make some necessary preparations on wavelet analysis in Sobolev space. For more details on this topic one  can refer to
\cite{waveletBOOK2,waveletBOOK3,waveletBOOK1,Hanbin1,Matla}.

\subsubsection{\textbf{Refinable function}}
We say that  a function  $\phi\in H^{s}(\mathbb{R}^{d})$ is
$2I_{d}$-refinable   if
\begin{align}\label{yydfr}\widehat{\phi}(2\xi)=\widehat{a}(\xi)\widehat{\phi}(\xi),\end{align}
where $\widehat{a}(\xi):=\sum_{\emph{\textbf{k}}\in\mathbb{Z}^{d}}a[\emph{\textbf{k}}]e^{\textbf{i}\emph{\textbf{k}}\cdot\xi}$
is a $2\pi \mathbb{Z}^d$-periodic function and it
is referred to as the mask symbol of $\phi$,  $I_{d}$ is the identity matrix of order $d.$
 It is required that $\widehat{\phi}(\textbf{0})=1$
and consequently $\widehat{a}(\textbf{0})=1.$
There are two typical classes   of refinable functions in Sobolev space.
One is the cardinal  B-splines defined in \eqref{lithography}. The other is the Dirac distribution and its tensor products.
They will be helpful for our discussion.

 Recall that the Fourier transform of the  Dirac distribution
$\delta$ on $\mathbb{R}$ is  $\widehat{\delta}\equiv1.$ From this, $\delta$ is $2$-refinable with the mask symbol $\widehat{\tilde{a}}\equiv1$, and
it is not difficult to check that $\delta\in H^{-s}(\mathbb{R})$ for any $s>1/2.$ By the tensor product  (c.f. \cite{fagep}),
one can define \begin{align}\label{gaoweidelta}\Delta(x_{1}, x_{2}, \ldots, x_{d})=\delta(x_{1})\delta(x_{2})\cdots\delta(x_{d}).\end{align}
Then $\widehat{\Delta}\equiv1$. From this,   $\Delta$ is $2I_{d}$-refinable such that its mask symbol $\widehat{\tilde{a}}\equiv1$ and
$\Delta\in H^{-s}(\mathbb{R}^d)$ for any  $s>d/2.$
By \eqref{neijiyouyiyi}, for any  $f\in H^{s}(\mathbb{R}^d)$
with $s>d/2$ the operation $\langle f, \Delta\rangle$
is well defined. What is more,  $\Delta$ has  the sampling property (\cite{fagep}) such that
\begin{align}\label{samplingproperty}
\langle f, \Delta\rangle=f(\textbf{0}).
\end{align}

%\textcolor[rgb]{1.00,0.00,0.00}{The cardinal B-splines $\{B_{m}\}_{m\in \mathbb{N}}$ on $\mathbb{R}$ are an interesting class of refinable functions  (c.f. \cite{waveletBOOK2,waveletBOOK1,Wenchangsun}), where
%$B_{m}$ is the $m$th cardinal B-spline
%defined by
%\begin{align}\label{spline} B_{m}:=\overbrace{\chi_{(0,1]}\star\ldots\star\chi_{(0,1]}}^{m\ \tiny{\hbox{copies}}},\end{align}
%with $\chi_{(0,1]}$    the characteristic function of the interval  $(0,1]$
%and $\star$  the convolution.
%Through the  simple calculation we have $\hbox{supp}(B_{m})=(0, m]$  and
%  $\widehat{B_{m}}(\xi)=e^{-\textbf{i} \frac{m\xi}{2}}[\frac{\sin\xi/2}{\xi/2}]^{m}. $ From this,   $B_{m}$ is $2$-refinable such that
%\begin{align}\label{yangtiaofuliyebianhuan} \widehat{B_{m}}(2\xi)=(\frac{1+e^{-\textbf{i}\xi}}{2})^{m}\widehat{B_{m}}(\xi).\end{align}}

\subsubsection{\textbf{Sum rule, approximation accuracy,  vanishing moment and orthogonality}}\label{kmnbvc}
Sum rule,  approximation accuracy, vanishing moment and orthogonality  are   important concepts in wavelet analysis.
Throughout the paper, denote by   $\{\gamma_{0}, \gamma_{1}, \ldots \gamma_{2^{d}-1}\}$ (with $\gamma_{0}=\textbf{0}$)   the  complete set  of representatives of distinct cosets  of the quotient group $ [2^{-1}\mathbb{Z}^{d}]/\mathbb{Z}^{d}$.
Here   we say that the mask symbol $\widehat{a}$  in \eqref{yydfr} has $\kappa+1$ \textbf{sum rules}  if \begin{align}\label{mmbcxx}\widehat{a}(\xi+2\pi\gamma_{j})=\hbox{O}(\|\xi\|^{\kappa+1}_{2})
%\widehat{a}(\xi)=1+\hbox{O}(\|\xi\|^{\kappa+1}_{2})
\ \hbox{as} \ \xi\rightarrow \textbf{0},\end{align}
for  any  $\textbf{0}\neq\gamma_{j}\in [2^{-1}\mathbb{Z}^{d}]/\mathbb{Z}^{d}$.
%By \eqref{yangtiaofuliyebianhuan},
%$B_{m}$ has $m$ sum rules (c.f. \cite{waveletBOOK1}).
Additionally, we say that  $\phi$ has the  \textbf{accuracy} order $\kappa+1$ (c.f. \cite[Page 405]{waveletBOOK1}) if for  any $d$-variate  monomial $p_{\mu}$ such that $\hbox{deg}(p_{\mu})\leq\kappa$,
there exists a sequence $\{c_{\emph{\textbf{k}}}\}_{\emph{\textbf{k}}\in \mathbb{Z}^d}\subseteq \mathbb{C}$ such that
%$p_{\mu}$ can be expressed by the shifts of $\phi$ via
\begin{align}\label{jingdu}p_{\mu}=\sum_{\emph{\textbf{k}}\in \mathbb{Z}^d}c_{\emph{\textbf{k}}}\phi(\cdot-\emph{\textbf{k}}).\end{align}
The mask symbol $\widehat{a}$ of    $\phi$ has $\kappa+1$ sum rules if and only if
$\phi$ has the  accuracy order $\kappa+1$ (c.f. \cite{waveletBOOK1,Jiarongqing}).
It is    required in wavelet analysis  that the refinable function $\phi\in L^2(\mathbb{R}^d)$
has at least one sum rule. What is more, the accuracy order $1$ admits that
\begin{align}\label{yydnew} 1\equiv\sum_{\emph{\textbf{k}}\in \mathbb{Z}^{d}}\phi(\cdot-\emph{\textbf{k}}).\end{align}
Such an  interesting  expression is also referred to as the unit decomposition.

We say that a function  $f:
\mathbb{R}^{d}\longrightarrow \mathbb{C}$ has $\kappa+1
(\in\mathbb{N})$ \textbf{vanishing moments}  if
$\frac{\partial^{\alpha}}{\partial \emph{\textbf{x}}^{\alpha}}\widehat{f}(\textbf{0})=0$
for any  $\alpha\in\mathbb{N}^{d}_{0}$ such that $||\alpha||_{1}\leq
\kappa$, where $\emph{\textbf{x}}=(x_{1}, \ldots, x_{d})$, $\alpha=(\alpha_{1}, \ldots, \alpha_{d})$
and $\frac{\partial^{\alpha}}{\partial \emph{\textbf{x}}^{\alpha}}f=\frac{\partial_{1}^{\alpha_{1}}\partial_{2}^{\alpha_{2}}\cdots
\partial_{d}^{\alpha_{d}}}{\partial x_{1}^{\alpha_{1}}\partial x_{2}^{\alpha_{2}}\cdots
\partial x_{d}^{\alpha_{d}}}f.$
We say that the refinable function $\phi\in H^{s}(\mathbb{R}^{d})$ with $s\geq0$ is  orthogonal
if for any $\emph{\textbf{k}}\in \mathbb{Z}^{d}$ the inner product $\int_{\mathbb{R}^d}
\phi(\emph{\textbf{x}})\overline{\phi(\emph{\textbf{x}}-\emph{\textbf{k}})}d\emph{\textbf{x}}$
takes $1$ and $0$ for $\emph{\textbf{k}}=\textbf{0}$ and $\emph{\textbf{k}}\neq\textbf{0},$ respectively.

Now  it is  ready to  recall  the dual  wavelet frames in dual Sobolev spaces (c.f. \cite{Hanbin1}).
%For more details on this topic one  can refer to \cite{Hanbin1}.

\subsubsection{\textbf{Dual wavelet frames in dual Sobolev spaces and the approximation property}}

Suppose that $\{\psi^{\ell}\}^{L}_{\ell=1}$ is a set of   functions
derived from the $2I_{d}$-refinable function $\phi\in H^{s}(\mathbb{R}^{d})$ such that
\begin{align}\label{yy3}\begin{array}{lllll} \widehat{\psi^{\ell}}(2\cdot)=\widehat{b^{\ell}}(\cdot)\widehat{\phi}(\cdot),\end{array}\end{align}
where $\widehat{b^{\ell}}(\cdot)$ is a   $2\pi \mathbb{Z}^{d}$-periodic trigonometric polynomial and is referred to as
the mask symbol of $\psi^{\ell}$. Now
the so-called  wavelet system $X^{s}(\phi;$ $\psi^{1},\ldots,\psi^{L})$ in
$H^{s}(\mathbb{R}^{d})$ is defined as
\begin{align}\begin{array}{lllll}\label{lisan}
X^{s}(\phi;\psi^{1},\ldots,\psi^{L})&:=\{\phi_{0,\emph{\textbf{k}}}:\emph{\textbf{k}}\in
\mathbb{Z}^{d}\}
\\&\quad \cup\{\psi^{\ell,s}_{j,\emph{\textbf{k}}}:\emph{\textbf{k}}\in\mathbb{Z}^{d}, j\in
\mathbb{N}_{0},\ell=1,\ldots,L\},
\end{array}\end{align}
where $\phi_{0,\emph{\textbf{k}}}=\phi(\cdot-\emph{\textbf{k}})$,
$\psi^{\ell,s}_{j,\emph{\textbf{k}}}=2^{j(d/2-s)}\psi^{\ell}(2^{j}\cdot-\emph{\textbf{k}})$ and
$\mathbb{N}_{0}=\mathbb{N}\cup\{0\}$. On the other hand, $\tilde{\phi}$
is a $2I_{d}$-refinable function in $H^{-s}(\mathbb{R}^d)$
such that
\begin{align}\label{yydfrduiou}\widehat{\tilde{\phi}}(2\cdot)=\widehat{\tilde{a}}(\cdot)\widehat{\tilde{\phi}}(\cdot),\end{align}
and
$X^{-s}(\tilde{\phi};$ $\tilde{\psi}^{1},\ldots,\tilde{\psi}^{L})$ is a wavelet system in
$H^{-s}(\mathbb{R}^{d})$ such that
\begin{align}\begin{array}{lllll}\label{lisan}
X^{-s}(\tilde{\phi};\tilde{\psi}^{1},\ldots,\tilde{\psi}^{L})&:=\{\tilde{\phi}_{0,\emph{\textbf{k}}}:\emph{\textbf{k}}\in
\mathbb{Z}^{d}\}
\\&\quad \cup\{\tilde{\psi}^{\ell,-s}_{j,\emph{\textbf{k}}}:\emph{\textbf{k}}\in\mathbb{Z}^{d}, j\in
\mathbb{N}_{0},\ell=1,\ldots,L\},
\end{array}\end{align}
where
\begin{align}\begin{array}{lllll} \label{ABCyy3}\widehat{\tilde{\psi}^{\ell}}(2\cdot)=\widehat{\tilde{b}^{\ell}}(\cdot)\widehat{\tilde{\phi}}(\cdot),
\tilde{\phi}_{0,\emph{\textbf{k}}}=\tilde{\phi}(\cdot-\emph{\textbf{k}}), \tilde{\psi}^{\ell,-s}_{j,\emph{\textbf{k}}}=2^{j(d/2+s)}\tilde{\psi}^{\ell}(2^{j}\cdot-\emph{\textbf{k}}).\end{array}\end{align}
For any $ f\in
H^{s}(\mathbb{R}^{d})$ and $ g\in H^{-s}(\mathbb{R}^{d})$, if
\begin{align}\label{hanbin4}\begin{array}{lllll}\displaystyle \langle f,g\rangle=\sum_{\emph{\textbf{k}}\in \mathbb{Z}^{d}}\langle
\phi_{0,\emph{\textbf{k}}}, g\rangle\langle f, \tilde{\phi}_{0,\emph{\textbf{k}}}\rangle+
\sum^{L}_{\ell=1}\sum_{j\in \mathbb{\mathbb{N}}_{0}}\sum_{\emph{\textbf{k}}\in
\mathbb{Z}^{d}} \langle \psi^{\ell,s}_{j,\emph{\textbf{k}}}, g \rangle\langle f,
\tilde{\psi}^{\ell,-s}_{j,\emph{\textbf{k}}} \rangle,
\end{array}\end{align}
then we say that $X^{s}(\phi;\psi^{1},\ldots,\psi^{L})$ and
$X^{-s}(\tilde{\phi};\tilde{\psi}^{1},\ldots,\tilde{\psi}^{L})$
form a pair of dual $2I_{d}$-wavelet frame  systems for  the dual Sobolev spaces
$(H^{s}(\mathbb{R}^{d}),H^{-s}(\mathbb{R}^{d}))$. Now it follows from \eqref{hanbin4}
that
\begin{align}\label{hanbin5}\begin{array}{lllll}\displaystyle f=\sum_{\emph{\textbf{k}}\in \mathbb{Z}^{d}}\langle f, \tilde{\phi}_{0,\emph{\textbf{k}}}\rangle
\phi_{0,\emph{\textbf{k}}}+
\sum^{L}_{\ell=1}\sum_{j\in \mathbb{\mathbb{N}}_{0}}\sum_{\emph{\textbf{k}}\in
\mathbb{Z}^{d}}\langle f,
\tilde{\psi}^{\ell,-s}_{j,\emph{\textbf{k}}} \rangle\psi^{\ell,s}_{j,\emph{\textbf{k}}},
\end{array}\end{align}
with the series converging in $H^{s}(\mathbb{R}^d).$
Define the truncation of the series w.r.t the level  $j$ as follows,
\begin{align}\label{rt} \mathcal{S}_{\phi}^{N}f:=\sum_{\emph{\textbf{k}}\in
\mathbb{Z}^{d}}\langle f, \tilde{\phi}_{0,\emph{\textbf{k}}}\rangle
\phi_{0,\emph{\textbf{k}}}+\sum^{L}_{\ell=1}\sum^{N-1}_{j=0}\sum_{\emph{\textbf{k}}\in
\mathbb{Z}^{d}} \langle f, \tilde{\psi}^{\ell,-s}_{j,\emph{\textbf{k}}} \rangle
\psi^{\ell,s}_{j,\emph{\textbf{k}}}. \end{align}

%
%
%
%
%Now it follows from \eqref{jingdu} that for such a refinable function $\phi\in L^2(\mathbb{R}^d)$,
%the so called unit decomposition property naturally holds, namely
%A useful property of the  refinable function  $\phi$ is the so called
%unit decomposition, namely,
The following lemma is derived from \cite{LiSCI}.
It provides  the approximation to $f$ by $\mathcal{S}_{N}f$.

\begin{lem}\label{Theorem x1}
Suppose that $X^{s}(\phi;\psi^{1}, \psi^{2},\ldots, \psi^{L})$ and
$X^{-s}(\tilde{\phi};\tilde{\psi}^{1},
 \tilde{\psi}^{2}, \ldots, \tilde{\psi}^{L})$ form  a pair of dual $2I_{d}$-wavelet frame systems
 for $(H^{s}(\mathbb{R}^{d}),H^{-s}(\mathbb{R}^{d}))$. Moreover,    $\phi\in
 H^{\varsigma}(\mathbb{R}^{d})$
%$\nu(\tilde{\phi})>-s$
and
 $\tilde{\psi}^{\ell}$ has $\kappa+1$ vanishing moments, where  $0<s<\varsigma<\kappa+1$,
$\kappa\in\mathbb{N}_{0}$ and $\ell=1,2, \ldots, L$.
 Then  there exists a positive constant $C(s, \varsigma)$  such that for any $f\in
H^{\varsigma}(\mathbb{R}^{d})$,
\begin{align}\begin{array}{lll} \label{bound1}
\displaystyle
\|f-\mathcal{S}_{\phi}^{N}f\|_{H^{s}(\mathbb{R}^{d})} \leq C(s,
\varsigma)\|f\|_{H^{\varsigma}(\mathbb{R}^{d})} 2^{-N(\varsigma-s)/2},
\end{array}
\end{align}
where $\mathcal{S}_{\phi}^{N}f$ is defined in \eqref{rt}.
% where  $\eta$ is defined
%in \eqref{rt609}.
\end{lem}

%\section{Single-shot interference based  stable phase retrieval for functions in SIS}
%
%\begin{theo}
%Suppose that $\phi$ is a continuous function on $\mathbb{R}^{d}$ such that
%$\hbox{supp}(\phi)\subseteq\Xi:=\prod^{D}_{k=1}[a_{k}, b_{k}]$. Moreover, $g: \mathbb{R}^{d}\rightarrow \mathbb{C}$
%is a complex-valued reference wave such that $\hbox{supp}(g)=\mathbb{R}^{d}$.
%
%\end{theo}

 As in subsection \ref{kmnbvc}, $\{\gamma_{0}, \gamma_{1}, \ldots \gamma_{2^{d}-1}\}$ (with $\gamma_{0}=\textbf{0}$) is   the  complete set  of representatives of distinct cosets of the quotient group  $ [2^{-1}\mathbb{Z}^{d}]/\mathbb{Z}^{d}$.
The mixed extension principle (MEP) is an efficient algorithm  (c.f.\cite{Hanbin1}) for designing  dual wavelet frame  systems
$X^{s}(\phi;\psi^{1}, \psi^{2},\ldots, \psi^{L})$ and
$X^{-s}(\tilde{\phi};\tilde{\psi}^{1},
 \tilde{\psi}^{2}, \ldots, \tilde{\psi}^{L})$.
Particularly, it  states that if the mask symbols $(\widehat{a}; \widehat{b}_{1}, \ldots, \widehat{b}_{L})$
and $(\widehat{\tilde{a}}; \widehat{\tilde{b}}_{1}, \ldots, \widehat{\tilde{b}}_{L})$
of $(\phi;\psi^{1}, \psi^{2},\ldots, \psi^{L})$ and $(\tilde{\phi};\tilde{\psi}^{1},
 \tilde{\psi}^{2}, \ldots, \tilde{\psi}^{L})$ satisfy
\begin{align}\label{fffff56}\begin{array}{llllll}
\Big(\overline{\widehat{b}_{j}}(\cdot+2\pi\gamma_{k})\Big)^{2^d-1; L}_{k=0; j=1}
(\widehat{\tilde{b}}_{1}, \ldots, \widehat{\tilde{b}}_{L})^{T}&=(1-\overline{\widehat{a}}(\cdot+2\pi\gamma_{0})\widehat{\tilde{a}}(\cdot),
-\overline{\widehat{a}}(\cdot+2\pi\gamma_{1})\widehat{\tilde{a}}(\cdot), \ldots,\\ &\quad\quad-\overline{\widehat{a}}(\cdot+2\pi\gamma_{2^d-1})\widehat{\tilde{a}}(\cdot))^{T}
\end{array}
\end{align}
 %
%\begin{align}\label{fffff56}\begin{array}{llllll}
%\left[
%\begin{array}{cccccccc}
%\overline{\widehat{b}_{1}}(\cdot+2\pi\gamma_{0})&\overline{\widehat{b}_{2}}(\cdot+2\pi\gamma_{0})&\cdots&\overline{\widehat{b}_{L}}(\cdot+2\pi\gamma_{0})\\
%\overline{\widehat{b}_{1}}(\cdot+2\pi\gamma_{1})&\overline{\widehat{b}_{2}}(\cdot+2\pi\gamma_{1})&\cdots&\overline{\widehat{b}_{L}}(\cdot+2\pi\gamma_{1})\\
%    \vdots&\vdots&\ddots&\vdots\\
%\overline{\widehat{b}_{1}}(\cdot+2\pi\gamma_{2^d-1})&\overline{\widehat{b}_{2}}(\cdot+2\pi\gamma_{2^d-1})&\cdots&\overline{\widehat{b}_{L}}(\cdot+2\pi\gamma_{2^d-1})
%\end{array}
%\right]\left[
%\begin{array}{cccccccc}
%\widehat{\tilde{b}}_{1}(\cdot+2\pi\gamma_{0})\\
%\widehat{\tilde{b}}_{2}(\cdot+2\pi\gamma_{0})\\
%\vdots\\
%\widehat{\tilde{b}}_{L}(\cdot+2\pi\gamma_{0})
%\end{array}
%\right]\\
%=\left[
%\begin{array}{cccccccc}
%1-\overline{\widehat{a}}(\cdot+2\pi\gamma_{0})\widehat{\tilde{a}}(\cdot)\\
%-\overline{\widehat{a}}(\cdot+2\pi\gamma_{1})\widehat{\tilde{a}}(\cdot)\\
%\vdots\\
%-\overline{\widehat{a}}(\cdot+2\pi\gamma_{2^d-1})\widehat{\tilde{a}}(\cdot)
%\end{array}
%\right],
%\end{array}
%\end{align}
then $X^{s}(\phi;\psi^{1}, \psi^{2},\ldots, \psi^{L})$ and
$X^{-s}(\tilde{\phi};\tilde{\psi}^{1},
 \tilde{\psi}^{2}, \ldots, \tilde{\psi}^{L})$ are a pair of dual $2I_{d}$-wavelet frames in $(H^s(\mathbb{R}^d), H^{-s}(\mathbb{R}^d))$.
% $ \textcolor[rgb]{1.00,0.00,0.00}{[2^{-1}\mathbb{Z}^{d}]/\mathbb{Z}^{d}}$
The mask symbols $\widehat{b}_{1}, \ldots, \widehat{b}_{L};
\widehat{\tilde{b}}_{1},$
$ \ldots, \widehat{\tilde{b}}_{L}$ are not difficult to construct.
For example,
suppose that  $L=2^d-1$ and choose $\{\widehat{b}_{2^{d}}\}\cup\{\widehat{b}_{l}: l=1, \ldots, 2^{d}-1\}$
as the mask symbols of  orthogonal refinable function and  wavelets
$\{\mathring{\phi}; \mathring{\psi}^{1}, \ldots, \mathring{\psi}^{2^{d}-1}\}$
in $L^{2}(\mathbb{R}^{d})$ such that $\widehat{\mathring{\phi}}(2\xi)=\widehat{b}_{2^{d}}(\xi)
\widehat{\mathring{\phi}}(\xi)$ and $\widehat{\mathring{\psi}^{l}}(2\xi)=
\widehat{b}_{l}(\xi)
\widehat{\mathring{\phi}}(\xi), l=1, \ldots, 2^{d}-1$.
Then
\begin{align}\label{peigong}\Big(\overline{\widehat{b}_{j}}(\cdot+2\pi\gamma_{k})\Big)^{2^d-1; 2^d}_{k=0; j=1}
\end{align}
is a paraunitary matrix.  Now  $\{\widehat{\tilde{b}}_{l}: l=1, \ldots, 2^{d}-1\}$
can be calculated by \eqref{fffff56}.
If \eqref{fffff56} holds then it follows from \eqref{yy3} and $\widehat{\tilde{\psi}^{\ell}}(2\cdot)=\widehat{\tilde{b}^{\ell}}(\cdot)\widehat{\tilde{\phi}}(\cdot)$ in \eqref{ABCyy3} that   $\mathcal{S}^{N}_{\phi}f$ in \eqref{rt} can be expressed as the so called $N$-level  approximation
\begin{align}\label{OPPP2}\begin{array}{llllll}
\mathcal{S}^{N}_{\phi}f&\displaystyle=\sum_{\emph{\textbf{k}}\in
\mathbb{Z}^{d}}\langle f, \widetilde{\phi}_{N,\emph{\textbf{k}}}\rangle\phi_{N,\emph{\textbf{k}}}\\
&\displaystyle=\sum_{\emph{\textbf{k}}\in
\mathbb{Z}^{d}}\langle f, 2^{N(d/2+s)}\widetilde{\phi}(2^{N}\cdot-\emph{\textbf{k}})\rangle2^{N(d/2-s)}\phi(2^N\cdot-\emph{\textbf{k}})\\
&\displaystyle=\sum_{\emph{\textbf{k}}\in
\mathbb{Z}^{d}}2^{Nd}\langle f, \widetilde{\phi}(2^{N}\cdot-\emph{\textbf{k}})\rangle\phi(2^N\cdot-\emph{\textbf{k}}). \quad  (\ref{OPPP2}A)
\end{array}
\end{align}

\subsubsection{\textbf{Sampling-based approximation in Sobolev space}}
The following sampling-based  approximation will be helpful in next section. It can be derived from
\cite{LiSCI}. For the self-contained purpose, we provide its proof here.

\begin{prop}\label{caiyangapproximation}
Suppose that $\phi\in H^{s}(\mathbb{R}^d)$ is $2I_{d}$-refinable such that
its mask symbol $\widehat{a}$ has $\kappa+1$ sum rules, where $s>d/2$ and $\kappa+1>s$. Moreover, $\phi\in H^{\varsigma}(\mathbb{R}^d)$ such that  $\kappa+1>\varsigma>s.$
Then there exists a constant $C(s,\varsigma)$ such that for any $f\in H^{\varsigma}(\mathbb{R}^d)$ there holds
\begin{align}\label{cxzbc}
\|f-\sum_{\emph{\textbf{k}}\in \mathbb{Z}^d}f(2^{-N}\emph{\textbf{k}})\phi(2^N\cdot-\emph{\textbf{k}})\|_{H^{s}(\mathbb{R}^d)}\leq C(s,
\varsigma)||f||_{H^{\varsigma}(\mathbb{R}^{d})} 2^{-N(\varsigma-s)/2}.
\end{align}
\end{prop}
\begin{proof}
We choose the  $2I_{d}$-refinable function  $\tilde{\phi}=\Delta$  defined in \eqref{gaoweidelta}.
Then $\widehat{\Delta}\equiv1$ and consequently $\Delta\in H^{-\eta}(\mathbb{R}^d)$ for any  $\eta>d/2.$ By \eqref{fffff56} and  \eqref{peigong}
one can construct the dual $2I_{d}$-wavelet frames $(\phi; \psi^{1}, \ldots, \psi^{2^d-1})$
and $(\tilde{\phi}; \tilde{\psi}^{1}, \ldots, \tilde{\psi}^{2^d-1})$ in $(H^{s}(\mathbb{R}^d), H^{-s}(\mathbb{R}^d))$.
Next we  prove that $\tilde{\psi}^{1}, \ldots, \tilde{\psi}^{2^d-1}$
all have $\kappa+1$ vanishing moments. Since the mask symbol  $\widehat{a}$ has $\kappa+1$ sum rules,
then it follows from \eqref{mmbcxx} that
\begin{align}\label{fanzeng}\widehat{a}(\xi+2\pi\gamma_{j})=\hbox{O}(||\xi||^{\kappa+1}_{2})
%\widehat{a}(\xi)=1+\hbox{O}(||\xi||^{\kappa+1}_{2})
 \ \hbox{as} \ \xi\rightarrow \textbf{0},\end{align}
for  any  representative $\textbf{0}\neq\gamma_{j}$ of the  cosets of   $[2^{-1}\mathbb{Z}^{d}]/\mathbb{Z}^{d}$. Recall again that the matrix in \eqref{peigong}
is a paraunitary matrix, and the mask symbol of $\tilde{\phi}$ is  $\widehat{\tilde{a}}\equiv1.$ Combining this, \eqref{fanzeng} and  \eqref{fffff56},
we  have  that $\frac{\partial^{\alpha}}{\partial \emph{\textbf{x}}^{\alpha}}\widehat{\tilde{b}}_{l}(\textbf{0})=0, l=1, \ldots, 2^d-1.$
Now by \eqref{ABCyy3}, $\tilde{\psi}^{1}, \ldots, \tilde{\psi}^{2^d-1}$
 have $\kappa+1$ vanishing moments. The rest of the proof can be completed by \eqref{bound1} and the sampling property
 $\langle f, \Delta\rangle=f(\textbf{0})$.
\end{proof}

Based on Proposition \ref{caiyangapproximation}, in what follows we establish the  Lipschitz continuity of $f\in H^{s}(\mathbb{R}^{d})$ and
the bound for $\sup_{\emph{\textbf{x}}\in \mathbb{R}^{d}}|f(\emph{\textbf{x}})|$, where $s>d/2.$

\begin{prop}\label{1stlemma} There exists a constant $\widehat{C}(s,
\varsigma)\geq1$ such that for any $f\in H^{s}(\mathbb{R}^{d})$  satisfying $\nu_{2}(f)>\varsigma>s>d/2$, there holds
\begin{align}\label{ghgh} \sup_{\emph{\textbf{x}}\in \mathbb{R}^{d}}|f(\emph{\textbf{x}})|\leq \widehat{C}(s,
\varsigma)\|f\|_{H^{\varsigma}(\mathbb{R}^{d})}.\end{align}
Moreover, for any $\eta\in \mathbb{R}^{d}$ it holds that
\begin{align}\label{ghgh1}\sup_{\emph{\textbf{x}}\in \mathbb{R}^d}|f(\emph{\textbf{x}})-f(\emph{\textbf{x}}+\eta)|\leq \widehat{C}(s,\varsigma)2^{1-\zeta}\|f\|_{H^{\varsigma}(\mathbb{R}^{d})}\|\eta\|_{2}^{\zeta},\end{align}
where $\zeta=\min\{1, \varsigma-s\}$.
% and  $|f(2^{-N}K_{1})-f(2^{-N}K_{2})|\leq(1+C(s,\varsigma))2^{1-\zeta}||f||_{H^{\varsigma}(\mathbb{R}^{2})}2^{-N\zeta}¡£$
\end{prop}
\begin{proof}
Suppose that $\phi\in H^{s}(\mathbb{R}^{d})$ is an orthogonal  refinable  function
such that for any $\emph{\textbf{k}}\in \mathbb{Z}^{d}$, the inner product  $\int_{\mathbb{R}^d}
\phi(\emph{\textbf{x}})\overline{\phi(\emph{\textbf{x}}-\emph{\textbf{k}})}d\emph{\textbf{x}}$
takes   $1$ and $0$ for $\emph{\textbf{k}}=\textbf{0}$ and $\emph{\textbf{k}}\neq\textbf{0},$ respectively.
Moreover, the sum rule order $\hbox{sr}_{\phi}\geq1$. Then by Proposition \ref{caiyangapproximation}
 there exists $C(s,
\varsigma)>0$ (being independent of $f$) such that
\begin{align}\begin{array}{lll} \label{bound67890}
\displaystyle
\big\|f-\sum_{\emph{\textbf{k}}\in \mathbb{Z}^d}f(2^{-N}\emph{\textbf{k}})\phi(2^{N}\cdot-\emph{\textbf{k}})\big\|_{L^{2}(\mathbb{R}^d)}&\leq
\displaystyle\|f-\sum_{\emph{\textbf{k}}\in \mathbb{Z}^d}f(2^{-N}\emph{\textbf{k}})\phi(2^N\cdot-\emph{\textbf{k}})\|_{H^{s}(\mathbb{R}^d)}\\
 &\leq C(s,
\varsigma)\|f\|_{H^{\varsigma}(\mathbb{R}^{d})} 2^{-N(\varsigma-s)/2}.
\end{array}
\end{align}
Then for any  $\emph{\textbf{k}}\in \mathbb{Z}^{d}$, we have
\begin{align}\label{09876}
 \begin{array}{lll} |f(\emph{\textbf{k}})|&\leq\displaystyle\big(\sum_{\emph{\textbf{j}}\in \mathbb{Z}^{d}}|f(\emph{\textbf{j}})|^{2}\big)^{1/2}\\
 &=\displaystyle\|\sum_{\emph{\textbf{j}}\in \mathbb{Z}^{d}}f(\emph{\textbf{j}})\phi(\cdot-\emph{\textbf{j}})\|_{L^2(\mathbb{R}^d)} \quad   (\ref{09876} A)\\
 &\leq \displaystyle
 \|f-\sum_{\emph{\textbf{j}}\in \mathbb{Z}^{d}}f(\emph{\textbf{j}})\phi(\cdot-\emph{\textbf{j}})\|_{L^2(\mathbb{R}^d)}+\|f\|_{L^2(\mathbb{R}^d)}\\
 &\leq C(s,
\varsigma)\|f\|_{H^{\varsigma}(\mathbb{R}^{d})}+\|f\|_{H^{\varsigma}(\mathbb{R}^{d})} \quad   (\ref{09876} B)\\
&=(1+C(s,
\varsigma))\|f\|_{H^{\varsigma}(\mathbb{R}^{d})}\\
&:=\widehat{C}(s, \varsigma)\|f\|_{H^{\varsigma}(\mathbb{R}^{d})},\end{array}\end{align}
where (\ref{09876}$A$) and (\ref{09876}$B$) is derived from the orthogonality of $\phi$
and \eqref{bound67890}, respectively.
Define $F_{\emph{\textbf{x}}_{0}}:=f(\cdot+\emph{\textbf{x}}_{0})$ for any fixed  $\emph{\textbf{x}}_{0}\in \mathbb{R}^d$. It is straightforward  to check that
$F_{\emph{\textbf{x}}_{0}}\in H^{\varsigma}(\mathbb{R}^{d})$ and
$\|F_{\emph{\textbf{x}}_{0}}\|_{H^{\varsigma}(\mathbb{R}^{d})}=\|f\|_{H^{\varsigma}(\mathbb{R}^{d})}.$ Applying \eqref{09876}
 with $\emph{\textbf{k}}=\textbf{0}$ to $F_{\emph{\textbf{x}}_{0}}$, we have
$|f(\emph{\textbf{x}}_{0})|\leq \widehat{C}(s,
\varsigma)\|f\|_{H^{\varsigma}(\mathbb{R}^{d})}.$
Then \eqref{ghgh} holds.

On the other hand,
%\begin{align}\label{dddf}\begin{array}{lllll}
% \displaystyle||F-F(\cdot\pm2^{-N}\textbf{e}_{1})||_{H^{s}(\mathbb{R}^{d})}\\
%= \displaystyle\Big[\frac{1}{(2\pi)^{D}}\int_{\mathbb{R}^{d}}|\widehat{F}(\xi)(1-e^{\pm\textbf{i}2^{-N}\textbf{e}_{1}\cdot\xi})|^{2}
%(1+\|\xi\|^{2}_{2})^{s}d\xi\Big]^{1/2}\\
% \displaystyle\leq\Big[\frac{1}{(2\pi)^{D}}\int_{\mathbb{R}^{d}}4|\sin\big(2^{-N-1}\textbf{e}_{1}\cdot\xi\big)|^{2\zeta}
% \displaystyle|\widehat{F}(\xi)|^{2}(1+\|\xi\|^{2}_{2})^{s}d\xi\Big]^{1/2}\\
% \displaystyle\leq\Big[\frac{2^{-2(N+1)\zeta}}{(2\pi)^{D}}\int_{\mathbb{R}^{d}}\|\xi\|_{2}^{2\zeta}
%|\widehat{F}(\xi)|^{2}(1+\|\xi\|^{2}_{2})^{s}d\xi\Big]^{1/2}\\
%\leq\displaystyle\frac{2^{-(N+1)\zeta}}{(2\pi)^{D/2}}\|F\|_{H^{\varsigma}(\mathbb{R}^{d})}\\
%=\displaystyle\frac{2^{-(N+1)\zeta}}{(2\pi)^{D/2}}\|f\|_{H^{\varsigma}(\mathbb{R}^{d})},
%\end{array}\end{align}
\begin{align}\label{dddf}\begin{array}{lllll}
 \displaystyle\|f-f(\cdot+\eta)\|_{H^{s}(\mathbb{R}^{d})}\\
= \displaystyle\Big[\frac{1}{(2\pi)^{d}}\int_{\mathbb{R}^{d}}|\widehat{f}(\xi)(1-e^{-\textbf{i}\eta\cdot\xi})|^{2}
(1+\|\xi\|^{2}_{2})^{s}d\xi\Big]^{1/2}\\
 \displaystyle\leq\Big[\frac{1}{(2\pi)^{d}}\int_{\mathbb{R}^{d}}4|\sin\big(\frac{1}{2}\eta\cdot\xi\big)|^{2\zeta}
 \displaystyle|\widehat{f}(\xi)|^{2}(1+\|\xi\|^{2}_{2})^{s}d\xi\Big]^{1/2} \quad   (\ref{dddf} A) \\
 \displaystyle\leq2\Big[\frac{(\|\eta\|_{2}/2)^{2\zeta}}{(2\pi)^{d}}\int_{\mathbb{R}^{d}}\|\xi\|_{2}^{2\zeta}
|\widehat{f}(\xi)|^{2}(1+\|\xi\|^{2}_{2})^{s}d\xi\Big]^{1/2} \quad   (\ref{dddf} B) \\
\displaystyle\leq2\Big[\frac{(\|\eta\|_{2}/2)^{2\zeta}}{(2\pi)^{d}}\int_{\mathbb{R}^{d}}(1+\|\xi\|^2_{2})^{\zeta}
|\widehat{f}(\xi)|^{2}(1+\|\xi\|^{2}_{2})^{s}d\xi\Big]^{1/2}\\
\leq\displaystyle2^{1-\zeta}\|\eta\|_{2}^{\zeta}\|f\|_{H^{\varsigma}(\mathbb{R}^{d})},
\quad   (\ref{dddf} C)\end{array}\end{align}
where (\ref{dddf}$A$) is derived from $|1-e^{-\textbf{i}\eta\cdot\xi}|=2|\sin(\eta/2\cdot\xi)|$
and $\zeta\leq1$, (\ref{dddf}$B$) is  from $|\sin(\eta/2\cdot\xi)|\leq |\eta/2\cdot\xi|\leq \frac{1}2\|\eta\|_{2}\|\xi\|_{2}$,
and  (\ref{dddf}$C$) is from $\zeta\leq \varsigma-s.$ Then
\begin{align}\label{budengshi}\begin{array}{lllll} \sup_{\emph{\textbf{x}}\in \mathbb{R}^d}|f(\emph{\textbf{x}}+\eta)-f(\emph{\textbf{x}})|
&\displaystyle\leq\widehat{C}(s,\varsigma)||f-f(\cdot+\eta)||_{H^{s}(\mathbb{R}^{d})}\quad (\ref{budengshi} A)\\
&\displaystyle\leq\widehat{C}(s,\varsigma)2^{1-\zeta}\|f\|_{H^{\varsigma}(\mathbb{R}^{d})}\|\eta\|_{2}^{\zeta},
\quad (\ref{budengshi} B)\end{array}\end{align}
(\ref{budengshi}$A$) is from  \eqref{ghgh}, and (\ref{budengshi}$B$) is from \eqref{dddf}.
This completes the proof.
%Applying \eqref{09876} to $F-F(\cdot+2^{-N}\textbf{e}_{1})$
%and $F-F(\cdot-2^{-N}\textbf{e}_{1})$, it follows from  \eqref{dddf} that
%$|f(2^{-N}K)-f(2^{-N}(K\pm\textbf{e}_{1}))|\leq (1+C(s,
%\varsigma))\frac{2^{-(N+1)\zeta}}{(2\pi)^{D/2}}\|f\|_{H^{\varsigma}(\mathbb{R}^{d})}$.
\end{proof}

\section{A scheme for single-shot phase retrieval   and admissible  reference  wave}
%By the single-shot interference intensity measurements, this section is to establish the  PR  for any function  $f\in H^{\varsigma}(\mathbb{R}^{d})$ ($\varsigma>d/2$) on a bounded  ROI $\Omega$. Before proceeding further,
%some denotations are necessary.

\subsection{A two-step scheme for single-shot PR}\label{twosteps}
Based on the wavelet $N$-level approximation in  Proposition \ref{caiyangapproximation}, this subsection is to sketch our  scheme for the single-shot PR of  functions in  $ H^{s}(\mathbb{R}^d)$, $s>d/2$.
Before proceeding further, we need some necessary  denotations.

 From now on  we suppose that the refinable function $\phi\in H^{s}(\mathbb{R}^d)$
 is compactly supported,  where $s>d/2$. Particularly, without loss of generality  suppose that
 \begin{align}\label{zhichengjiashe}\hbox{supp}(\phi)\subseteq [0, M_{1}]\times \cdots\times [0, M_{d}].\end{align}
 As previously,   the target   $f$ is supposed to sit  in $ H^{s}(\mathbb{R}^d)$ and    the  region of  interest (ROI) on
  $\mathbb{R}^d$ is denoted by  $\Omega$. Correspondingly,   denote %% ÒÔÏÂÀ¶É«²¿·ÖÁ½¸ö¼ÇºÅÔÝʱ²»ÐèÒª
 \begin{align}\label{zuidaqujianbianyuan}
 M:=\max\{M_{1}, \ldots, M_{d}\}
 \end{align}
 and
\begin{align}\label{yixiejihao}\|\Omega\|_{2,\sup}:=\sup\{\|\emph{\textbf{x}}\|_{2}: \emph{\textbf{x}}\in \Omega\}, \mathbb{Z}_{\Omega}^{d}:=\Omega\cap \mathbb{Z}^{d}.\end{align}
Throughout this paper, it is assumed that $\Omega$ is bounded and $\mathbb{Z}_{\Omega}^{d}$ is   not empty.
Under this   assumption, clearly $\mathbb{Z}_{\Omega}^{d}$ is finite.
For any  $l\in \{1, \ldots, d\}$, denote
\begin{align}\label{zuixiao}L_{l,\min}:=\min\{k_{l}: (k_{1}, \ldots, k_{d})\in\mathbb{Z}_{\Omega}^{d}\}, L_{l,\max}:=\max\{k_{l}: (k_{1}, \ldots, k_{d})\in\mathbb{Z}_{\Omega}^{d}\}.\end{align}

Based on Proposition \ref{caiyangapproximation},  it is  ready to address  how to establish the scheme for the  single-shot PR.
By Proposition \ref{caiyangapproximation} \eqref{cxzbc},
$f$ can be  approximated by using  the $N$-level data $\{f(2^{-N}\emph{\textbf{k}}): \emph{\textbf{k}}\in \mathbb{Z}^{d}\}$  such that
\begin{align}\label{series}
f(\emph{\textbf{x}})\approx\sum_{\emph{\textbf{k}}\in \mathbb{Z}^d}f(2^{-N}\emph{\textbf{k}})\phi(2^N\emph{\textbf{x}}-\emph{\textbf{k}})
\end{align}
for sufficiently large level  $N.$
If restricting   $\emph{\textbf{x}}\in \Omega$, then it follows from \eqref{zhichengjiashe} and \eqref{zuixiao} that   the   series
in \eqref{series} reduce to
finitely many  sums such that
\begin{align}\label{roi}
f(\emph{\textbf{x}})\approx\sum_{\emph{\textbf{k}}\in \Lambda_{\Omega,N}} f(2^{-N}\emph{\textbf{k}})\phi(2^N\emph{\textbf{x}}-\emph{\textbf{k}}),
\end{align}
where
%
%Associated with $\mathcal{M}$, $\Omega$,  a   level $N$ and any fixed $j\in \{1, \ldots, D\}$, define three   subsets of $\mathbb{R}^{d}$:
 \begin{align}\label{fft1}\begin{array}{lll}
 \Lambda_{\Omega,N}
 :=\{\emph{\textbf{k}}=(k_{1},\ldots,  k_{d})\in \mathbb{Z}^{d}: 2^{N}L_{l,\min}-M_{l}\leq k_{l}\leq2^{N}L_{l,\max}, l=1, \ldots, d\}.
\end{array}\end{align}
That is, we can  use the  $N$-level formula   \eqref{roi} to approximate $f$ at any  $\emph{\textbf{x}}\in \Omega$ provided that the finitely many $N$-level  data $\{f(2^{-N}\emph{\textbf{k}}): \emph{\textbf{k}}\in  \Lambda_{\Omega,N}\}$
are available.
Motivated by this, our  scheme for the  single-shot   PR of $f$ is sketched  as the following two steps:
\begin{tcolorbox}
\begin{sche}\label{jiusuanfa}
\textbf{step (1)}:
Based on  finitely many
single-shot  interference intensities
$\{|f(\widehat{\emph{\textbf{t}}})+g_{N}(\widehat{\emph{\textbf{t}}})|: \widehat{\emph{\textbf{t}}}\in \widehat{\Lambda}_{\Omega,N}\}$
we construct the approximation to   the $N$-level data
$\{f(2^{-N}\emph{\textbf{k}}): \emph{\textbf{k}}\in  \Lambda_{\Omega,N}\}.$ \%\% Here $\widehat{\Lambda}_{\Omega,N}$ is a finite sampling  set that will be soon    discussed in subsection \ref{samplingset}  and  $g_{N}$ is the reference wave we use for the recovery of the above  $N$-level data.

\textbf{step (2)}:  Based on the above  approximation to the  $N$-level data $\{f(2^{-N}\emph{\textbf{k}}): \emph{\textbf{k}}\in  \Lambda_{\Omega,N}\}$,
we construct the approximation to $f$ on $\Omega$. \\
 \end{sche}
 %In the following subsection, the two steps are formulated.
\end{tcolorbox}

\subsection{Concretization  of  Scheme  \ref{jiusuanfa}} \label{samplingset}
\subsubsection{\textbf{An updated scheme: concretization of Scheme \ref{jiusuanfa} from the perspective of quasi-interference intensity}  }
As previously, the ROI and  the reference wave we use at level $N$ are   denoted by $\Omega$ and $g_{N}$, respectively. Suppose that   the interference intensities at level $N$:
 \begin{align}\label{intensity123} I_{N, \emph{\textbf{k}}}=|f(2^{-N}\emph{\textbf{k}})+g_{N}(2^{-N}\emph{\textbf{k}})|^{2},  \emph{\textbf{k}}\in  \Lambda_{\Omega,N}\end{align}
 are observed, where $\Lambda_{\Omega,N}$ is as in \eqref{fft1}.
Since the phase of $f(2^{-N}\emph{\textbf{k}})+g_{N}(2^{-N}\emph{\textbf{k}})$ is missing,
only  $I_{N, \emph{\textbf{k}}}$ is not sufficient for recovering the $N$-level data
$f(2^{-N}\emph{\textbf{k}})$.
Then  additional intensities are necessary. For this, we require  two additional  intensities
\begin{align}\label{1BXCZSDF}I_{N, \emph{\textbf{k}}^{'}}=|f(2^{-N}\emph{\textbf{k}}^{'})+g_{N}(2^{-N}\emph{\textbf{k}}^{'})|^{2}, I_{N, \emph{\textbf{k}}^{''}}=|f(2^{-N}\emph{\textbf{k}}^{''})+g_{N}(2^{-N}\emph{\textbf{k}}^{''})|^{2},\end{align}
 where $\emph{\textbf{k}}^{'}\neq\emph{\textbf{k}}^{''}$ such that $2^{-N}\emph{\textbf{k}}^{'}\approx2^{-N}\emph{\textbf{k}}$ and $ 2^{-N}\emph{\textbf{k}}^{''}\approx2^{-N}\emph{\textbf{k}}$.
 %Now the sampling set
% \begin{align}
% \Lambda=\{\emph{\textbf{k}}, \emph{\textbf{k}}', \emph{\textbf{k}}'': \emph{\textbf{k}}\in\Lambda_{\Omega,N}, 2^{-N}\emph{\textbf{k}}\approx
% 2^{-N}\emph{\textbf{k}}, 2^{-N}\emph{\textbf{k}}\approx
% 2^{-N}\emph{\textbf{k}}\}
% \end{align}
 Now  the  set of  sampling points for Scheme  \ref{jiusuanfa} is chosen as
 $$\widehat{\Lambda}_{\Omega,N}=\{2^{-N}\emph{\textbf{k}}, 2^{-N}\emph{\textbf{k}}', 2^{-N}\emph{\textbf{k}}'': \emph{\textbf{k}}\in  \Lambda_{\Omega,N}\}.$$
 For simplicity, denote a set in $\mathbb{R}^{3d}$ associated with $\widehat{\Lambda}_{\Omega,N}$ as follows,
\begin{align}\label{lingwaijihe}
\Xi_{\Omega,N}:=\{(\emph{\textbf{k}}, \emph{\textbf{k}}', \emph{\textbf{k}}''): \emph{\textbf{k}}\in  \Lambda_{\Omega,N},
2^{-N}\emph{\textbf{k}}^{'}\approx2^{-N}\emph{\textbf{k}}, 2^{-N}\emph{\textbf{k}}^{''}\approx2^{-N}\emph{\textbf{k}}\}.
\end{align}
Correspondingly, the   required  single-shot interference
 intensity set    is
\begin{align}\label{zongnengliang}
I_{\Xi_{\Omega,N}}:=\{I_{N, \emph{\textbf{k}}}, I_{N, \emph{\textbf{k}}^{'}}, I_{N, \emph{\textbf{k}}^{''}}: (\emph{\textbf{k}}, \emph{\textbf{k}}', \emph{\textbf{k}}'')\in \Xi_{\Omega,N}\}.
\end{align}

Based on $I_{\Xi_{\Omega,N}}$, in what follows we address how to formulate step (1) of Scheme  \ref{jiusuanfa}.
Since $2^{-N}\emph{\textbf{k}}^{'}\approx2^{-N}\emph{\textbf{k}}$ and $ 2^{-N}\emph{\textbf{k}}^{''}\approx2^{-N}\emph{\textbf{k}}$ then
\begin{align}\label{intensity1234}\left\{\begin{array}{lllll}
I_{N, \emph{\textbf{k}}^{'}}=|f(2^{-N}\emph{\textbf{k}}^{'})+g_{N}(2^{-N}\emph{\textbf{k}}^{'})|^{2}&\approx
|f(2^{-N}\emph{\textbf{k}})+g_{N}(2^{-N}\emph{\textbf{k}}^{'})|^{2}\\
&:=A_{N,\emph{\textbf{k}}, \emph{\textbf{k}}'}, \quad (\ref{intensity1234} \hbox{A})\\
I_{N, \emph{\textbf{k}}^{''}}=|f(2^{-N}\emph{\textbf{k}}^{''})+g_{N}(2^{-N}\emph{\textbf{k}}^{''})|^{2}&\approx
|f(2^{-N}\emph{\textbf{k}})+g_{N}(2^{-N}\emph{\textbf{k}}^{''})|^{2}\\
&:=A_{N,\emph{\textbf{k}}, \emph{\textbf{k}}''}. \quad (\ref{intensity1234} \hbox{B})
\end{array}
\right.
\end{align}
From now on $A_{N,\emph{\textbf{k}}, \emph{\textbf{k}}'}$
and $A_{N,\emph{\textbf{k}}, \emph{\textbf{k}}''}$
are referred to as the \textbf{quasi-interference} intensities.
It follows from  \eqref{intensity123},  (\ref{intensity1234}A) and (\ref{intensity1234}B)
that $f(2^{-N}\emph{\textbf{k}})$ is a solution to the  equation system w.r.t $x$:
 \begin{align}\label{intensity8765}\left\{\begin{array}{lllll}
 |x+g_{N}(2^{-N}\emph{\textbf{k}})|^{2}=I_{N, \emph{\textbf{k}}},\\
 |x+g_{N}(2^{-N}\emph{\textbf{k}}^{'})|^{2}=A_{N, \emph{\textbf{k}}, \emph{\textbf{k}}^{'}},\\
 |x+g_{N}(2^{-N}\emph{\textbf{k}}^{''})|^{2}=A_{N, \emph{\textbf{k}}, \emph{\textbf{k}}^{''}}.
\end{array}
\right.
\end{align}
Note that the quasi-interference intensities
$A_{N, \emph{\textbf{k}}, \emph{\textbf{k}}^{'}}$ and $A_{N, \emph{\textbf{k}}, \emph{\textbf{k}}^{''}}$
are not necessarily  the interference intensities  of $f$. As such,  we modify  \eqref{intensity8765} as follows,
\begin{align}\label{intensity18765}\left\{\begin{array}{lllll}
 |x+g_{N}(2^{-N}\emph{\textbf{k}})|^{2}=I_{N, \emph{\textbf{k}}},\\
 |x+g_{N}(2^{-N}\emph{\textbf{k}}^{'})|^{2}=I_{N, \emph{\textbf{k}}^{'}},\\
 |x+g_{N}(2^{-N}\emph{\textbf{k}}^{''})|^{2}=I_{N, \emph{\textbf{k}}^{''}}.
\end{array}
\right.
\end{align}
%\subsubsection{\textbf{An updated scheme: concretization of Scheme}  \ref{jiusuanfa}}

For  the  perturbations in \eqref{intensity1234}: $A_{N,\emph{\textbf{k}},\emph{\textbf{k}}'}\approx I_{N, \emph{\textbf{k}}'}
$ and $A_{N,\emph{\textbf{k}},\emph{\textbf{k}}''}\approx I_{N, \emph{\textbf{k}}''}
$, equation system  \eqref{intensity18765} can be considered as the perturbation
of \eqref{intensity8765}. From  such a perturbation perspective, based on \eqref{intensity18765} we will construct an approximation
to $f(2^{-N}\emph{\textbf{k}})$ that is the solution to \eqref{intensity8765}. Correspondingly,
Scheme \ref{jiusuanfa} is  concretized   as follows:\\

\begin{tcolorbox}
\begin{sche}\label{xinsuanfa}
\textbf{step $(1')$}:
Based on the intensity set  in \eqref{zongnengliang}:
\begin{align}\label{yongnaxienengliang}I_{\Xi_{\Omega,N}}=\{I_{N, \emph{\textbf{k}}}, I_{N, \emph{\textbf{k}}^{'}}, I_{N, \emph{\textbf{k}}^{''}}
: (\emph{\textbf{k}}, \emph{\textbf{k}}', \emph{\textbf{k}}'')\in \Xi_{\Omega,N}\},\end{align}
 for any $ \emph{\textbf{k}}\in  \Lambda_{\Omega,N}$ we will construct an approximation
  $\mathring{f}(2^{-N}\emph{\textbf{k}})$ to the $N$-level  data $f(2^{-N}\emph{\textbf{k}})$ through \eqref{intensity18765}.\\
\textbf{step $(2')$}: Based on  $\{\mathring{f}(2^{-N}\emph{\textbf{k}}): \emph{\textbf{k}}\in  \Lambda_{\Omega,N}\}$,
 the approximation to $f$ on $\Omega$ is defined as
\begin{align}\label{formulationstep2}f\approx\sum_{\emph{\textbf{k}}\in  \Lambda_{\Omega,N}}\mathring{f}(2^{-N}\emph{\textbf{k}})\phi(2^{N}\cdot-\emph{\textbf{k}}),\end{align}
where the refinable function $\phi$ is as in \eqref{roi}. Establish the approximation error on $\Omega$:
\begin{align}\label{wuchaguji}
\|\big(f-\sum_{\emph{\textbf{k}}\in  \Lambda_{\Omega,N}}\mathring{f}(2^{-N}\emph{\textbf{k}})\phi(2^{N}\cdot-\emph{\textbf{k}})\big)|_{\Omega}\|_{L^2(\mathbb{R}^d)}.
\end{align}
\end{sche}
\end{tcolorbox}

\subsubsection{\textbf{Analysis of the recovery error in} \textbf{step $(1')$} \textbf{of Scheme} \ref{xinsuanfa}:
\textbf{a motivation for admissible reference wave}}\label{erroranalysis}
Note that $\|\big(f-\sum_{\emph{\textbf{k}}\in  \Lambda_{\Omega,N}}f(2^{-N}\emph{\textbf{k}})\phi(2^{N}\cdot-\emph{\textbf{k}})\big)|_{\Omega}\|_{L^2(\mathbb{R}^d)}$
can be estimated by \eqref{cxzbc}, and $\mathring{f}(2^{-N}\emph{\textbf{k}})$  in  \textbf{step $(1')$}  is an approximation to $f(2^{-N}\emph{\textbf{k}}).$
Therefore, the error estimate for \eqref{wuchaguji} is closely related to  $|f(2^{-N}\emph{\textbf{k}})-
\mathring{f}(2^{-N}\emph{\textbf{k}})|$ and
will be addressed  in Note \ref{wuchafenxil}.
Now the key task  is to   construct   the approximation $\mathring{f}(2^{-N}\emph{\textbf{k}})$
to $f(2^{-N}\emph{\textbf{k}})$.
Recall that
$\mathring{f}(2^{-N}\emph{\textbf{k}})$ is from \eqref{intensity18765} which is the perturbation of
\eqref{intensity8765}. It is naturally required that \eqref{intensity8765} has a unique solution.
The following gives a sufficient and necessary condition for such a uniqueness requirement.

\begin{theo}\label{prop123} As in \eqref{intensity123}, the interference wave we use  at level $N$ is denoted by  $g_{N}$.
Additionally, the set
$\Xi_{\Omega,N}=\{(\emph{\textbf{k}}, \emph{\textbf{k}}', \emph{\textbf{k}}''): \emph{\textbf{k}}\in  \Lambda_{\Omega,N},
2^{-N}\emph{\textbf{k}}^{'}\approx2^{-N}\emph{\textbf{k}}, 2^{-N}\emph{\textbf{k}}^{''}\approx2^{-N}\emph{\textbf{k}}\}$ is as in
\eqref{lingwaijihe}.
Then $f(2^{-N}\emph{\textbf{k}})=\Re(f(2^{-N}\emph{\textbf{k}}))+\textbf{i}\Im(f(2^{-N}\emph{\textbf{k}}))$ with $\emph{\textbf{k}}\in  \Lambda_{\Omega,N}$ is the unique solution to   \eqref{intensity8765}
if and only if
\begin{align}\label{cba1}\mu_{N;\emph{\textbf{k}}, \emph{\textbf{k}}', \emph{\textbf{k}}''}:=-\Im[(g_{N}(2^{-N}\emph{\textbf{k}})-g_{N}(2^{-N}\emph{\textbf{k}}^{'}))(\overline{g_{N}}(2^{-N}\emph{\textbf{k}})-\overline{g_{N}}(2^{-N}\emph{\textbf{k}}^{''}))]\neq0.\end{align}
Particularly, if $\mu_{N;\emph{\textbf{k}}, \emph{\textbf{k}}', \emph{\textbf{k}}''}\neq0$ then  $f(2^{-N}\emph{\textbf{k}})
$ can be determined uniquely  by
\begin{align}\label{BVC123} \begin{array}{lllll}\left[\begin{array}{cccccccccc}\Re(f(2^{-N}\emph{\textbf{k}}))\\
\Im(f(2^{-N}\emph{\textbf{k}}))
\end{array}\right]\\
=\displaystyle \frac{1}{2\mu_{N;\emph{\textbf{k}}, \emph{\textbf{k}}', \emph{\textbf{k}}''}}\left[\begin{array}{cccccccccc}
 e &-d \\-h &c\end{array}\right]\left[\begin{array}{cccccccccc}I_{N,\emph{\textbf{k}}}-A_{N,\emph{\textbf{k}}, \emph{\textbf{k}}'}+|g_{N}(2^{-N}\emph{\textbf{k}}^{'})|^{2}-|g_{N}(2^{-N}\emph{\textbf{k}})|^{2}\\
I_{N,\emph{\textbf{k}}}-A_{N,\emph{\textbf{k}}, \emph{\textbf{k}}''}+|g_{N}(2^{-N}\emph{\textbf{k}}^{''})|^{2}-|g_{N}(2^{-N}\emph{\textbf{k}})|^{2}\end{array}\right],
\end{array}
\end{align}
where $I_{N,\emph{\textbf{k}}}$, $A_{N; \emph{\textbf{k}}, \emph{\textbf{k}}'}$
and $A_{N; \emph{\textbf{k}}, \emph{\textbf{k}}''}$ are as in \eqref{intensity8765}, and
\begin{align}\label{somecoefficients}\left\{\begin{array}{lllll}
c:=\Re(g_{N}(2^{-N}\emph{\textbf{k}})-g_{N}(2^{-N}\emph{\textbf{k}}^{'})),&d:=\Im(g_{N}(2^{-N}\emph{\textbf{k}})-g_{N}(2^{-N}\emph{\textbf{k}}^{'})),\\
 h:=\Re(g_{N}(2^{-N}\emph{\textbf{k}})-g_{N}(2^{-N}\emph{\textbf{k}}^{''})),&e:=\Im(g_{N}(2^{-N}\emph{\textbf{k}})-g_{N}(2^{-N}\emph{\textbf{k}}^{''})).
\end{array}
\right.
\end{align}
\end{theo}
\begin{proof}
Denote $x=\Re(x)+\textbf{i}\Im(x)$. Then it is directly from \eqref{intensity8765} that
\begin{align}\label{positivemood}\left\{\begin{array}{lllll}
 2(\Re(x)\Re(g_{N}(2^{-N}\emph{\textbf{k}}))+\Im(x)\Im(g_{N}(2^{-N}\emph{\textbf{k}})))&=-|x|^2+I_{N,\emph{\textbf{k}}}-|g_{N}(2^{-N}\emph{\textbf{k}})|^{2},\\
 2(\Re(x)\Re(g_{N}(2^{-N}\emph{\textbf{k}}^{'}))+\Im(x)\Im(g_{N}(2^{-N}\emph{\textbf{k}}^{'})))&=-|x|^2+A_{N,\emph{\textbf{k}}, \emph{\textbf{k}}^{'}}-|g_{N}(2^{-N}\emph{\textbf{k}}^{'})|^{2},\\
 2(\Re(x)\Re(g_{N}(2^{-N}\emph{\textbf{k}}^{''}))+\Im(x)\Im(g_{N}(2^{-N}\emph{\textbf{k}}^{''})))&=-|x|^2+A_{N,\emph{\textbf{k}}, \emph{\textbf{k}}^{''}}-|g_{N}(2^{-N}\emph{\textbf{k}}^{''})|^{2}.
\end{array}
\right.
\end{align}
From this, we have
\begin{align}\label{BVC}2\left[\begin{array}{cccccccccc}
 c &d \\h &e\end{array}\right] \left[\begin{array}{cccccccccc}\Re(x)\\
\Im(x)
\end{array}\right]= \left[\begin{array}{cccccccccc}I_{N,\emph{\textbf{k}}}-A_{N,\emph{\textbf{k}}, \emph{\textbf{k}}'}+|g_{N}(2^{-N}\emph{\textbf{k}}^{'})|^{2}-|g_{N}(2^{-N}\emph{\textbf{k}})|^{2}\\
I_{N,\emph{\textbf{k}}}-A_{N,\emph{\textbf{k}}, \emph{\textbf{k}}''}+|g_{N}(2^{-N}\emph{\textbf{k}}^{''})|^{2}-|g_{N}(2^{-N}\emph{\textbf{k}})|^{2}\end{array}\right].
\end{align}
Note that  $f(2^{-N}\emph{\textbf{k}})$ is a solution to \eqref{intensity8765}.
Equivalently,
$(\Re(f(2^{-N}\emph{\textbf{k}})), $
$\Im(f(2^{-N}\emph{\textbf{k}})))^{T}$ is a solution to \eqref{BVC}.
Note that  $(\Re(f(2^{-N}\emph{\textbf{k}})), $
$\Im(f(2^{-N}\emph{\textbf{k}})))^{T}$ can be determined uniquely  by  \eqref{BVC}  if  and only if
\begin{align}\label{hanglieshi}\det\Big(\left[\begin{array}{cccccccccc}
 c &d \\h &e\end{array}\right]\Big)=ce-dh\neq0.\end{align}
 Through the direct calculation we have
 \begin{align}\label{fcxvbnm}\mu_{N;\emph{\textbf{k}}, \emph{\textbf{k}}', \emph{\textbf{k}}''}=
 ce-dh.\end{align}
Then $(\Re(f(2^{-N}\emph{\textbf{k}})),$
$\Im(f(2^{-N}\emph{\textbf{k}})))^{T}$ can be determined uniquely   if and only if  $\mu_{N;\emph{\textbf{k}}, \emph{\textbf{k}}', \emph{\textbf{k}}''}\neq0$.  The first part of the theorem is proved.  Now \eqref{BVC123} is derived from \eqref{BVC}
and \eqref{fcxvbnm}.
\end{proof}

Recall again that      $I_{N,\emph{\textbf{k}}'}\approx A_{N,\emph{\textbf{k}}, \emph{\textbf{k}}'}
$ and $I_{N,\emph{\textbf{k}}''}\approx A_{N,\emph{\textbf{k}}, \emph{\textbf{k}}''}
$.
 Motivated by  Theorem  \ref{prop123} \eqref{BVC123}, in what follows we construct  the  approximation $\mathring{f}(2^{-N}\emph{\textbf{k}})$
to $f(2^{-N}\emph{\textbf{k}})$.
It is the central task  for  \textbf{step $(1')$} in Scheme \ref{xinsuanfa}.
%It is the formulation of \textbf{step (1)} of single-shot holography.

\begin{theo}\label{chubuguji}
As in \eqref{intensity123}, suppose that $g_{N}$ is the reference wave we use at level $N$. The set
$\Xi_{\Omega,N}=\{(\emph{\textbf{k}}, \emph{\textbf{k}}', \emph{\textbf{k}}''): \emph{\textbf{k}}\in  \Lambda_{\Omega,N},
2^{-N}\emph{\textbf{k}}^{'}\approx2^{-N}\emph{\textbf{k}}, 2^{-N}\emph{\textbf{k}}^{''}\approx2^{-N}\emph{\textbf{k}}\}$ is as in
\eqref{lingwaijihe},
and  $\mu_{N;\emph{\textbf{k}}, \emph{\textbf{k}}', \emph{\textbf{k}}''}$
is  as in \eqref{cba1}. For any $(\emph{\textbf{k}}, \emph{\textbf{k}}', \emph{\textbf{k}}'')\in \Xi_{\Omega, N}$
such that $\mu_{N;\emph{\textbf{k}}, \emph{\textbf{k}}', \emph{\textbf{k}}''}\neq0$,
define a complex number  $\mathring{f}(2^{-N}\emph{\textbf{k}}):=\mathring{\Re}(f(2^{-N}\emph{\textbf{k}}))$
$+\textbf{i}\mathring{\Im}(f(2^{-N}\emph{\textbf{k}}))$
by
\begin{align}\label{BVC1234} \begin{array}{lllll}\left[\begin{array}{cccccccccc}\mathring{\Re}(f(2^{-N}\emph{\textbf{k}}))\\
\mathring{\Im}(f(2^{-N}\emph{\textbf{k}}))
\end{array}\right]
=\displaystyle \frac{1}{2\mu_{N;\emph{\textbf{k}}, \emph{\textbf{k}}', \emph{\textbf{k}}''}}\left[\begin{array}{cccccccccc}
 e &-d \\-h &c\end{array}\right]\left[\begin{array}{cccccccccc}I_{N,\emph{\textbf{k}}}-I_{N, \emph{\textbf{k}}'}+|g_{N}(2^{-N}\emph{\textbf{k}}^{'})|^{2}-|g_{N}(2^{-N}\emph{\textbf{k}})|^{2}\\
I_{N,\emph{\textbf{k}}}-I_{N, \emph{\textbf{k}}''}+|g_{N}(2^{-N}\emph{\textbf{k}}^{''})|^{2}-|g_{N}(2^{-N}\emph{\textbf{k}})|^{2}\end{array}\right],
\end{array}
\end{align}
where $c, d, e$ and $h$ are as in \eqref{somecoefficients}.
Then
\begin{align}\label{cucaoguji}\begin{array}{cccccccccc}
\Big\|\left[\begin{array}{cccccccccc}\Re(f(2^{-N}\emph{\textbf{k}}))\\
\Im(f(2^{-N}\emph{\textbf{k}}))
\end{array}\right]-\left[\begin{array}{cccccccccc}\mathring{\Re}(f(2^{-N}\emph{\textbf{k}}))\\
\mathring{\Im}(f(2^{-N}\emph{\textbf{k}}))
\end{array}\right]\Big\|_{2}\\
\leq\frac{\max\{|g_{N}(2^{-N}\emph{\textbf{k}})-g_{N}(2^{-N}\emph{\textbf{k}}')|, |\overline{g_{N}}(2^{-N}\emph{\textbf{k}})-\overline{g_{N}}(2^{-N}\emph{\textbf{k}}'')|\}}{|\Im[(g_{N}(2^{-N}\emph{\textbf{k}})-g_{N}(2^{-N}\emph{\textbf{k}}^{'}))(\overline{g_{N}}(2^{-N}\emph{\textbf{k}})-\overline{g_{N}}(2^{-N}\emph{\textbf{k}}^{''}))]|}
\sqrt{(I_{N,\emph{\textbf{k}}'}-A_{N,\emph{\textbf{k}},\emph{\textbf{k}}'})^{2}+(I_{N,\emph{\textbf{k}}''}-A_{N,\emph{\textbf{k}},\emph{\textbf{k}}''})^{2}}.
\end{array}
\end{align}
\end{theo}
\begin{proof}
For the matrix
$\left[\begin{array}{cccccccccc}
 e &-d \\-h &c\end{array}\right]$,
 it is straightforward
 that
 \begin{align}\label{fanshudengjiaxing}\begin{array}{lllllll}
 \Big\|\left[\begin{array}{cccccccccc}
 e &-d \\-h &c\end{array}\right]\Big\|^{2}_{2}&\leq\Big\|\left[\begin{array}{cccccccccc}
 e &-d \\-h &c\end{array}\right]\Big\|_{1}\Big\|\left[\begin{array}{cccccccccc}
 e &-d \\-h &c\end{array}\right]\Big\|_{\infty}\\
 &=\max\{|e|+|h|, |c|+|d|\}\max\{|e|+|d|, |h|+|c|\}.
 \end{array}\end{align}
Through the direct estimate we have
\begin{align}\label{qiujie8716}\begin{array}{lllllll}\begin{array}{lllllll}
\Big\|\displaystyle\frac{1}{2\mu_{N;\emph{\textbf{k}}, \emph{\textbf{k}}', \emph{\textbf{k}}''}}\left[\begin{array}{cccccccccc}
 e &-d \\-h &c\end{array}\right]\Big\|_{2}
\end{array}&\leq\displaystyle\frac{1}{|\mu_{N;\emph{\textbf{k}}, \emph{\textbf{k}}', \emph{\textbf{k}}''}|}
\max\{|e|,|d|,|h|,|c|\} \quad (\ref{qiujie8716}A)\\
&\leq\displaystyle\frac{1}{|\mu_{N;\emph{\textbf{k}}, \emph{\textbf{k}}', \emph{\textbf{k}}''}|}\max\{\sqrt{c^2+d^2}, \sqrt{e^2+h^2}\}\\
&=\frac{\max\{|g_{N}(2^{-N}\emph{\textbf{k}})-g_{N}(2^{-N}\emph{\textbf{k}}')|, |\overline{g_{N}}(2^{-N}\emph{\textbf{k}})-\overline{g_{N}}(2^{-N}\emph{\textbf{k}}'')|\}}{|\Im[(g_{N}(2^{-N}\emph{\textbf{k}})-g_{N}(2^{-N}\emph{\textbf{k}}^{'}))(\overline{g_{N}}(2^{-N}\emph{\textbf{k}})-\overline{g_{N}}(2^{-N}\emph{\textbf{k}}^{''}))]|},
(\ref{qiujie8716} B)
\end{array}
\end{align}
where $(\ref{qiujie8716}A)$ is derived from \eqref{fanshudengjiaxing}, and  $(\ref{qiujie8716}B)$ is  from  \eqref{cba1}
and \eqref{somecoefficients}.
Now \eqref{cucaoguji} is  from \eqref{BVC123}, \eqref{BVC1234} and \eqref{qiujie8716}.
\end{proof}

\begin{note}\label{wuchafenxil}
Recall that we use the formula in \eqref{formulationstep2} to approximate $f$ on $\Omega$,
namely,
\begin{align}\label{formulationstep20}f(\emph{\textbf{x}})\approx\sum_{\emph{\textbf{k}}\in  \Lambda_{\Omega,N}}\mathring{f}(2^{-N}\emph{\textbf{k}})\phi(2^{N}\emph{\textbf{x}}-\emph{\textbf{k}}).\end{align}
Correspondingly,  the approximation error is analyzed as follows,
\begin{align}\label{error11}\begin{array}{lllllll}
\displaystyle \|\big(f(\emph{\textbf{x}})-\sum_{\emph{\textbf{k}}\in  \Lambda_{\Omega,N}}\mathring{f}(2^{-N}\emph{\textbf{k}})\phi(2^{N}\emph{\textbf{x}}-\emph{\textbf{k}})\big)|_{\Omega}\|_{L^2(\mathbb{R}^d)}\\
\displaystyle\leq\|\big(f(\emph{\textbf{x}})-\sum_{\emph{\textbf{k}}\in  \Lambda_{\Omega,N}}f(2^{-N}\emph{\textbf{k}})\phi(2^{N}\cdot-\emph{\textbf{k}})\big)|_{\Omega}\|_{L^2(\mathbb{R}^d)}\quad (\ref{error11}A)\\
+\displaystyle\|\big(\sum_{\emph{\textbf{k}}\in  \Lambda_{\Omega,N}}(f(2^{-N}\emph{\textbf{k}})-
\mathring{f}(2^{-N}\emph{\textbf{k}}))\phi(2^{N}\emph{\textbf{x}}-\emph{\textbf{k}})\big)|_{\Omega}\|_{L^2(\mathbb{R}^d)}. \quad (\ref{error11}B)
\end{array}
\end{align}
The   term in $(\ref{error11}A)$ can be estimated by \eqref{cxzbc}.
% We intend to apply  \eqref{cucaoguji} and   the unit decomposition in \eqref{yydnew},
%namely,
%\begin{align}\label{danweifenjie1} 1\equiv\sum_{\emph{\textbf{k}}\in \mathbb{Z}^{d}}\phi(\cdot-\emph{\textbf{k}})\end{align}
%to establish the error estimate for $(\ref{error11}B)$. Such an estimation from \eqref{danweifenjie1} can be given  if $\phi\geq0$ and   $|f(2^{-N}\emph{\textbf{k}})-
%\mathring{f}(2^{-N}\emph{\textbf{k}})|$ can be bounded uniformly on $\Lambda_{\Omega,N}$.
To estimate the term in $(\ref{error11}B)$,
we next  analyze the upper bound in \eqref{cucaoguji} for $|f(2^{-N}\emph{\textbf{k}})-
\mathring{f}(2^{-N}\emph{\textbf{k}})|$.
\end{note}

\begin{rem}\label{FDVCX}
The upper  bound in \eqref{cucaoguji} consists of the  two terms
\begin{align}\label{PTGH}\frac{\max\{|g_{N}(2^{-N}\emph{\textbf{k}})-g_{N}(2^{-N}\emph{\textbf{k}}')|, |\overline{g_{N}}(2^{-N}\emph{\textbf{k}})-\overline{g_{N}}(2^{-N}\emph{\textbf{k}}'')|\}}{|\Im[(g_{N}(2^{-N}\emph{\textbf{k}})-g_{N}(2^{-N}\emph{\textbf{k}}^{'}))(\overline{g_{N}}(2^{-N}\emph{\textbf{k}})-\overline{g_{N}}(2^{-N}\emph{\textbf{k}}^{''}))]|}
\end{align}
and
\begin{align}\label{PTGH1}
\sqrt{(I_{N,\emph{\textbf{k}}'}-A_{N,\emph{\textbf{k}},\emph{\textbf{k}}'})^{2}+(I_{N,\emph{\textbf{k}}''}-A_{N,\emph{\textbf{k}},\emph{\textbf{k}}''})^{2}},
\end{align}
where $(\emph{\textbf{k}}, \emph{\textbf{k}}', \emph{\textbf{k}}'')\in  \Xi_{\Omega,N}$.
As mentioned previously,  $I_{N,\emph{\textbf{k}}'}\approx A_{N,\emph{\textbf{k}}, \emph{\textbf{k}}'}
$ and $I_{N,\emph{\textbf{k}}''}\approx A_{N,\emph{\textbf{k}}, \emph{\textbf{k}}''}
$. Then the term   \eqref{PTGH1} is small.
 Although  the term  \eqref{PTGH} is independent of $f$, it is  complicated formally  and  related to $\emph{\textbf{k}}, \emph{\textbf{k}}'$ and $ \emph{\textbf{k}}''$.  It is expected to be  uniformly bounded on
$\Xi_{\Omega,N}.$
To sum up, if  the term in \eqref{PTGH1} is small and  the term in \eqref{PTGH}
is uniformly bounded then     it follows from \eqref{cucaoguji} that  $|f(2^{-N}\emph{\textbf{k}})-
\mathring{f}(2^{-N}\emph{\textbf{k}})|$ can be controlled. Such a perspective  will be helpful for estimating
$(\ref{error11}B)$ in section \ref{zqsd}.
\end{rem}

\subsection{Admissible reference wave}\label{yixiedingyi}
As in Remark \ref{FDVCX}, it is expected that the  term \eqref{PTGH} is uniformly bounded on $\Xi_{\Omega,N}.$
Motivated by such an expectation, in what follows we introduce the concept of an  admissible reference wave.

\begin{defi}\label{addefinition}
Let $g_{N}$ be the    reference  wave used at level $N$  and
the finite  set  $\Xi_{\Omega,N}\subseteq \mathbb{R}^{3d}$.
%Let the two sets  $\Lambda_{\Omega,N}\subseteq \mathbb{Z}^{d}$ and $\Xi_{\Omega,N}\subseteq \mathbb{R}^{3d}$ be as in \eqref{fft1} and  \eqref{lingwaijihe}, respectively,
%such that for  any $(\emph{\textbf{k}}, \emph{\textbf{k}}', \emph{\textbf{k}}'')\in \Xi_{\Omega,N}$ it holds that
%$\emph{\textbf{k}}\in \Lambda_{\Omega,N}$.
%A set  sits in $ \mathbb{R}^{3d}$.
We say that $g_{N}$ is $\gamma_{N}$-admissible on $\Xi_{\Omega,N}$ w.r.t   level $N$ where $0<\gamma_{N}<\infty$, if the following two items holds:

(i) For any $(\emph{\textbf{k}}, \emph{\textbf{k}}', \emph{\textbf{k}}'')\in
\Xi_{\Omega,N}$, it holds that  $g_{N}(2^{-N}\emph{\textbf{k}})\neq g_{N}(2^{-N}\emph{\textbf{k}}')$ and
$g_{N}(2^{-N}\emph{\textbf{k}})\neq g_{N}(2^{-N}\emph{\textbf{k}}'')$.

(ii) Admissibility condition:
\begin{align}\label{p45tgy}
\displaystyle\max_{(\emph{\textbf{k}}, \emph{\textbf{k}}', \emph{\textbf{k}}'')\in  \Xi_{\Omega,N}}\frac{\max\{|g_{N}(2^{-N}\emph{\textbf{k}})-g_{N}(2^{-N}\emph{\textbf{k}}')|, |\overline{g_{N}}(2^{-N}\emph{\textbf{k}})-\overline{g_{N}}(2^{-N}\emph{\textbf{k}}'')|\}}{|\Im[(g_{N}(2^{-N}\emph{\textbf{k}})-g_{N}(2^{-N}\emph{\textbf{k}}^{'}))(\overline{g_{N}}(2^{-N}\emph{\textbf{k}})-\overline{g_{N}}(2^{-N}\emph{\textbf{k}}^{''}))]|}\leq\gamma_{N}<\infty.
\end{align}
%where $\emph{\textbf{k}}, \emph{\textbf{k}}', \emph{\textbf{k}}''\in \mathbb{R}^{d}.$
\end{defi}

Based on Definition \ref{addefinition}, next  we give the definition of  uniform admissibility.

\begin{defi}\label{yizhirongxuxing}
 Suppose that $\{g_{N}\}^{\infty}_{N=1}$ is a sequence of reference waves such that every  $g_{N}$ is
$\gamma$-admissible on $\Xi_{\Omega,N}$ w.r.t  level $N$. Moreover,  if the uniform   boundedness  condition holds:
\begin{align}\label{PLKCV}
%\beta_{\Xi_{\Omega}}:=\sup_{N\geq1}\sup_{(\emph{\textbf{k}}, \emph{\textbf{k}}', \emph{\textbf{k}}'')\in
%\Xi_{\Omega,N}}\{|g_{N}(2^{-N}\emph{\textbf{k}})|, |g_{N}(2^{-N}\emph{\textbf{k}}')|, |g_{N}(2^{-N}\emph{\textbf{k}}'')|\}<\infty,
\sup_{N\geq1}\sup_{(\emph{\textbf{k}}, \emph{\textbf{k}}', \emph{\textbf{k}}'')\in
\Xi_{\Omega,N}}\{|g_{N}(2^{-N}\emph{\textbf{k}})|, |g_{N}(2^{-N}\emph{\textbf{k}}')|, |g_{N}(2^{-N}\emph{\textbf{k}}'')|\}<\infty
\end{align}
 then we say that $\{g_{N}\}^{\infty}_{N=1}$ is uniformly $\gamma$-admissible on $\{\Xi_{\Omega,N}\}^{\infty}_{N=1}$.
\end{defi}

The following concerns on the correlation between the admissibility and the uniqueness of the equation  system
\eqref{intensity8765}.

\begin{prop}\label{rongxuxing} Suppose that $\Xi_{\Omega,N}\subseteq \mathbb{R}^{3d}$
and $g_{N}$
are  as in Theorem \ref{prop123}.
If $g_{N}$ is $\gamma_{N}$-admissible on $\Xi_{\Omega,N}$ w.r.t   level $N$,
then for any $(\emph{\textbf{k}}, \emph{\textbf{k}}', \emph{\textbf{k}}'')\in \Xi_{\Omega,N}$ it holds that
$\mu_{N;\emph{\textbf{k}}, \emph{\textbf{k}}', \emph{\textbf{k}}''}$ defined in \eqref{cba1} is nonzero, or equivalently
the equation  system \eqref{intensity8765} has a unique solution.
\end{prop}
\begin{proof}
It follows from  Definition \ref{addefinition} (i) that
\begin{align}\notag \max\{|g_{N}(2^{-N}\emph{\textbf{k}})-g_{N}(2^{-N}\emph{\textbf{k}}')|, |\overline{g_{N}}(2^{-N}\emph{\textbf{k}})-\overline{g_{N}}(2^{-N}\emph{\textbf{k}}'')|\}
>0.\end{align} If $0=\mu_{N;\emph{\textbf{k}}, \emph{\textbf{k}}', \emph{\textbf{k}}''}=-\Im[(g_{N}(2^{-N}\emph{\textbf{k}})-g_{N}(2^{-N}\emph{\textbf{k}}^{'}))(\overline{g_{N}}(2^{-N}\emph{\textbf{k}})-\overline{g_{N}}(2^{-N}\emph{\textbf{k}}^{''}))]$
then
$$\frac{\max\{|g_{N}(2^{-N}\emph{\textbf{k}})-g_{N}(2^{-N}\emph{\textbf{k}}')|, |\overline{g_{N}}(2^{-N}\emph{\textbf{k}})-\overline{g_{N}}(2^{-N}\emph{\textbf{k}}'')|\}}{|\Im[(g_{N}(2^{-N}\emph{\textbf{k}})-g_{N}(2^{-N}\emph{\textbf{k}}^{'}))(\overline{g_{N}}(2^{-N}\emph{\textbf{k}})-\overline{g_{N}}(2^{-N}\emph{\textbf{k}}^{''}))]|}=\infty.$$
Consequently,  $g_{N}$ is not  $\gamma_{N}$-admissible on $\Xi_{\Omega,N}.$
This leads to a contradiction.  Then $\mu_{N;\emph{\textbf{k}}, \emph{\textbf{k}}', \emph{\textbf{k}}''}\neq0$.
By Theorem \ref{prop123}, the equation  system \eqref{intensity8765} has a unique solution.
\end{proof}

\begin{rem}
(1) Suppose that
$\{g_{N}\}^{\infty}_{N=1}$ is uniformly $\gamma$-admissible on $\{\Xi_{\Omega,N}\}^{\infty}_{N=1}$,
where $\Xi_{\Omega,N}$ is as in Theorem \ref{chubuguji} such that for any $(\emph{\textbf{k}}, \emph{\textbf{k}}', \emph{\textbf{k}}'')
\in \Xi_{\Omega,N}$
we have $\emph{\textbf{k}}\in \Lambda_{\Omega,N}$.
For such $\emph{\textbf{k}}$   it follows from \eqref{cucaoguji} and  \eqref{p45tgy} (with $\gamma_{N}=\gamma$)  that
\begin{align}\label{bzxsd}
|f(2^{-N}\emph{\textbf{k}})-\mathring{f}(2^{-N}\emph{\textbf{k}})|\leq\gamma\sqrt{(I_{N,\emph{\textbf{k}}'}-A_{N,\emph{\textbf{k}},\emph{\textbf{k}}'})^{2}+(I_{N,\emph{\textbf{k}}''}-A_{N,\emph{\textbf{k}},\emph{\textbf{k}}''})^{2}}.
\end{align}
(2) For a general reference wave $g_{N}$, it is complicated  to establish the bound $\gamma_{N}$ in
\eqref{p45tgy}. Recall that the plane  and spherical reference waves are two types of reference waves commonly applied in holography
(c.f. \cite{Fourieroptics,holographic}).
We will address the admissibility for such two types of  reference waves.
\end{rem}

\subsection{Admissible plane and spherical waves on $\mathbb{R}^d$}\label{JKKKK12345}
%The plane and spherical waves are two commonly applied reference waves in holography (c.f. \cite{Fourieroptics}).
We generalize the  concepts  of plane and spherical waves  from the case of $d\leq3$ (c.f. \cite{Fourieroptics,holographic})
 to the general  case of $d\in \mathbb{N}$. Throughout the paper, the plane and spherical waves are denoted by
  $ae^{\textbf{i}\mathcal{K}\cdot \emph{\textbf{x}}}$ and $\frac{b}{\|\emph{\textbf{x}}\|_{2}}e^{\textbf{i}\nu\|\emph{\textbf{x}}\|_{2}}$ ($\mathcal{K}, \emph{\textbf{x}}\in \mathbb{R}^{d}, a>0, b>0,  \nu\in \mathbb{R}\setminus\{0\}$), respectively. %%%% Ó¦¸ÃÔÚµÚÒ»½Úʱºò½éÉÜƽÃ沨ºÍÇòÃ沨
  In optics, $\mathcal{K}$ is commonly denoted by  $(\frac{2\pi}{\lambda_{1}}, \ldots, \frac{2\pi}{\lambda_{d}})$
such that $\lambda_{j}$ and  $\frac{2\pi}{\lambda_{j}}$ are  called the \textbf{wave length} and   \textbf{wave number} w.r.t  the $j$th coordinate direction, respectively. Similarly, $|\nu|$ is  referred to as the wave number of the spherical wave.
The constants $a$ and $b$ are called  the  amplitudes of the plane wave and the spherical wave (at the unit sphere
 $\{\emph{\textbf{x}}\in \mathbb{R}^d: \|\emph{\textbf{x}}\|_{2}=1\}$),  respectively.

\begin{note}\label{budnegyuling}
From its definition,  a spherical wave is   defined on   $\mathbb{R}^{d}\setminus\{\textbf{0}\}$.
This is  different from the domain  $\mathbb{R}^{d}$ of a plane wave.
If the reference wave $g_{N}$  in the intensity formula   \eqref{intensity123}  is chosen as  a  spherical wave,
 then it is required that  $\textbf{0}\notin
 \Lambda_{\Omega,N}.$ The following concerns on a sufficient and necessary condition for
 $L_{l,\min}$ and $L_{l,\max}$ (defined in \eqref{zuixiao})
  such that $\textbf{0}\notin
 \Lambda_{\Omega,N}$ for every level  $N\geq1$.
\end{note}

\begin{prop}\label{dbdn}
As in \eqref{yixiejihao}, the ROI is denoted by $\Omega$ such that $\mathbb{Z}_{\Omega}^{d}:=\Omega\cap \mathbb{Z}^{d}\neq\emptyset$.
Moreover, as in \eqref{fft1} the set $\Lambda_{\Omega,N}$ is defined as
 \begin{align}\label{fft12345}\begin{array}{lll}
 \Lambda_{\Omega,N}=\{(k_{1},\ldots,  k_{d})\in \mathbb{Z}^{d}: 2^{N}L_{l,\min}-M_{l}\leq k_{l}\leq2^{N}L_{l,\max}, l=1, \ldots, d\},
\end{array}\end{align}
where every  $M_{l}>0$ and
\begin{align}\label{zuocao}L_{l,\min}=\min\{k_{l}: (k_{1}, \ldots, k_{d})\in\mathbb{Z}_{\Omega}^{d}\}, L_{l,\max}=\max\{k_{l}: (k_{1}, \ldots, k_{d})\in\mathbb{Z}_{\Omega}^{d}\}.\end{align}
Then $\textbf{0}\notin
 \Lambda_{\Omega,N}$ for every level  $N\geq1$ if and only if  there exists $l_{0}\in
 \{1, \ldots, d\}$ such that  $2L_{l_{0},\min}-M_{l_{0}}>0$ or
 $L_{l_{0},\max}<0$. Moreover, if $\textbf{0}\notin
 \Lambda_{\Omega,N}$ for every  $N\geq1$ then
 $\|2^{-N}\emph{\textbf{k}}\|_{2}\geq\min\{|L_{l_{0},\min}-2^{-1}M_{l_{0}}|, |L_{l_{0},\max}|\}>0$
   for any  $\emph{\textbf{k}}\in \Lambda_{\Omega,N}$, where $l_{0}\in
 \{1, \ldots, d\}$ is as above
   such that $2L_{l_{0},\min}-M_{l_{0}}>0$ or
 $L_{l_{0},\max}<0$.
\end{prop}
\begin{proof}
We prove the first part.
Note that  $\Lambda_{\Omega,N}$ can be also expressed  as
\begin{align}\label{lzisx} \Lambda_{\Omega,N}=\big([2^{N}L_{1,\min}-M_{1}, 2^{N}L_{1,\max}]\times \cdots\times [2^{N}L_{d,\min}-M_{d},  2^{N}L_{d,\max}]\big)\cap \mathbb{Z}^{d}.\end{align}
The \textbf{necessity} can be obtained by \eqref{lzisx} with $N=1$.
\textbf{Sufficiency}: If $2L_{l_{0},\min}-M_{l_{0}}>0$ or $L_{l_{0},\max}<0$
then $0\notin [2L_{l_{0},\min}-M_{l_{0}}, 2L_{l_{0},\max}]$. By \eqref{lzisx} we have
$\textbf{0}\notin\Lambda_{\Omega,1}.$
If $2L_{l_{0},\min}-M_{l_{0}}>0$ then it follows from $M_{l_{0}}>0$ that  $2^{N}L_{l_{0},\min}-M_{l_{0}}>0$ for all $N>1$,
and consequently $0\notin [2^{N}L_{l_{0},\min}-M_{l_{0}}, 2^{N}L_{l_{0},\max}]$
and $\textbf{0}\notin\Lambda_{\Omega,N}$.
Similarly, if $L_{l_{0},\max}<0$ then for all $N>1$  we have $2^{N}L_{l_{0},\max}<0$ and consequently  $\textbf{0}\notin\Lambda_{\Omega,N}$.

%Clearly, $\textbf{0}\notin\Lambda_{\Omega,N}$ if and only if
%there exists $l_{0, N}\in \{1, \ldots, d\}$ such that $0\notin
%[2^{N}L_{l_{0, N},\min}-M_{1}, 2^{N}L_{l_{0, N},\max}]$. This is equivalent to that
%$2^{N}L_{l_{0, N},\max}<0$ or $2^{N}L_{l_{0, N},\min}-M_{1}>0.$
%
%If all the $d$ intervals $[2^{N}L_{l,\min}-M_{l},  2^{N}L_{l,\max}], l=1, \ldots, d$
%contain  $0$ then it is clear that
%
%For any level  $N\geq1$,
%it is straightforward that
%if and only if $2^{N}L_{l,\min}>M_{l}$ or
% $L_{l,\max}<0$ for all $l\in \{1, \ldots, d\}$. Now the proof can be completed by this.

We next prove the second part. Choose any $\emph{\textbf{k}}=(k_{1}, \ldots, k_{d})\in \Lambda_{\Omega,N}$
such that $k_{l_{0}}\in [2^{N}L_{l_{0},\min}-M_{l_{0}}, 2^{N}L_{l_{0},\max}]$.
Since $\textbf{0}\notin\Lambda_{\Omega,N}$, by the above discussion we have  $0\notin [2^{N}L_{l_{0},\min}-M_{l_{0}}, 2^{N}L_{l_{0},\max}]$.
Then $\|2^{-N}\emph{\textbf{k}}\|_{2}\geq|2^{-N}k_{l_{0}}|\geq\min\{|L_{l_{0},\min}-2^{-N}M_{l_{0}}|,
|L_{l_{0},\max}|\}$.   If $2L_{l_{0},\min}-M_{l_{0}}>0$ then it follows from $M_{l_{0}}>0$ that
$L_{l_{0},\min}-2^{-N}M_{l_{0}}\geq L_{l_{0},\min}-2^{-1}M_{l_{0}}>0$. Therefore, $
\|2^{-N}\emph{\textbf{k}}\|_{2}\geq\min\{|L_{l_{0},\min}-2^{-1}M_{l_{0}}|, |L_{l_{0},\max}|\}$
$=|L_{l_{0},\min}-2^{-1}M_{l_{0}}|$.
For the case of $L_{l_{0},\max}<0$, this inequality can be proved similarly.
% From the above discussion,  if $2L_{l_{0},\min}-M_{l_{0}}>0$ then $k_{l_{0}}\geq2^{N}L_{l_{0},\min}-M_{l_{0}}>0$.
% From this, we have $\|2^{-N}\emph{\textbf{k}}\|_{2}\geq 2^{-N}k_{l_{0}}\geq L_{l_{0},\min}-2^{-N}M_{l_{0}}>0.$
% Note that $M_{l_{0}}>0$. Then $\|2^{-N}\emph{\textbf{k}}\|_{2}\geq L_{l_{0},\min}-2^{-1}M_{l_{0}}>0.$
% Through the similar analysis, if $L_{l_{0},\max}<0$ then $k_{l_{0}}\leq 2^{N}L_{l_{0},\max}$ and consequently,
% $\|2^{-N}\emph{\textbf{k}}\|_{2}\geq |L_{l_{0},\max}|$.
\end{proof}

\subsubsection{A class of uniformly  admissible plane waves}\label{pmianboqingxing}
If
the reference wave $g_{N}$ is chosen  as the plane wave  $a_{N}e^{\textbf{i}\mathcal{K}_{N}\cdot \emph{\textbf{x}}}$,
then the term in \eqref{p45tgy} can be simplified as follows,
 \begin{align}\label{kkk2345}\begin{array}{lllllll}\displaystyle
 \max_{(\emph{\textbf{k}}, \emph{\textbf{k}}', \emph{\textbf{k}}'')\in  \Xi_{\Omega,N}}\frac{\max\{|g_{N}(2^{-N}\emph{\textbf{k}})-g_{N}(2^{-N}\emph{\textbf{k}}')|, |\overline{g_{N}}(2^{-N}\emph{\textbf{k}})-\overline{g_{N}}(2^{-N}\emph{\textbf{k}}'')|\}}{|\Im[(g_{N}(2^{-N}\emph{\textbf{k}})-g_{N}(2^{-N}\emph{\textbf{k}}^{'}))(\overline{g_{N}}(2^{-N}\emph{\textbf{k}})-\overline{g_{N}}(2^{-N}\emph{\textbf{k}}^{''}))]|}\\
\displaystyle =\max_{(\emph{\textbf{k}}, \emph{\textbf{k}}', \emph{\textbf{k}}'')\in  \Xi_{\Omega,N}}\frac{\max\{|1-e^{\textbf{i}\mathcal{K}_{N}\cdot2^{-N}(\emph{\textbf{k}}-\emph{\textbf{k}}')}|, |1-e^{-\textbf{i}\mathcal{K}_{N}\cdot2^{-N}(\emph{\textbf{k}}-\emph{\textbf{k}}'')}|\}}{a_{N}|\Im[(1-e^{\textbf{i}\mathcal{K}_{N}\cdot2^{-N}(\emph{\textbf{k}}-\emph{\textbf{k}}')})(1-e^{-\textbf{i}\mathcal{K}_{N}\cdot2^{-N}(\emph{\textbf{k}}-\emph{\textbf{k}}'')})]|}. \quad
(\ref{kkk2345}A)
 \end{array}\end{align}
% where $(\emph{\textbf{k}}, \emph{\textbf{k}}', \emph{\textbf{k}}'')\in \Xi_{\Omega,N}$.
By direct calculation we have that Definition \ref{addefinition} (i) is equivalent to
\begin{align}\label{TCV}\min\{|1-e^{\textbf{i}\mathcal{K}_{N}\cdot2^{-N}(\emph{\textbf{k}}-\emph{\textbf{k}}')}|,
|1-e^{-\textbf{i}\mathcal{K}_{N}\cdot2^{-N}(\emph{\textbf{k}}-\emph{\textbf{k}}'')}|\}>0.\end{align}
On the other hand, $\max\{|1-e^{\textbf{i}\mathcal{K}_{N}\cdot2^{-N}(\emph{\textbf{k}}-\emph{\textbf{k}}')}|, |1-e^{-\textbf{i}\mathcal{K}_{N}\cdot2^{-N}(\emph{\textbf{k}}-\emph{\textbf{k}}'')}|\}$
$\leq2.
$
Then it follows from  $(\ref{kkk2345}A)$  that   the plane wave $g_{N}$ being admissible is equivalent to
\eqref{TCV} and
\begin{align}\label{mbvc}
\max_{(\emph{\textbf{k}}, \emph{\textbf{k}}', \emph{\textbf{k}}'')\in\Xi_{\Omega,N}}\frac{1}{a_{N}|\Im[(1-e^{\textbf{i}\mathcal{K}_{N}\cdot2^{-N}(\emph{\textbf{k}}-\emph{\textbf{k}}')})(1-e^{-\textbf{i}\mathcal{K}_{N}\cdot2^{-N}(\emph{\textbf{k}}-\emph{\textbf{k}}'')})]|}<\infty.
\end{align}
Now \eqref{TCV} and \eqref{mbvc} implies that $g_{N}$ being admissible
 depends on the  values of $e^{\textbf{i}\mathcal{K}_{N}\cdot \emph{\textbf{x}}}$ at   $2^{-N}(\emph{\textbf{k}}-\emph{\textbf{k}}')$ and $2^{-N}(\emph{\textbf{k}}-\emph{\textbf{k}}'')$.
 Motivated by this, in what follows we construct  a class of uniformly  admissible plane waves
 for Scheme \ref{xinsuanfa}.

\begin{theo}\label{rongxutiaojianpingm}
Suppose that $\{g_{N}\}^{\infty}_{N=1}$ is  a sequence of plane reference waves where  $g_{N}(x_{1}, \ldots, x_{d})=a_{N}e^{\textbf{i}(\frac{2\pi}{\lambda_{1,N}}x_{1}+\cdots+\frac{2\pi}{\lambda_{d,N}}x_{d})}$ with $a_{N}>0$.
%is the reference wave we use at level $N$.
% for Scheme \ref{xinsuanfa}.
Moreover, there exists $j_{0}\in\{1, \ldots, d\}$
such that $\frac{2^{-N}}{\lambda_{j_{0},N}}\notin \mathbb{Z}$.
The two sets   $\Lambda_{\Omega,N}\subseteq \mathbb{Z}^{d}$ and $\Xi_{\Omega,N}\subseteq \mathbb{R}^{3d}$  are as  in  Scheme \ref{xinsuanfa}
such that for  any $(\emph{\textbf{k}}, \emph{\textbf{k}}', \emph{\textbf{k}}'')\in \Xi_{\Omega,N}$
it holds that
$\emph{\textbf{k}}\in \Lambda_{\Omega,N}$
and
\begin{align}\label{bvxcv}\emph{\textbf{k}}'=\emph{\textbf{k}}+\emph{\textbf{e}}_{j_{0}},
\emph{\textbf{k}}''=\emph{\textbf{k}}-\emph{\textbf{e}}_{j_{0}}, \end{align}
 where  $\emph{\textbf{e}}_{j_{0}}=(0, \ldots, 0, 1, 0, \ldots, 0)$
with the $j_{0}$th-coordinate being  $1$ and other coordinates  being zero.
%Moreover, it is required that  Definition \ref{addefinition} (i) holds for such an arbitrary
%$(\emph{\textbf{k}}, \emph{\textbf{k}}', \emph{\textbf{k}}'')\in  \Xi_{\Omega,N}$.
%
%Moreover, for  any $(\emph{\textbf{k}}, \emph{\textbf{k}}', \emph{\textbf{k}}'')\in \Xi_{\Omega,N}$
%it is required here   that.
%, and the three values   $g_{N}(2^{-N}\emph{\textbf{k}}), g_{N}(2^{-N}\emph{\textbf{k}}')$ and
%$g_{N}(2^{-N}\emph{\textbf{k}}'')$ are distinct from each other.
Then every $g_{N}$ is $\gamma$-admissible on
$\Xi_{\Omega,N}$  w.r.t   level $N$  if  \begin{align}\label{KKK2345678}a_{N}|\sin(2^{-N}\frac{2\pi}{\lambda_{j_{0},N}})\sin^{2}(2^{-N-1}\frac{2\pi}{\lambda_{j_{0},N}})|\geq\frac{1}{2\gamma}>0.\end{align} Moreover, if \eqref{KKK2345678} holds for every level  $N$ and $\sup_{k\geq1} a_{k}<\infty$ then
 $\{g_{N}\}^{\infty}_{N=1}$
is a sequence of uniformly $\gamma$-admissible plane waves on $\{\Xi_{\Omega,N}\}^{\infty}_{N=1}$.

%Then such a sequence is uniformly $\gamma$-admissible w.r.t the $j$th coordinate direction if and only if $0<\gamma\leq\inf\{2a_{N}|\sin(2^{-N}\frac{2\pi}{\lambda_{j,N}})\sin(2^{-N-1}\frac{2\pi}{\lambda_{j,N}})|: N\geq1\}<\infty$ and $\sup\{a_{N}: N\geq1\}<\infty$.
%\mu
%If there exists $j\in \{1, 2, \ldots, D\}$ such that $\gamma:=\inf_{N\geq1}|\sin(\frac{2\pi}{2^{N}\lambda_{j,N}})|$,
%then the plane reference wave
%is admissible at level $N$ w.r.t the $j$th coordinate direction.
%In optics, a complex exponent $g(x_{1}, x_{2})=e^{\textbf{i}\frac{2\pi}{\lambda}(x_{1}+x_{2})}$
%is referred to as a plane wave (c.f. \cite{Fourieroptics}) with $\lambda>0$ mentioned as   the wavelength.
\end{theo}
\begin{proof}
%In the case of \eqref{bvxcv}, the term in \eqref{mbvc}
%associated with the plane reference wave $g_{N}$
% is calculated as follows,
% \begin{align} \label{upp1234}\begin{array}{lll}
% \displaystyle\frac{1}{|\Im[(1-e^{\textbf{i}\mathcal{K}_{N}\cdot2^{-N}(\emph{\textbf{k}}-\emph{\textbf{k}}')})(1-e^{-\textbf{i}\mathcal{K}_{N}\cdot2^{-N}(\emph{\textbf{k}}-\emph{\textbf{k}}'')})]|}\\
% \displaystyle=\frac{1}{|\Im[(1-e^{-\textbf{i}\mathcal{K}_{N}\cdot2^{-N}\emph{\textbf{e}}_{j}})(1-e^{-\textbf{i}\mathcal{K}_{N}\cdot2^{-N}\textbf{e}_{j}})]|}\\
% \displaystyle=\frac{1}{4|\sin(2^{-N}\frac{2\pi}{\lambda_{j,N}})\sin^{2}(2^{-N-1}\frac{2\pi}{\lambda_{j,N}})|}.
%\end{array}\end{align}
%holds  for arbitrary   $K\in \mathbb{Z}^{d}$,
%where $\mathcal{K}_{N}=(\frac{2\pi}{\lambda_{1,N}}, \ldots, \frac{2\pi}{\lambda_{D,N}})$.
%Then \eqref{upp123} is equivalent to
%\begin{align} \label{upp1234ui2}\begin{array}{lll}
%0<\gamma\leq2a_{N}|\sin(2^{-N}\frac{2\pi}{\lambda_{j,N}})\sin(2^{-N-1}\frac{2\pi}{\lambda_{j,N}})|<\infty.
%\end{array}\end{align}
% On the other hand, $|g_{N}(x_{1}, \ldots, x_{D})|=a_{N}$
%for every $(x_{1}, \ldots, x_{D})\in \mathbb{R}^{d}$.
Denote $\mathcal{K}_{N}:=(\frac{2\pi}{\lambda_{1,N}}, \ldots, \frac{2\pi}{\lambda_{d,N}})$.
By the  direct calculation we have $|g_{N}(2^{-N}\emph{\textbf{k}})-g_{N}(2^{-N}\emph{\textbf{k}}')|=a_{N}|1-e^{-\textbf{i}\mathcal{K}_{N}\cdot2^{-N}\emph{\textbf{e}}_{j_{0}}}|=
a_{N}|1-e^{-\textbf{i}\frac{2^{-N+1}\pi}{\lambda_{j_{0},N}}}|$,
$|g_{N}(2^{-N}\emph{\textbf{k}})-g_{N}(2^{-N}\emph{\textbf{k}}'')|=a_{N}|1-e^{\textbf{i}\frac{2^{-N+1}\pi}{\lambda_{j_{0},N}}}|$.
%and $|g_{N}(2^{-N}\emph{\textbf{k}}')-g_{N}(2^{-N}\emph{\textbf{k}}'')|=a_{N}|1-e^{-\textbf{i}\frac{2^{-N+2}\pi}{\lambda_{j_{0},N}}}|$.
Since $\frac{2^{-N}}{\lambda_{j_{0},N}}\notin \mathbb{Z}$, Definition \ref{addefinition} (i) holds.
Moreover,
\begin{align}\label{1234}\begin{array}{lll}
\displaystyle\frac{\max\{|g_{N}(2^{-N}\emph{\textbf{k}})-g_{N}(2^{-N}\emph{\textbf{k}}')|, |\overline{g_{N}}(2^{-N}\emph{\textbf{k}})-\overline{g_{N}}(2^{-N}\emph{\textbf{k}}'')|\}}{|\Im[(g_{N}(2^{-N}\emph{\textbf{k}})-g_{N}(2^{-N}\emph{\textbf{k}}^{'}))(\overline{g_{N}}(2^{-N}\emph{\textbf{k}})-\overline{g_{N}}(2^{-N}\emph{\textbf{k}}^{''}))]|}\\
\displaystyle=\frac{\max\{|1-e^{\textbf{i}\mathcal{K}_{N}\cdot2^{-N}(\emph{\textbf{k}}-\emph{\textbf{k}}')}|, |1-e^{-\textbf{i}\mathcal{K}_{N}\cdot2^{-N}(\emph{\textbf{k}}-\emph{\textbf{k}}'')}|\}}{a_{N}|\Im[(1-e^{\textbf{i}\mathcal{K}_{N}\cdot2^{-N}(\emph{\textbf{k}}-\emph{\textbf{k}}')})(1-e^{-\textbf{i}\mathcal{K}_{N}\cdot2^{-N}(\emph{\textbf{k}}-\emph{\textbf{k}}'')})]|} \\
\displaystyle\leq
\frac{2}{a_{N}|\Im[(1-e^{\textbf{i}\mathcal{K}_{N}\cdot2^{-N}(\emph{\textbf{k}}-\emph{\textbf{k}}')})(1-e^{-\textbf{i}\mathcal{K}_{N}\cdot2^{-N}(\emph{\textbf{k}}-\emph{\textbf{k}}'')})]|}\\
\displaystyle=
\frac{1}{2a_{N}|\sin(2^{-N}\frac{2\pi}{\lambda_{j_{0},N}})\sin^{2}(2^{-N-1}\frac{2\pi}{\lambda_{j_{0},N}})|}\quad (\ref{1234}A)\\
\leq\gamma, \quad (\ref{1234}B)
\end{array}\end{align}
where $ (\ref{1234}A)$
and $(\ref{1234}B)$ are  from \eqref{bvxcv} and \eqref{KKK2345678}, respectively.
This means that  every   $g_{N}$ is $\gamma$-admissible on $\Xi_{\Omega,N}$ w.r.t  level $N$.
Note that $|g_{N}(\emph{\textbf{x}})|\equiv a_{N}$ for any $\emph{\textbf{x}}\in \mathbb{R}^{d}$.
If \eqref{KKK2345678} holds for every $N$ and $\sup_{k\geq1} a_{k}<\infty$, then
$\{g_{N}\}^{\infty}_{N=1}$
is a sequence of uniformly $\gamma$-admissible plane waves on $\{\Xi_{\Omega,N}\}^{\infty}_{N=1}$.
\end{proof}

\begin{note}\label{yilaixing}
By Definition \ref{addefinition}, the admissibility  depends on $\Xi_{\Omega,N}$.
Since  $\Xi_{\Omega,N}$ in \eqref{bvxcv} is   related to the coordinate $j_{0}$,
the admissibility of $g_{N}$ is naturally related to $j_{0}.$ Moreover, it follows from
$(\ref{1234}B)$ that the choice of $j_{0}$ probably affects  the admissibility exponent $\gamma$.
\end{note}

\subsubsection{A class of admissible spherical waves}
As a counterpart of the plane reference  wave case in subsection \ref{pmianboqingxing}, this subsection is to design a class of sequences of  uniformly admissible spherical reference waves $\{g_{N}\}^{\infty}_{N=1}$ for Scheme \ref{xinsuanfa},
where $g_{N}(\emph{\textbf{x}})=\frac{a_{N}}{\|\emph{\textbf{x}}\|_{2}}e^{\textbf{i}\nu_{N}\|\emph{\textbf{x}}\|_{2}}$.
%Recall again that the admissibility condition  is defined by  \eqref{p45tgy} through  the   set
%$\Xi_{\Omega,N}$. The following concerns on a  design of
%$\{\Xi_{\Omega,N}\}^{\infty}_{N=1}$ such that   $\{g_{N}\}^{\infty}_{N=1}$ in Scheme \ref{xinsuanfa}
%is  uniformly admissible on $\{\Xi_{\Omega,N}\}^{\infty}_{N=1}$.
%%%% ËƺõÕⲿ·ÖÓëÇ¿¶ÈÎ޹أ¬ËùÒÔcomment
%Recall that
%the  purpose  of \textbf{step $(1')$} of Scheme  \ref{xinsuanfa} is to recover the $N$-level data  $\{f(2^{-N}\emph{\textbf{k}}):
%\emph{\textbf{k}}\in \Lambda_{\Omega,N}\}$ through the intensities
%\begin{align}\label{stdax}
%I_{\Xi_{\Omega,N}}=\{I_{N, \emph{\textbf{k}}}, I_{N, \emph{\textbf{k}}^{'}}, I_{N, \emph{\textbf{k}}^{''}}
%: \emph{\textbf{k}}\in  \Lambda_{\Omega,N},
%2^{-N}\emph{\textbf{k}}^{'}\approx2^{-N}\emph{\textbf{k}}, 2^{-N}\emph{\textbf{k}}^{''}\approx2^{-N}\emph{\textbf{k}}\}.
%\end{align}
%where
%where  $(\emph{\textbf{k}}, \emph{\textbf{k}}', \emph{\textbf{k}}'')\in \Xi_{\Omega,N}$
%such that $2^{-N}\emph{\textbf{k}}'\approx 2^{-N}\emph{\textbf{k}}$ and $2^{-N}\emph{\textbf{k}}''\approx 2^{-N}\emph{\textbf{k}}$.
The first issue  is to  design   $\Xi_{\Omega,N}$ in  Definition  \ref{addefinition}.
Since  $g_{N}$ is not  defined at $\textbf{0}$,
$\Xi_{\Omega,N}$ designed  via \eqref{bvxcv} for the plane wave case  is not necessarily well-defined for the spherical wave case.
Note that  $g_{N}=\frac{a_{N}}{\|\emph{\textbf{x}}\|_{2}}e^{\textbf{i}\nu_{N}\|\emph{\textbf{x}}\|_{2}}$
is essentially a radial function on $\mathbb{R}^d\setminus\{\textbf{0}\}.$
Motivated by this, for any $\emph{\textbf{k}}\in  \Lambda_{\Omega,N}$ we  set $\emph{\textbf{k}}'=(1+2^{-N})\emph{\textbf{k}}$
and $\emph{\textbf{k}}''=(1-2^{-N})\emph{\textbf{k}}$ such that
\begin{align}\label{lingwaijihehewenjie}
\Xi_{\Omega,N}=\{(\emph{\textbf{k}}, \emph{\textbf{k}}', \emph{\textbf{k}}''): \emph{\textbf{k}}\in  \Lambda_{\Omega,N},
\emph{\textbf{k}}'=(1+2^{-N})\emph{\textbf{k}}, \emph{\textbf{k}}''=(1-2^{-N})\emph{\textbf{k}}\}.
\end{align}
%and the set  of intensities in \eqref{stdax}
%is chosen  as
%\begin{align}\label{nengliangqiumianbo}\begin{array}{lllllll} I_{\Xi_{\Omega,N}}=\{I_{N, \emph{\textbf{k}}}, I_{N, \emph{\textbf{k}}'},
%I_{N, \emph{\textbf{k}}''}: \emph{\textbf{k}}\in  \Lambda_{\Omega,N},
%\emph{\textbf{k}}'=(1+2^{-N})\emph{\textbf{k}},
%\emph{\textbf{k}}''=(1-2^{-N})\emph{\textbf{k}}\}.
%\end{array}\end{align}
%\begin{align}\label{nengliangqiumianbo}\begin{array}{lll}
%I_{\Xi_{\Omega,N}}&=\Big\{|f(2^{-N}\emph{\textbf{k}})+g(2^{-N}\emph{\textbf{k}})|^2, |f(2^{-N}(1+2^{-N})\emph{\textbf{k}})+g(2^{-N}(1+2^{-N})\emph{\textbf{k}})|^2,\\
%&\quad |f(2^{-N}(1-2^{-N})\emph{\textbf{k}})+g(2^{-N}(1-2^{-N})\emph{\textbf{k}})|^2:
%\emph{\textbf{k}}\in  \Lambda_{\Omega,N}\Big\}.
%\end{array}\end{align}
%in  Scheme \ref{xinsuanfa}.
Clearly, by the definition of a spherical wave and the definitions of the  intensities  $I_{N, \emph{\textbf{k}}}, I_{N, \emph{\textbf{k}}^{'}}, I_{N, \emph{\textbf{k}}^{''}}$ in \eqref{intensity123} and  \eqref{1BXCZSDF},   it is required that $\emph{\textbf{k}}, \emph{\textbf{k}}', \emph{\textbf{k}}''\neq \textbf{0}$.
If $\textbf{0}\notin\Lambda_{\Omega,N}$ then $\Xi_{\Omega,N}$ designed  in  \eqref{lingwaijihehewenjie} naturally  satisfies such a requirement.
Moreover,
the following implies  that the analysis of the admissibility of  $g_{N}$ can be simplified by  such a design  of  $\Xi_{\Omega,N}$.

 \begin{prop}\label{KKKDFG}
 As previously, suppose that the spherical reference wave $g_{N}(\emph{\textbf{x}})=\frac{a_{N}}{\|\emph{\textbf{x}}\|_{2}}e^{\textbf{i}\nu_{N}\|\emph{\textbf{x}}\|_{2}}$
 such that $a_{N}>0.$ The ROI is denoted by   $\Omega$ such that the associated
 $\Lambda_{\Omega,N}\subseteq \mathbb{Z}^{d}$ defined in \eqref{fft1} does not contain $\textbf{0}$ for every level $N$.
 Specifically,
  \begin{align}\label{fft89}\begin{array}{lll}
 \Lambda_{\Omega,N}=\{(k_{1},\ldots,  k_{d})\in \mathbb{Z}^{d}: 2^{N}L_{l,\min}-M_{l}\leq k_{l}\leq2^{N}L_{l,\max}, l=1, \ldots, d\},
\end{array}\end{align}
where every $M_{l}>0$,  $L_{l,\min}$ and $L_{l,\max}$ are related with $\Omega$ via \eqref{zuixiao}.
Associated with $\Lambda_{\Omega,N}$ the set $\Xi_{\Omega,N}\subseteq \mathbb{R}^{3d}$ is defined  in \eqref{lingwaijihehewenjie}.
Then    for any $(\emph{\textbf{k}}, \emph{\textbf{k}}', \emph{\textbf{k}}'')\in \Xi_{\Omega,N}$
 we have $g_{N}(2^{-N}\emph{\textbf{k}})\neq g_{N}(2^{-N}\emph{\textbf{k}}')$
 and $g_{N}(2^{-N}\emph{\textbf{k}})\neq g_{N}(2^{-N}\emph{\textbf{k}}'')$.
Moreover,
 %Choose $\Xi_{\Omega,N}$ as in \eqref{lingwaijihehewenjie}.
%for any $(\emph{\textbf{k}}, \emph{\textbf{k}}', \emph{\textbf{k}}'')\in \Xi_{\Omega,N}$  it holds that
 \begin{align}\label{123p45tgy}\begin{array}{lll}
\displaystyle\frac{\max\{|g_{N}(2^{-N}\emph{\textbf{k}})-g_{N}(2^{-N}\emph{\textbf{k}}')|, |\overline{g_{N}}(2^{-N}\emph{\textbf{k}})-\overline{g_{N}}(2^{-N}\emph{\textbf{k}}'')|\}}{|\Im[(g_{N}(2^{-N}\emph{\textbf{k}})-g_{N}(2^{-N}\emph{\textbf{k}}^{'}))(\overline{g_{N}}(2^{-N}\emph{\textbf{k}})-\overline{g_{N}}(2^{-N}\emph{\textbf{k}}^{''}))]|}\\
\displaystyle\leq
\frac{9\sqrt{d}(2^{N}\|\Omega\|_{2,\sup}+M)}{ 2^{N+3}a_{N}|\sin(2^{-2N}\nu_{N}\|\emph{\textbf{k}}\|_{2})|\sin^{2}(2^{-2N-1}\nu_{N}\|\emph{\textbf{k}}\|_{2})},
\end{array}
\end{align}
where
$\|\Omega\|_{2,\sup}=\sup\{\|\emph{\textbf{x}}\|_{2}: \emph{\textbf{x}}\in \Omega\}$
and $M=\max\{M_{1}, \ldots, M_{d}\}$.
 \end{prop}

\begin{proof}
For any $(\emph{\textbf{k}}, \emph{\textbf{k}}', \emph{\textbf{k}}'')\in \Xi_{\Omega,N}$,
it follows from  \eqref{lingwaijihehewenjie} and $\textbf{0}\notin\Lambda_{\Omega, N}$  that    $\|\emph{\textbf{k}}\|_{2}>0.$
Clearly, $|g_{N}(2^{-N}\emph{\textbf{k}}')|=\frac{1}{1+2^{-N}}|g_{N}(2^{-N}\emph{\textbf{k}})|\neq0$
and $|g_{N}(2^{-N}\emph{\textbf{k}}'')|=\frac{1}{1-2^{-N}}|g_{N}(2^{-N}\emph{\textbf{k}})|\neq0$.
Then the first part is  true.

 By the  direct calculation we have
\begin{align} \label{upp12ff34}\begin{array}{lll}
|\Im[(g_{N}(2^{-N}\emph{\textbf{k}})-g_{N}(2^{-N}\emph{\textbf{k}}^{'}))(\overline{g_{N}}(2^{-N}\emph{\textbf{k}})-\overline{g_{N}}(2^{-N}\emph{\textbf{k}}^{''}))]|\\
=|\Im\big[\big(g_{N}(2^{-N}\emph{\textbf{k}})-g_{N}(2^{-N}(1+2^{-N})\emph{\textbf{k}})\big)\big(\overline{g_{N}}(2^{-N}\emph{\textbf{k}})-\overline{g_{N}}(2^{-N}(1-2^{-N})\emph{\textbf{k}})\big)\big]|\\
=\frac{2^{2N}a^{2}_{N}}{\|\emph{\textbf{k}}\|_{2}^{2}}|\Im\big[(1-\frac{1}{1+2^{-N}}e^{\textbf{i}2^{-2N}\nu_{N}\|\emph{\textbf{k}}\|_{2}})(1-\frac{1}{1-2^{-N}}e^{\textbf{i}2^{-2N}\nu_{N}\|\emph{\textbf{k}}\|_{2}})\big]|\\
=\frac{2^{2N}a^{2}_{N}}{\|\emph{\textbf{k}}\|_{2}^{2}}|[-\frac{2\sin(2^{-2N}\nu_{N}\|\emph{\textbf{k}}\|_{2})}{(1-2^{-N})(1+2^{-N})}+\frac{\sin(2^{-2N+1}\nu_{N}\|\emph{\textbf{k}}\|_{2})}{(1-2^{-N})(1+2^{-N})}]|\\
=\frac{2^{2N+2}a^{2}_{N}}{\|\emph{\textbf{k}}\|_{2}^{2}}\frac{1}{(1-2^{-N})(1+2^{-N})}|\sin(2^{-2N}\nu_{N}\|\emph{\textbf{k}}\|_{2})|\sin^{2}(2^{-2N-1}\nu_{N}\|\emph{\textbf{k}}\|_{2})\\
\geq \frac{2^{2N+3}a^{2}_{N}}{3\|\emph{\textbf{k}}\|_{2}^{2}}|\sin(2^{-2N}\nu_{N}\|\emph{\textbf{k}}\|_{2})|\sin^{2}(2^{-2N-1}\nu_{N}\|\emph{\textbf{k}}\|_{2}), \quad (\ref{upp12ff34}A)
\end{array}\end{align}
where $(\ref{upp12ff34}A)$ is from $0<2^{-N}\leq1/2$ for $N\geq1$.
On the other hand,
 \begin{align} \label{hhhhb0}\begin{array}{lll}
 |g_{N}(2^{-N}\emph{\textbf{k}})-g_{N}(2^{-N}\emph{\textbf{k}}^{'})|\\
=|g_{N}(2^{-N}\emph{\textbf{k}})-g_{N}(2^{-N}(1+2^{-N})\emph{\textbf{k}}))|\\
=\frac{2^{N}a_{N}}{\|\emph{\textbf{k}}\|_{2}}|1-\frac{1}{1+2^{-N}}e^{\textbf{i}2^{-2N}\nu_{N}\|\emph{\textbf{k}}\|_{2}}|\\
\leq\frac{2^{N+1}a_{N}}{\|\emph{\textbf{k}}\|_{2}}
\end{array}\end{align}
and
 \begin{align} \label{hhhhb1}\begin{array}{lll}
 |\overline{g_{N}}(2^{-N}\emph{\textbf{k}})-\overline{g_{N}}(2^{-N}\emph{\textbf{k}}^{''})|\\
=|\overline{g_{N}}(2^{-N}\emph{\textbf{k}})-\overline{g_{N}}(2^{-N}(1-2^{-N})\emph{\textbf{k}}))|\\
=\frac{2^{N}a_{N}}{\|\emph{\textbf{k}}\|_{2}}|1-\frac{1}{1-2^{-N}}e^{\textbf{i}2^{-2N}\nu_{N}\|\emph{\textbf{k}}\|_{2}}|\\
\leq\frac{2^{N}3a_{N}}{\|\emph{\textbf{k}}\|_{2}}.
\end{array}\end{align}
Then it follows from \eqref{hhhhb0} and \eqref{hhhhb1} that
\begin{align}\label{liangliang}
\max\{|g_{N}(2^{-N}\emph{\textbf{k}})-g_{N}(2^{-N}\emph{\textbf{k}}')|, |\overline{g_{N}}(2^{-N}\emph{\textbf{k}})-\overline{g_{N}}(2^{-N}\emph{\textbf{k}}'')|\}\leq\frac{2^{N}3a_{N}}{\|\emph{\textbf{k}}\|_{2}}.
\end{align}
Combing \eqref{upp12ff34} and \eqref{liangliang} we have
\begin{align}\label{123p45gbvx}\begin{array}{lll}
\displaystyle\frac{\max\{|g_{N}(2^{-N}\emph{\textbf{k}})-g_{N}(2^{-N}\emph{\textbf{k}}')|, |\overline{g_{N}}(2^{-N}\emph{\textbf{k}})-\overline{g_{N}}(2^{-N}\emph{\textbf{k}}'')|\}}{|\Im[(g_{N}(2^{-N}\emph{\textbf{k}})-g_{N}(2^{-N}\emph{\textbf{k}}^{'}))(\overline{g_{N}}(2^{-N}\emph{\textbf{k}})-\overline{g_{N}}(2^{-N}\emph{\textbf{k}}^{''}))]|}\\
\displaystyle\leq\frac{\frac{2^{N}3a_{N}}{\|\emph{\textbf{k}}\|_{2}}}{\frac{2^{2N+3}a^{2}_{N}}{3\|\emph{\textbf{k}}\|_{2}^{2}}|\sin(2^{-2N}\nu_{N}\|\emph{\textbf{k}}\|_{2})|\sin^{2}(2^{-2N-1}\nu_{N}\|\emph{\textbf{k}}\|_{2})}\\
\displaystyle=\frac{9\|\emph{\textbf{k}}\|_{2}}{ 2^{N+3}a_{N}|\sin(2^{-2N}\nu_{N}\|\emph{\textbf{k}}\|_{2})|\sin^{2}(2^{-2N-1}\nu_{N}\|\emph{\textbf{k}}\|_{2})}.
%\leq\displaystyle\frac{\|\Omega\|_{2,\sup}}{3\times 2^{N+1}a_{N}|\sin(2^{-2N}\nu_{N}\|\emph{\textbf{k}}\|_{2})|\sin^{2}(2^{-2N-1}\nu_{N}\|\emph{\textbf{k}}\|_{2})}
\end{array}\end{align}
On the other hand,
it follows from \eqref{fft89} that
\begin{align}\begin{array}{lll}
k_{l}&\leq 2^{N}\max\{|L_{l,\max}|, |L_{l,\min}|\}+M_{l}\\
&\leq2^{N}\|\Omega\|_{2,\sup}+M.
\end{array}\end{align}
%Here,  as in \eqref{zuidaqujianbianyuan}   $\textcolor[rgb]{1.00,0.00,0.00}{M=\max\{M_{1}, \ldots, M_{d}\}}$.
Then \begin{align}\label{qingyise}\|\emph{\textbf{k}}\|_{2}\leq\sqrt{d}(2^{N}\|\Omega\|_{2,\sup}+M).\end{align}
Combining \eqref{qingyise} and \eqref{123p45gbvx} we have
\begin{align}\label{67y23p45tgy}\begin{array}{lll}
\displaystyle\frac{\max\{|g_{N}(2^{-N}\emph{\textbf{k}})-g_{N}(2^{-N}\emph{\textbf{k}}')|, |\overline{g_{N}}(2^{-N}\emph{\textbf{k}})-\overline{g_{N}}(2^{-N}\emph{\textbf{k}}'')|\}}{|\Im[(g_{N}(2^{-N}\emph{\textbf{k}})-g_{N}(2^{-N}\emph{\textbf{k}}^{'}))(\overline{g_{N}}(2^{-N}\emph{\textbf{k}})-\overline{g_{N}}(2^{-N}\emph{\textbf{k}}^{''}))]|}\\
\displaystyle\leq
\frac{9\sqrt{d}(2^{N}\|\Omega\|_{2,\sup}+M)}{ 2^{N+3}a_{N}|\sin(2^{-2N}\nu_{N}\|\emph{\textbf{k}}\|_{2})|\sin^{2}(2^{-2N-1}\nu_{N}\|\emph{\textbf{k}}\|_{2})}.
\end{array}
\end{align}
This completes the proof.
\end{proof}
Based on Proposition \ref{KKKDFG}, we next design a class of sequences of uniformly  admissible waves.

\begin{theo}\label{P456789}
Associated with the ROI $\Omega$, for every $N\geq1$  the two  sets  $\Lambda_{\Omega, N}\subseteq \mathbb{R}^{d}$
and $\Xi_{\Omega,N}\subseteq \mathbb{R}^{3d}$ are  as  in  \eqref{fft89} and \eqref{lingwaijihehewenjie}, respectively. Particularly,
 \begin{align}\label{234fft89}\begin{array}{lll}
 \Lambda_{\Omega,N}=\{(k_{1},\ldots,  k_{d})\in \mathbb{Z}^{d}: 2^{N}L_{l,\min}-M_{l}\leq k_{l}\leq2^{N}L_{l,\max}, l=1, \ldots, d\}
\end{array}\end{align}
satisfies  $\textbf{0}\notin \Lambda_{\Omega,N}.$
%Consequently, it follows from Proposition \ref{dbdn} that there exists $l_{0}\in \{1, \ldots, d\}$ such that
%$\|2^{-N}\emph{\textbf{k}}\|_{2}\geq $
%for every $\emph{\textbf{k}}\in\Lambda_{\Omega,N}$.
%
%and \textcolor[rgb]{1.00,0.00,0.00}{$\textbf{0}\notin\Lambda_{\Omega,N}$}, and for any $(\emph{\textbf{k}}, \emph{\textbf{k}}', \emph{\textbf{k}}'')
%\in \Xi_{\Omega,N}$
%the three values   $g_{N}(2^{-N}\emph{\textbf{k}}), g_{N}(2^{-N}\emph{\textbf{k}}')$ and
%$g_{N}(2^{-N}\emph{\textbf{k}}'')$ \textcolor[rgb]{1.00,0.00,0.00}{are different from each other}.
%
%
%As previously, suppose  that  the refinable function $\phi$
%satisfies $\hbox{supp}(\phi)\subseteq [0, M_{1}]\times \cdots\times [0, M_{d}]$ and
%denote $\textcolor[rgb]{1.00,0.00,0.00}{M=\max\{M_{1}, \ldots, M_{d}\}}$.
Suppose that the wave number and amplitude sequences  $\{\nu_{N}: N\geq1\}$ and $\{a_{N}: N\geq1\}$ satisfies
the following three  items:\\
$ \hbox{(i)} \ 0<\alpha\leq2^{-N}|\nu_{N}|q_{l_{0}}\leq
2^{-2N}|\nu_{N}|\sqrt{d}(2^{N}\|\Omega\|_{2,\sup}+M)\leq\frac{\pi}{2},$\\
$\hbox{(ii)} \
\frac{9\sqrt{d}(2^{N}\|\Omega\|_{2,\sup}+M)}{ 2^{N+3}a_{N}}\leq\beta<0,
$
\\
$\hbox{(iii)} \sup_{N\geq1}a_{N}<\infty$,
\\
where the constants  $\alpha$ and $\beta$ are independent of $N$,  $q_{l_{0}}=\min\{|L_{l_{0},\min}-2^{-1}M_{l_{0}}|, |L_{l_{0},\max}|\}$
$>0$,
$\|\Omega\|_{2,\sup}=\sup\{\|\emph{\textbf{x}}\|_{2}: \emph{\textbf{x}}\in \Omega\}$
and $M=\max\{M_{1}, \ldots, M_{d}\}$.
Then  the  sequence of spherical waves  $\{\frac{a_{N}}{\|\emph{\textbf{x}}\|_{2}}e^{\textbf{i}\nu_{N}\|\emph{\textbf{x}}\|_{2}}\}^{\infty}_{N=1}$
is uniformly $\frac{\beta}{\sin{\alpha}\sin^{2}{(\alpha/2)}}$-admissible on $\{\Xi_{\Omega,N}\}^{\infty}_{N=1}$.
\end{theo}
\begin{proof}
For any $(\emph{\textbf{k}}, \emph{\textbf{k}}', \emph{\textbf{k}}'')\in
\Xi_{\Omega,N}$, it follows from
Proposition \ref{KKKDFG} that Definition \ref{addefinition} (i) holds.
It follows from Proposition \ref{dbdn} and  \eqref{qingyise}
 that $0<q_{l_{0}} \leq 2^{-N}\sqrt{d}(2^{N}\|\Omega\|_{2,\sup}+M)$.
 Therefore, item (i) makes sense.
%From this, we have $q_{l_{0}}\leq $
%
%We first prove that Definition \ref{addefinition} (i) holds.
Now  it follows from \eqref{qingyise} and item (i) that
 \begin{align}\label{KBVCXZZZ}0<\alpha\leq2^{-N}|\nu_{N}|q_{l_{0}}\leq 2^{-N}|\nu_{N}|\|2^{-N}\emph{\textbf{k}}\|_{2}
\leq 2^{-2N}|\nu_{N}|\sqrt{d}(2^{N}\|\Omega\|_{2,\sup}+M)\leq \frac{\pi}{2},\end{align}
where $\emph{\textbf{k}}\in \Lambda_{\Omega, N}$. Then
\begin{align} \label{huupp12ff34}\begin{array}{lll}
|\Im[(g_{N}(2^{-N}\emph{\textbf{k}})-g_{N}(2^{-N}\emph{\textbf{k}}^{'}))(\overline{g_{N}}(2^{-N}\emph{\textbf{k}})-\overline{g_{N}}(2^{-N}\emph{\textbf{k}}^{''}))]|\\
%=|\Im\big[\big(g_{N}(2^{-N}\emph{\textbf{k}})-g_{N}(2^{-N}(1+2^{-N})\emph{\textbf{k}})\big)\big(\overline{g_{N}}(2^{-N}\emph{\textbf{k}})-\overline{g_{N}}(2^{-N}(1-2^{-N})\emph{\textbf{k}})\big)\big]|\\
%=\frac{2^{2N}a^{2}_{N}}{\|\emph{\textbf{k}}\|_{2}^{2}}|\Im\big[(1-\frac{1}{1+2^{-N}}e^{\textbf{i}2^{-2N}\nu_{N}\|\emph{\textbf{k}}\|_{2}})(1-\frac{1}{1-2^{-N}}e^{\textbf{i}2^{-2N}\nu_{N}\|\emph{\textbf{k}}\|_{2}})\big]|\\
%=\frac{2^{2N}a^{2}_{N}}{\|\emph{\textbf{k}}\|_{2}^{2}}|[-\frac{2\sin(2^{-2N}\nu_{N}\|\emph{\textbf{k}}\|_{2})}{(1-2^{-N})(1+2^{-N})}+\frac{\sin(2^{-2N+1}\nu_{N}\|\emph{\textbf{k}}\|_{2})}{(1-2^{-N})(1+2^{-N})}]|\\
%=\frac{2^{2N+2}a^{2}_{N}}{\|\emph{\textbf{k}}\|_{2}^{2}}\frac{1}{(1-2^{-N})(1+2^{-N})}|\sin(2^{-2N}\nu_{N}\|\emph{\textbf{k}}\|_{2})|\sin^{2}(2^{-2N-1}\nu_{N}\|\emph{\textbf{k}}\|_{2})\\
\geq \frac{2^{2N+3}a^{2}_{N}}{3\|\emph{\textbf{k}}\|_{2}^{2}}|\sin(2^{-N}\nu_{N}\|2^{-N}\emph{\textbf{k}}\|_{2})|\sin^{2}(2^{-N-1}\nu_{N}\|2^{-N}\emph{\textbf{k}}\|_{2}) \quad (\ref{huupp12ff34}A)\\
\geq\frac{2^{2N+3}a^{2}_{N}}{3\|\emph{\textbf{k}}\|_{2}^{2}}\sin{\alpha}\sin^{2}{(\alpha/2)} \quad (\ref{huupp12ff34}B)\\
>0, \quad (\ref{huupp12ff34}C)
\end{array}\end{align}
where $(\ref{huupp12ff34}A)$, $(\ref{huupp12ff34}B)$ and $(\ref{huupp12ff34}C)$ are  from  $(\ref{upp12ff34}A)$, \eqref{KBVCXZZZ}
and $\emph{\textbf{k}}\neq0,$ respectively.
%Then $g_{N}(2^{-N}\emph{\textbf{k}})\neq g_{N}(2^{-N}\emph{\textbf{k}}^{'})$.
%Similarly, we can prove that $g_{N}(2^{-N}\emph{\textbf{k}})\neq g_{N}(2^{-N}\emph{\textbf{k}}^{''})$.
%That is,
%For any $(\emph{\textbf{k}}, \emph{\textbf{k}}', \emph{\textbf{k}}'')\in \Xi_{\Omega,N}$, it follows from
% \eqref{qingyise} and (i) that
%\begin{align}\label{turg}
%|\sin(2^{-2N}\nu_{N}\|\emph{\textbf{k}}\|_{2})|\sin^{2}(2^{-2N-1}\nu_{N}\|\emph{\textbf{k}}\|_{2})\geq\sin{\alpha}\sin^{2}{(\alpha/2)}.
%\end{align}
Consequently,
\begin{align}\label{1267y23p45tgy}\begin{array}{lll}
\displaystyle\frac{\max\{|g_{N}(2^{-N}\emph{\textbf{k}})-g_{N}(2^{-N}\emph{\textbf{k}}')|, |\overline{g_{N}}(2^{-N}\emph{\textbf{k}})-\overline{g_{N}}(2^{-N}\emph{\textbf{k}}'')|\}}{|\Im[(g_{N}(2^{-N}\emph{\textbf{k}})-g_{N}(2^{-N}\emph{\textbf{k}}^{'}))(\overline{g_{N}}(2^{-N}\emph{\textbf{k}})-\overline{g_{N}}(2^{-N}\emph{\textbf{k}}^{''}))]|}\\
\displaystyle\leq
\frac{9\sqrt{d}(2^{N}\|\Omega\|_{2,\sup}+M)}{ 2^{N+3}a_{N}|\sin(2^{-2N}\nu_{N}\|\emph{\textbf{k}}\|_{2})|\sin^{2}(2^{-2N-1}\nu_{N}\|\emph{\textbf{k}}\|_{2})}\quad (\ref{1267y23p45tgy} A)\\
\displaystyle\leq \frac{\beta}{\sin{\alpha}\sin^{2}{(\alpha/2)}}, \quad (\ref{1267y23p45tgy} B)
\end{array}
\end{align}
where $(\ref{1267y23p45tgy} A)$ is  from \eqref{123p45tgy},  and  $(\ref{1267y23p45tgy} B)$ is from \eqref{huupp12ff34} and item  (ii).
Now the proof is completed by $(\ref{1267y23p45tgy} B)$ and item  (iii).
\end{proof}

\begin{exam}\label{KSDF}
In Theorem \ref{P456789}, if  $\nu_{N}$ and $a_{N}$ are chosen as $2^{N}\epsilon$ ($\epsilon>0$) and $1$
then item (iii) holds, and  the terms  $2^{-N}\nu_{N}q_{l_{0}}=\epsilon q_{l_{0}}$ and $\frac{9\sqrt{d}(2^{N}\|\Omega\|_{2,\sup}+M)}{ 2^{N+3}a_{N}}=\frac{9}{8}\sqrt{d}(
\|\Omega\|_{2,\sup}+2^{-N}M)\leq \frac{9}{8}\sqrt{d}(
\|\Omega\|_{2,\sup}+2^{-1}M):=\beta$. Choose appropriate $\epsilon$ and $\alpha$
such that $0<\alpha\leq\epsilon q_{l_{0}}\leq\epsilon\sqrt{d}(
\|\Omega\|_{2,\sup}+2^{-1}M)\leq\pi/2.$ Then it follows from Theorem \ref{P456789}
that $\{\frac{e^{\textbf{i}\tiny{2^{N}}\epsilon\|\emph{\textbf{x}}\|_{2}}}{\|\emph{\textbf{x}}\|_{2}}\}^{\infty}_{N=1}$
is uniformly $\frac{\beta}{\sin{\alpha}\sin^{2}{(\alpha/2)}}$-admissible on $\{\Xi_{\Omega,N}\}^{\infty}_{N=1}$.
\end{exam}

\section{Single-shot interference based     phase retrieval for functions in Sobolev space}\label{zqsd}

\subsection{The first main  result: the  plane reference wave method  for Scheme \ref{xinsuanfa}}\label{planewave}
Based on
 the single-shot  interference  intensities
 %$I_{\Xi_{\Omega,N}}$
   in \eqref{yongnaxienengliang}, we will provide    Approach \ref{ppp4} for  recovering the  $N$-level  data $\{f(2^{-N}\emph{\textbf{k}}): \emph{\textbf{k}}\in \Lambda_{\Omega, N}\}$
in Scheme \ref{xinsuanfa} \textbf{step $(1')$}.
%We next give a choice of the reference wave $g_{N}$,   the set $\Xi_{\Omega,N}$ and
%the associated  intensities $I_{\Xi_{\Omega,N}}$ for Scheme \ref{xinsuanfa}.
Herein the   interference wave $g_{N}$ for Scheme \ref{xinsuanfa}   is chosen as  a  plane reference wave.
As previously,
 \begin{align}\label{12fft6789}\begin{array}{lll}
 \Lambda_{\Omega,N}=\{(k_{1},\ldots,  k_{d})\in \mathbb{Z}^{d}: 2^{N}L_{l,\min}-M_{l}\leq k_{l}\leq2^{N}L_{l,\max}, l=1, \ldots, d\}.
\end{array}\end{align}
 Moreover,  associated with $\Lambda_{\Omega,N}$
the set $\Xi_{\Omega,N}\subseteq \mathbb{R}^{3d}$ is designed in    \eqref{bvxcv}. Particularly,
\begin{align}\label{xdclingwaijihe}
\Xi_{\Omega,N}=\{(\emph{\textbf{k}}, \emph{\textbf{k}}', \emph{\textbf{k}}''): \emph{\textbf{k}}\in  \Lambda_{\Omega,N},
\emph{\textbf{k}}'=\emph{\textbf{k}}+\emph{\textbf{e}}_{j_{0}},
\emph{\textbf{k}}''=\emph{\textbf{k}}-\emph{\textbf{e}}_{j_{0}}\}
\end{align}
where $j_{0}\in\{1, \ldots, d\}$ is fixed.
Correspondingly,
the single-shot  interference  intensity set  $I_{\Xi_{\Omega,N}}$ in   \eqref{yongnaxienengliang} is
\begin{align}\label{kkJJJ}\begin{array}{lllllll} I_{\Xi_{\Omega,N}}=\big\{I_{N, \emph{\textbf{k}}}=|f(2^{-N}\emph{\textbf{k}})+g_{N}(2^{-N}\emph{\textbf{k}})|^{2}, I_{N, \emph{\textbf{k}}'}=|f(2^{-N}\emph{\textbf{k}}')+g_{N}(2^{-N}\emph{\textbf{k}}')|^{2},\\
\quad\quad\quad\quad  I_{N, \emph{\textbf{k}}''}=|f(2^{-N}\emph{\textbf{k}}'')+g_{N}(2^{-N}\emph{\textbf{k}}'')|^{2}
: \emph{\textbf{k}}\in  \Lambda_{\Omega,N},
\emph{\textbf{k}}'=\emph{\textbf{k}}+\emph{\textbf{e}}_{j_{0}},
\emph{\textbf{k}}''=\emph{\textbf{k}}-\emph{\textbf{e}}_{j_{0}}\big\}.
\end{array}\end{align}

%If $f$ sits in Sobolev space,  the recovery error will be estimated   in Theorem \ref{approach1error}.

\begin{appr}\label{ppp4}
\textbf{Input}:
%a refinable  function  $\phi$ such that $\hbox{supp}(\phi)\subseteq
%$\mathcal{M}=[0, M_{1}]\times\cdots\times  [0, M_{D}]$,
level $N$, ROI $\Omega$,
the two sets
$ \Lambda_{\Omega,N}$ and $ \Xi_{\Omega,N}$   in \eqref{12fft6789} and \eqref{xdclingwaijihe},
a  plane  reference  wave $g_{N}$ being admissible on $ \Xi_{\Omega,N}$  w.r.t level   $N$,
%ROI $\Omega\subseteq \mathbb{R}^{d}$,,
single-shot interference intensity set
$I_{\Xi_{\Omega,N}}$  in \eqref{kkJJJ}.

%(Or $\{I_{N,K}=|F(2^{-N}K)+G(2^{-N}K)|: K\In \Mathring{\Lambda}_{\Phi,\Omega,N}\}$).

\textbf{Step 1}: For any $(\emph{\textbf{k}}, \emph{\textbf{k}}', \emph{\textbf{k}}'')\in \Xi_{\Omega,N}$
such that $\emph{\textbf{k}}\in  \Lambda_{\Omega,N}$, compute $c=\Re(g_{N}(2^{-N}\emph{\textbf{k}})-g_{N}(2^{-N}\emph{\textbf{k}}')),$
$d=\Im(g_{N}(2^{-N}\emph{\textbf{k}})-g_{N}(2^{-N}\emph{\textbf{k}}'))$,  $h=\Re(g_{N}(2^{-N}\emph{\textbf{k}})-g_{N}(2^{-N}\emph{\textbf{k}}''))$
and $e=\Im(g_{N}(2^{-N}\emph{\textbf{k}})-g_{N}(2^{-N}\emph{\textbf{k}}''))$, where $\emph{\textbf{k}}'=\emph{\textbf{k}}+\emph{\textbf{e}}_{j_{0}}$ and $ \emph{\textbf{k}}''=\emph{\textbf{k}}-\emph{\textbf{e}}_{j_{0}}$
with $j_{0}\in \{1, \ldots, d\}$ being as in \eqref{xdclingwaijihe}.

\textbf{Step 2}: Compute
\begin{gather}\label{qiujie255}
\left[\begin{array}{cccccccccc}\mathring{\Re}(f(2^{-N}\emph{\textbf{k}}))\\
\mathring{\Im}(f(2^{-N}\emph{\textbf{k}}))
\end{array}\right]=\frac{1}{2\mu_{N;\emph{\textbf{k}}, \emph{\textbf{k}}', \emph{\textbf{k}}''}}\left[\begin{array}{cccccccccc}
 e &-d \\-h &c\end{array}\right]\left[\begin{array}{cccccccccc}I_{N,\emph{\textbf{k}}}-I_{N, \emph{\textbf{k}}'}\\
I_{N,\emph{\textbf{k}}}-I_{N, \emph{\textbf{k}}''}\end{array}\right],
\end{gather}
where
\begin{align}\label{cvxxcba1}\mu_{N;\emph{\textbf{k}}, \emph{\textbf{k}}', \emph{\textbf{k}}''}=-\Im[(g_{N}(2^{-N}\emph{\textbf{k}})-g_{N}(2^{-N}\emph{\textbf{k}}^{'}))(\overline{g_{N}}(2^{-N}\emph{\textbf{k}})-\overline{g_{N}}(2^{-N}\emph{\textbf{k}}^{''}))].\end{align}

\textbf{Output}: $\mathring{f}(2^{-N}\emph{\textbf{k}}):=\mathring{\Re}(f(2^{-N}\emph{\textbf{k}}))+\textbf{i}\mathring{\Im}(f(2^{-N}\emph{\textbf{k}}))$.
\end{appr}

\begin{note} (1)
Since $g_{N}$ is  admissible on $ \Xi_{\Omega,N}$ w.r.t level  $N$,
by Proposition \ref{rongxuxing} we have $\mu_{N;\emph{\textbf{k}}, \emph{\textbf{k}}', \emph{\textbf{k}}''}\neq0$.
(2) Formula \eqref{qiujie255} is derived from \eqref{BVC1234}
and that  the amplitude of the plane wave  $g_{N}$ being  a constant.
\end{note}

\subsubsection{\textbf{Recovery error of Approach} \ref{ppp4}}
%It has been  stated in subsection \ref{positivemd} that any   $f\in H^{s}(\mathbb{R}^{d})$ with $s>d/2$
%is  continuous. The following  establishes a bound for $\sup_{x\in \mathbb{R}^{d}}|f(x)|$ and the Lipschitz continuity of $f$. They
%will be helpful for establishing the error of Approach \ref{ppp4}.
For any level $N$ and $(\emph{\textbf{k}}, \emph{\textbf{k}}', \emph{\textbf{k}}'')\in  \Xi_{\Omega,N}$
such that $\emph{\textbf{k}}'=\emph{\textbf{k}}+\emph{\textbf{e}}_{j_{0}}$ and $ \emph{\textbf{k}}''=\emph{\textbf{k}}-\emph{\textbf{e}}_{j_{0}}$ where
$j_{0}\in \{1, \ldots, d\}$ is fixed,
 the quasi-interference intensities $A_{N,\emph{\textbf{k}}, \emph{\textbf{k}}'}$
  and $A_{N,\emph{\textbf{k}}, \emph{\textbf{k}}''}$ are defined as in (\ref{intensity1234}A) and (\ref{intensity1234}B),
  namely,
  \begin{align}\label{fangshenengliang} A_{N,\emph{\textbf{k}}, \emph{\textbf{k}}'}=|f(2^{-N}\emph{\textbf{k}})+g_{N}(2^{-N}\emph{\textbf{k}}')|^{2}, A_{N,\emph{\textbf{k}}, \emph{\textbf{k}}''}=|f(2^{-N}\emph{\textbf{k}})+g_{N}(2^{-N}\emph{\textbf{k}}'')|^{2}.\end{align}
%where $\Lambda_{\Omega,N}$,   $K_{2}=K_{1}+\textbf{e}_{j}$ and $K_{3}=K_{1}-\textbf{e}_{j}$
%are as  in Approach \ref{ppp4}.
Note that $2^{-N}\emph{\textbf{k}}'\approx2^{-N}\emph{\textbf{k}}$
and $2^{-N}\emph{\textbf{k}}''\approx2^{-N}\emph{\textbf{k}}$. Motivated by this,
in what follows we use the interference  intensities $I_{N,\emph{\textbf{k}}'}=|f(2^{-N}\emph{\textbf{k}}')+g_{N}(2^{-N}\emph{\textbf{k}}')|^{2}$
and $I_{N,\emph{\textbf{k}}''}=|f(2^{-N}\emph{\textbf{k}}'')+g_{N}(2^{-N}\emph{\textbf{k}}'')|^{2}$ to approximate $A_{N,\emph{\textbf{k}}, \emph{\textbf{k}}'}$
and $A_{N,\emph{\textbf{k}}, \emph{\textbf{k}}''},$
respectively.
%
%introduce the interference  intensity-based  approximation to quasi-interference intensity, one of the key ingredients for our
%single-shot PR.
%In what follows, we address the approximation to $A^{2}_{N,K_{1},K_{2}}$
%and $A^{2}_{N,K_{1},K_{3}}$ by  $I^{2}_{N,K_{2}}$ and $I^{2}_{N,K_{3}}$ in \eqref{measurements}, respectively.

\begin{theo}\label{miyuzhong}
Suppose that the ROI $\Omega\subseteq \mathbb{R}^d$,  the set $ \Xi_{\Omega,N}\subseteq \mathbb{R}^{3d}$
and  the plane wave $g_{N}$ are  as in Approach \ref{ppp4} such that $g_{N}$ is admissible on $ \Xi_{\Omega,N}$ w.r.t level   $N$. Then
there exists a constant  $C_{1}(s,
\varsigma)>0$ (being independent of $g_{N}$) such that for any   $f\in H^{s}(\mathbb{R}^{d})$  satisfying $\nu_{2}(f)>\varsigma>s>d/2$,
 the quasi-interference intensities $A_{N,\emph{\textbf{k}},\emph{\textbf{k}}'}$ and $A_{N,\emph{\textbf{k}},\emph{\textbf{k}}''}$ in \eqref{fangshenengliang} can be   approximated by $I_{N,\emph{\textbf{k}}'}$ and $I_{N,\emph{\textbf{k}}''}$ as follows,
\begin{align}\label{qiujie7gh15}\begin{array}{lllllll}
\Big\|\left[\begin{array}{cccccccccc}A_{N,\emph{\textbf{k}},\emph{\textbf{k}}'}-I_{N,\emph{\textbf{k}}'} \\
A_{N,\emph{\textbf{k}},\emph{\textbf{k}}''}-I_{N,\emph{\textbf{k}}''}\end{array}\right]\Big\|_{2}
\leq C_{1}(s,
\varsigma)\|f\|_{H^{\varsigma}(\mathbb{R}^{d})}\big(\|f\|_{H^{\varsigma}(\mathbb{R}^{d})}+\sup_{\emph{\textbf{x}}\in \Omega}|g_{N}(\emph{\textbf{x}})|\big)2^{-(N+1)\zeta},
\end{array}\end{align}
where  $\zeta=\min\{1, \varsigma-s\}$.
\end{theo}
\begin{proof}
The square root $\sqrt{A_{N,\emph{\textbf{k}}, \emph{\textbf{k}}'}}$
is estimated as follows,
\begin{align}\label{Cjian}
\begin{array}{lll} \sqrt{A_{N,\emph{\textbf{k}}, \emph{\textbf{k}}'}}&=|f(2^{-N}\emph{\textbf{k}}')+g_{N}(2^{-N}\emph{\textbf{k}}')+f(2^{-N}\emph{\textbf{k}})-f(2^{-N}\emph{\textbf{k}}')|\\
&\leq \sqrt{I_{N,\emph{\textbf{k}}'}}+|f(2^{-N}\emph{\textbf{k}})-f(2^{-N}\emph{\textbf{k}}')|\\
&\leq \sqrt{I_{N,\emph{\textbf{k}}'}}+\widehat{C}(s,\varsigma)2^{1-\zeta}\|f\|_{H^{\varsigma}(\mathbb{R}^{d})}2^{-N\zeta}, \quad
(\ref{Cjian} A)
%&:=I_{N,K_{2}}+\textcolor[rgb]{1.00,0.00,0.00}{\widehat{C}(s, \varsigma)}\|f\|_{H^{\varsigma}(\mathbb{R}^{d})}2^{-(N+1)\zeta}.
\end{array}\end{align}
where $(\ref{Cjian} A)$  with $\widehat{C}(s,
\varsigma)\geq1$ is derived from Proposition  \ref{1stlemma} \eqref{ghgh1}
and $\emph{\textbf{k}}'=\emph{\textbf{k}}+\emph{\textbf{e}}_{j_{0}}$.
Similarly, by the triangle inequality and  \eqref{ghgh1} we have  $\sqrt{A_{N,\emph{\textbf{k}},\emph{\textbf{k}}'}}\geq \sqrt{I_{N,\emph{\textbf{k}}'}}-\widehat{C}(s,\varsigma)2^{1-\zeta}\|f\|_{H^{\varsigma}(\mathbb{R}^{d})}2^{-N\zeta}.$
From this and \eqref{Cjian} we have
\begin{align}\label{908} \big|\sqrt{A_{N,\emph{\textbf{k}},\emph{\textbf{k}}'}}-\sqrt{I_{N,\emph{\textbf{k}}'}}\big|\leq \widehat{C}(s,\varsigma)2^{1-\zeta}\|f\|_{H^{\varsigma}(\mathbb{R}^{d})}2^{-N\zeta}.\end{align}
Through the similar procedures as above  we have
\begin{align}\label{909} \big|\sqrt{A_{N,\emph{\textbf{k}},\emph{\textbf{k}}''}}-\sqrt{I_{N,\emph{\textbf{k}}''}}\big|\leq \widehat{C}(s,\varsigma)2^{1-\zeta}\|f\|_{H^{\varsigma}(\mathbb{R}^{d})}2^{-N\zeta}.\end{align}
Then
\begin{align}\label{qiujie715}\begin{array}{lllllll}
\Big\|\left[\begin{array}{cccccccccc}A_{N,\emph{\textbf{k}},\emph{\textbf{k}}'}-I_{N,\emph{\textbf{k}}'} \\
A_{N,\emph{\textbf{k}},\emph{\textbf{k}}''}-I_{N,\emph{\textbf{k}}''}\end{array}\right]\Big\|_{2}\\
=[((\sqrt{A_{N,\emph{\textbf{k}},\emph{\textbf{k}}'}})^2-(\sqrt{I_{N,\emph{\textbf{k}}'}})^2)^{2}+((\sqrt{A_{N,\emph{\textbf{k}},\emph{\textbf{k}}''}})^2-(\sqrt{I_{N,\emph{\textbf{k}}''}})^2)^{2}]^{1/2}\\
\leq \widehat{C}(s,\varsigma)2^{1-\zeta}\|f\|_{H^{\varsigma}(\mathbb{R}^{d})}2^{-N\zeta}\big((\sqrt{A_{N,\emph{\textbf{k}},\emph{\textbf{k}}'}}+\sqrt{I_{N,\emph{\textbf{k}}'}})^{2}+
(\sqrt{A_{N,\emph{\textbf{k}},\emph{\textbf{k}}''}}+\sqrt{I_{N,\emph{\textbf{k}}''}})^{2}\big)^{1/2}\quad (\ref{qiujie715} A)\\
%&\leq2C_{f}2^{-N(1+\delta)}\big(\sum^{3}_{n=2}(2I_{N,K_{n}}+\epsilon)^{2}\big)^{1/2}\\
\leq 2^{5/2}\|f\|_{H^{\varsigma}(\mathbb{R}^{d})}\widehat{C}(s,
\varsigma)\big(
\widehat{C}(s,
\varsigma)\|f\|_{H^{\varsigma}(\mathbb{R}^{d})}+\sup_{\emph{\textbf{x}}\in \Omega}|g_{N}(\emph{\textbf{x}})|\big)2^{-(N+1)\zeta}, \quad (\ref{qiujie715} B)
\end{array}\end{align}
where $(\ref{qiujie715} A)$  is derived from \eqref{908} and  \eqref{909},  and $(\ref{qiujie715} B)$ is   from  \eqref{ghgh}. Since $\widehat{C}(s,
\varsigma)\geq1$,
the proof is  completed  by   choosing $C_{1}(s,
\varsigma)=2^{5/2}\widehat{C}^{2}(s,
\varsigma)<\infty$.
\end{proof}
%%%%% ÒÔÏÂÊÇÉÏÊöÖ¤Ã÷¹ý³Ì
%\begin{align}\begin{array}{lllllll}
%\sqrt{A_{N,\emph{\textbf{k}}, \emph{\textbf{k}}'}}&=|f(2^{-N}\emph{\textbf{k}})+g_{N}(2^{-N}\emph{\textbf{k}}')|\leq\sup_{x\in \mathbb{R}^d}|f(x)|+
%\sup_{x\in \Omega}|g_{N}(x)|\\
%&\leq\widehat{C}(s,
%\varsigma)\|f\|_{H^{\varsigma}(\mathbb{R}^{d})}+\sup_{x\in \Omega}|g_{N}(x)|\\
%\sqrt{I_{N,\emph{\textbf{k}}}'}&=|f(2^{-N}\emph{\textbf{k}})+g_{N}(2^{-N}\emph{\textbf{k}})|\leq\sup_{x\in \mathbb{R}^d}|f(x)|+
%\sup_{x\in \Omega}|g_{N}(x)|\\
%&\leq\widehat{C}(s,
%\varsigma)\|f\|_{H^{\varsigma}(\mathbb{R}^{d})}+\sup_{x\in \Omega}|g_{N}(x)|.
%\end{array}
%\end{align}
%
%$$\{2[2(\widehat{C}(s,
%\varsigma)\|f\|_{H^{\varsigma}(\mathbb{R}^{d})}+\sup_{x\in \Omega}|g_{N}(x)|)]^{2}\}^{1/2}=2^{3/2}$$

Now   it is ready to establish the recovery error for Approach \ref{ppp4}.

\begin{theo}\label{approach1error}
Suppose that $\Omega\subseteq \mathbb{R}^d$ is the ROI and  the two sets
$ \Lambda_{\Omega,N}\subseteq \mathbb{Z}^d$ and $ \Xi_{\Omega,N}\subseteq\mathbb{R}^{3d}$ are  as  in \eqref{12fft6789} and \eqref{xdclingwaijihe}, respectively. Moreover,
the plane wave $g_{N}$    is $\gamma$-admissible on $ \Xi_{\Omega,N}$ w.r.t level   $N$.
Then for any   $f\in H^{s}(\mathbb{R}^{d})$  satisfying $\nu_{2}(f)>\varsigma>s>d/2$ and for  any $\emph{\textbf{k}}\in  \Lambda_{\Omega,N}$,  there holds
\begin{align}\label{qiujie4edf35}\begin{array}{lllllllll}
|f(2^{-N}\emph{\textbf{k}})-\mathring{f}(2^{-N}\emph{\textbf{k}})|
\leq \gamma C_{1}(s,
\varsigma)
\|f\|_{H^{\varsigma}(\mathbb{R}^{d})}\big(\|f\|_{H^{\varsigma}(\mathbb{R}^{d})}+\sup_{\emph{\textbf{x}}\in \Omega}|g_{N}(\emph{\textbf{x}})|\big)
2^{-(N+1)\zeta},\end{array}\end{align}
where $\mathring{f}(2^{-N}\emph{\textbf{k}})$ is the output of Approach \ref{ppp4},   the   constant $C_{1}(s,
\varsigma)$ is as Theorem  \ref{miyuzhong} such that  it  is independent of both  $f$ and $g_{N}$,
and   $\zeta=\min\{1, \varsigma-s\}$.
\end{theo}

\begin{proof}
For any $(\emph{\textbf{k}}, \emph{\textbf{k}}', \emph{\textbf{k}}'')\in \Xi_{\Omega,N}$,
since $g_{N}$    is $\gamma$-admissible on $ \Xi_{\Omega,N}$ w.r.t level   $N$ then
%we  consider   the  phaseless  equation system in \eqref{intensity8765} w.r.t $x$, namely,
%\begin{align} \label{up24}
%\left\{\begin{array}{lll}
%|x+g_{N}(2^{-N}\emph{\textbf{k}})|=I_{N,\emph{\textbf{k}}}, \\
%|x+g_{N}(2^{-N}\emph{\textbf{k}}')|=A_{N,\emph{\textbf{k}},\emph{\textbf{k}}'},\\
%|x+g_{N}(2^{-N}\emph{\textbf{k}}'')|=A_{N,\emph{\textbf{k}},\emph{\textbf{k}}''},
%\end{array}\right.
%\end{align}
%where  $A_{N,\emph{\textbf{k}},\emph{\textbf{k}}'}, A_{N,\emph{\textbf{k}},\emph{\textbf{k}}}$ and $I_{N,\emph{\textbf{k}}}$
%are as in \eqref{fangshenengliang} and \eqref{kkJJJ}, respectively.
%Clearly, $f(2^{-N}\emph{\textbf{k}})$ is a solution to \eqref{up24}. Since $g_{N}$ is  admissible at level  $N$,
%it follows from    that
\begin{align}\label{123cucaoguji}\begin{array}{lllllll}
|f(2^{-N}\emph{\textbf{k}})-\mathring{f}(2^{-N}\emph{\textbf{k}})|\\
=\Big\|\left[\begin{array}{cccccccccc}\Re(f(2^{-N}\emph{\textbf{k}}))\\
\Im(f(2^{-N}\emph{\textbf{k}}))
\end{array}\right]-\left[\begin{array}{cccccccccc}\mathring{\Re}(f(2^{-N}\emph{\textbf{k}}))\\
\mathring{\Im}(f(2^{-N}\emph{\textbf{k}}))
\end{array}\right]\Big\|_{2}\\
\leq\frac{\max\{|g_{N}(2^{-N}\emph{\textbf{k}})-g_{N}(2^{-N}\emph{\textbf{k}}')|, |\overline{g_{N}}(2^{-N}\emph{\textbf{k}})-\overline{g_{N}}(2^{-N}\emph{\textbf{k}}'')|\}}{|\Im[(g_{N}(2^{-N}\emph{\textbf{k}})-g_{N}(2^{-N}\emph{\textbf{k}}^{'}))(\overline{g_{N}}(2^{-N}\emph{\textbf{k}})-\overline{g_{N}}(2^{-N}\emph{\textbf{k}}^{''}))]|}
\\
\quad \times\sqrt{(I_{N,\emph{\textbf{k}}'}-A_{N,\emph{\textbf{k}},\emph{\textbf{k}}'})^{2}+(I_{N,\emph{\textbf{k}}''}-A_{N,\emph{\textbf{k}},\emph{\textbf{k}}''})^{2}}\quad (\ref{123cucaoguji}A)\\
\leq\gamma
\sqrt{(I_{N,\emph{\textbf{k}}'}-A_{N,\emph{\textbf{k}},\emph{\textbf{k}}'})^{2}+(I_{N,\emph{\textbf{k}}''}-A_{N,\emph{\textbf{k}},\emph{\textbf{k}}''})^{2}} \quad (\ref{123cucaoguji}B)\\
\leq\gamma C_{1}(s,
\varsigma)\|f\|_{H^{\varsigma}(\mathbb{R}^{d})}\big(\|f\|_{H^{\varsigma}(\mathbb{R}^{d})}+\sup_{\emph{\textbf{x}}\in \Omega}|g_{N}(\emph{\textbf{x}})|\big)2^{-(N+1)\zeta}, \quad (\ref{123cucaoguji}C)
\end{array}
\end{align}
where $(\ref{123cucaoguji}A)$, $(\ref{123cucaoguji}B)$ and $(\ref{123cucaoguji}C)$ are derived from Theorem \ref{chubuguji} \eqref{cucaoguji},
Definition \ref{addefinition} \eqref{p45tgy} (with $\gamma_{N}=\gamma$) and Theorem \ref{miyuzhong}, respectively.
This completes the proof.
\end{proof}

\subsubsection{\textbf{Recovery of functions in Sobolev space by single-shot interference intensity: plane wave method}}
Now based on Theorem \ref{approach1error},   it is  ready to establish our first  main result as follows.

\begin{theo}\label{firstmaintheorem}
Suppose that $\phi\in H^{s}(\mathbb{R}^{d})$   is a nonnegative refinable  function such that $\nu_{2}(\phi)>\varsigma>s>d/2$,  $\hbox{supp}(\phi)\subseteq \mathcal{M}=[0, M_{1}]\times \cdots\times [0, M_{d}]$, and its mask symbol has
$\kappa+1$  sum rules such that $\kappa+1>\varsigma$. As previously, $\Omega\subseteq \mathbb{R}^{d}$
is the bounded   ROI. Associated with $\Omega$ and $\mathcal{M}$
the set $\Xi_{\Omega,N}\subseteq \mathbb{R}^{3d}$ is defined  in \eqref{xdclingwaijihe}
such that for any  $(\emph{\textbf{k}}, \emph{\textbf{k}}', \emph{\textbf{k}}'')\in  \Xi_{\Omega,N}$
there  holds  $\emph{\textbf{k}}'=\emph{\textbf{k}}+\emph{\textbf{e}}_{j_{0}}$ and $ \emph{\textbf{k}}''=\emph{\textbf{k}}-\emph{\textbf{e}}_{j_{0}}$, where $j_{0}\in\{1, \ldots, d\}$
is fixed.
Additionally,   $\{g_{N}\}^{\infty}_{N=1}$ is
a sequence of   uniformly $\gamma$-admissible plane reference waves on  $\{\Xi_{\Omega,N}\}^{\infty}_{N=1}$.
%Consequently, every    reference wave $g_{N}$ is $\gamma$-admissible on  $\Xi_{\Omega,N}$
%w.r.t  level $N$. ,
%$C (s,
%\varsigma)$ is as in  \eqref{bound67890}
Then  there exist      $C (s,
\varsigma), C_{1}(s,
\varsigma, \Omega)>0$  such that  for  any   $f\in H^{s}(\mathbb{R}^{d})$
satisfying  $\nu_{2}(f)>\varsigma>s$,
% such  that it is $(1+\delta)$-Lipschitz continuous  (i.e. $|f(x_{1})-f(x_{2})|\leq C_{f}\|x_{1}-x_{2}\|^{1+\delta}_{2}$, $\delta>0, \forall x_{1},
%x_{2}\in \mathbb{R}^{d}$) and $\nu_{2}(f)\geq\varsigma$
its restriction
on $\Omega$ can be approximated by $\big(\sum_{\emph{\textbf{k}}\in  \Lambda_{\Omega,N}}\mathring{f}(2^{-N}\emph{\textbf{k}})\phi(2^{N}\cdot-\emph{\textbf{k}})\big)|_{\Omega}$
as follows,
\begin{align}\label{90876}\begin{array}{lll}
\Big\|f|_{\Omega}-\big(\sum_{\emph{\textbf{k}}\in  \Lambda_{\Omega,N}}\mathring{f}(2^{-N}\emph{\textbf{k}})\phi(2^{N}\cdot-\emph{\textbf{k}})\big)|_{\Omega}\Big\|_{L^2(\mathbb{R}^d)}\\
\leq
\|f\|_{H^{\varsigma}(\mathbb{R}^{d})}\Big[C(s,
\varsigma) 2^{-N(\varsigma-s)/2}+\gamma C_{1}(s,
\varsigma, \Omega)\big(\|f\|_{H^{\varsigma}(\mathbb{R}^{d})}+\sup_{l\geq1}\sup_{\emph{\textbf{x}}\in \Omega}|g_{l}(\emph{\textbf{x}})|\big)
2^{-(N+1)\zeta}\Big],
\end{array}\end{align}
where the level $N\geq1$ is arbitrary,  the data set  $\{\mathring{f}(2^{-N}\emph{\textbf{k}}): \emph{\textbf{k}}\in  \Lambda_{\Omega,N}\}$ is derived from   Approach  \ref{ppp4} and
$\zeta=\min\{1, \varsigma-s\}$.
%and
% \begin{align}\label{cconstant} \mathcal{C}_{\Omega, s, \varsigma,f}:=(C(s,\varsigma)+2^{\frac{7}{2}}C_{f}\frac{\lambda}{2\pi}((1+C(s,\varsigma))))\|f\|_{H^{\varsigma}(\mathbb{R}^{d})}+
%2^{\frac{7}{2}}\frac{\lambda}{2\pi}C_{f}(1+C_{f})\hbox{Vol}^{1/2}(\Omega)<\infty.\end{align}
\end{theo}
\begin{proof}
We first estimate
\begin{align}\label{8765}\begin{array}{llllll}
 \|\big(\sum_{\emph{\textbf{k}}\in\mathbb{Z}^{d}}\mathring{f}(2^{-N}\emph{\textbf{k}})\phi(2^{N}\cdot-\emph{\textbf{k}})-f\big)\big|_{\Omega}\|_{L^2(\mathbb{R}^d)}\\
 =\|\big(\sum_{\emph{\textbf{k}}\in \Lambda_{\Omega,N}}\mathring{f}(2^{-N}\emph{\textbf{k}})\phi(2^{N}\cdot-\emph{\textbf{k}})-f\big)\big|_{\Omega}\|_{L^2(\mathbb{R}^d)} \quad (\ref{8765} A)\\
 \leq \|\big(\sum_{\emph{\textbf{k}}\in \Lambda_{\Omega,N}}f(2^{-N}\emph{\textbf{k}})\phi(2^{N}\cdot-\emph{\textbf{k}})-f\big)\big|_{\Omega}\|_{L^2(\mathbb{R}^d)}\\
\ +\|\big(\sum_{\emph{\textbf{k}}\in \Lambda_{\Omega,N}}(f(2^{-N}\emph{\textbf{k}})-\mathring{f}(2^{-N}\emph{\textbf{k}}))\phi(2^{N}\cdot-\emph{\textbf{k}})\big)\big|_{\Omega}\|_{L^2(\mathbb{R}^d)}\\
 \leq \|\sum_{\emph{\textbf{k}}\in \mathbb{Z}^d}f(2^{-N}\emph{\textbf{k}})\phi(2^{N}\cdot-\emph{\textbf{k}})-f\|_{L^2(\mathbb{R}^d)}\quad (\ref{8765} B)\\
\ +\|\big(\sum_{\emph{\textbf{k}}\in \Lambda_{\Omega,N}}(f(2^{-N}\emph{\textbf{k}})-\mathring{f}(2^{-N}\emph{\textbf{k}}))\phi(2^{N}\cdot-\emph{\textbf{k}})\big)\big|_{\Omega}\|_{L^2(\mathbb{R}^d)}, \quad (\ref{8765} C)\end{array}\end{align}
where $(\ref{8765} A)$ is derived from \eqref{roi}.
On the other hand, by $\kappa+1>\varsigma$ it follows from Proposition  \ref{caiyangapproximation}
and $\|\cdot\|_{L^2(\mathbb{R}^{d})}\leq\|\cdot\|_{H^s(\mathbb{R}^d)}$ that  there exists a  constant $C(s,
\varsigma)$ being independent of $f$
such that $(\ref{8765} B)$ is estimated as follows,
\begin{align}\begin{array}{lll} \label{bound67890123}
\displaystyle
\big\|\sum_{\emph{\textbf{k}}\in \mathbb{Z}^d}f(2^{-N}\emph{\textbf{k}})\phi(2^{N}\cdot-\emph{\textbf{k}})-f\big\|_{L^2(\mathbb{R}^d)} \leq C(s,
\varsigma)\|f\|_{H^{\varsigma}(\mathbb{R}^{d})} 2^{-N(\varsigma-s)/2}.
\end{array}
\end{align}
Now we   estimate $(\ref{8765} C)$ as follows,
\begin{align}\begin{array}{llllll}\label{yyy345}
 \|\big(\sum_{\emph{\textbf{k}}\in\Lambda_{\Omega,N}}(f(2^{-N}\emph{\textbf{k}})-\mathring{f}(2^{-N}\emph{\textbf{k}}))\phi(2^{N}\cdot-\emph{\textbf{k}})\big)\big|_{\Omega}\|_{L^2(\mathbb{R}^d)}\\
 \leq\|\big(\sum_{\emph{\textbf{k}}\in\Lambda_{\Omega,N}}|f(2^{-N}\emph{\textbf{k}})-\mathring{f}(2^{-N}\emph{\textbf{k}})|\phi(2^{N}\cdot-\emph{\textbf{k}})\big)\big|_{\Omega}\|_{L^2(\mathbb{R}^d)} \quad  (\ref{yyy345} \hbox{A})\\
 \leq\sup_{\emph{\textbf{k}}\in\Lambda_{\Omega,N}}\{|f(2^{-N}\emph{\textbf{k}})-\mathring{f}(2^{-N}\emph{\textbf{k}})|\}\|\big(\sum_{\emph{\textbf{k}}\in\Lambda_{\Omega,N}}\phi(2^{N}\cdot-\emph{\textbf{k}})\big)\big|_{\Omega}\|_{L^2(\mathbb{R}^d)} \\
 %=\|\big(\sum_{K\in \mathbb{Z}^D}|f(2^{-N}K)-\mathring{f}(2^{-N}K)|\phi(2^{N}\cdot-K)\big)_{\Omega}\|_{2}\\
 %\leq2^{\frac{7}{2}}C_{f}\frac{\lambda}{2\pi}[(1+C(s,\varsigma))\|f\|_{H^{s}(\mathbb{R}^{d})}+1+C_{f}]2^{-N(\delta+1/2)}\\
 \leq \gamma C_{1}(s,
\varsigma)
\|f\|_{H^{\varsigma}(\mathbb{R}^{d})}\big(\|f\|_{H^{\varsigma}(\mathbb{R}^{d})}+\sup_{l}\sup_{\emph{\textbf{x}}\in \Omega}|g_{l}(\emph{\textbf{x}})|\big)
\hbox{Vol}^{1/2}(\Omega)2^{-(N+1)\zeta}, \quad  (\ref{yyy345} \hbox{B})
 %2^{\frac{7}{2}}C_{f}\frac{\lambda}{2\pi}[(1+C(s,\varsigma))\|f\|_{H^{\varsigma}(\mathbb{R}^{d})}+1+C_{f}]\hbox{Vol}^{1/2}(\Omega)2^{-N\delta},
 \end{array}\end{align}
where the triangle inequality and  nonnegativity   of  $\phi$
are   used in (\ref{yyy345}A), and (\ref{yyy345}\hbox{B}) is derived from \eqref{qiujie4edf35} and  the unit decomposition  property  \eqref{yydnew} of $\phi$.
Combining \eqref{8765}, \eqref{bound67890123} and \eqref{yyy345} we have
\begin{align}\begin{array}{llllll}
\|\big(\sum_{\emph{\textbf{k}}\in\mathbb{Z}^{d}}\mathring{f}(2^{-N}\emph{\textbf{k}})\phi(2^{N}\cdot-\emph{\textbf{k}})-f\big)\big|_{\Omega}\|_{L^2(\mathbb{R}^d)}\\
\leq C(s,
\varsigma)\|f\|_{H^{\varsigma}(\mathbb{R}^{d})} 2^{-N(\varsigma-s)/2}\\
+\gamma C_{1}(s,
\varsigma)
\|f\|_{H^{\varsigma}(\mathbb{R}^{d})}\big(\|f\|_{H^{\varsigma}(\mathbb{R}^{d})}+\sup_{l\geq1}\sup_{\emph{\textbf{x}}\in \Omega}|g_{l}(\emph{\textbf{x}})|\big)
\hbox{Vol}^{1/2}(\Omega)2^{-(N+1)\zeta}.
\end{array}\end{align}
Now the proof can be concluded by  choosing $C_{1}(s,
\varsigma, \Omega)=C_{1}(s, \varsigma)\hbox{Vol}^{1/2}(\Omega)<\infty$.
%This completes the proof.
\end{proof}

\begin{rem}\label{herit}(1)
It follows from Theorem  \ref{rongxutiaojianpingm}   that there are many sequences  of plane  reference waves such that
the corresponding  approximation in Theorem \ref{firstmaintheorem} is stable as the level $N$ tends to $\infty$.
A typical choice  is   \begin{align}\label{KJKKK}\{g_{N}(x_{1},\ldots, x_{d})\}^{\infty}_{N=1}=\{e^{\textbf{i}2^{N-2}\vartheta\pi(x_{1}+\cdots+x_{d})}\}^{\infty}_{N=1}\end{align}
 with   $\vartheta\approx1$ such that the term $\sup_{l}\sup_{\emph{\textbf{x}}\in \Omega}|g_{l}(\emph{\textbf{x}})|$
in \eqref{90876} equals to $1$ and Definition \ref{addefinition} (i) holds. For every $g_{N}$,  by \eqref{KKK2345678}  we have
its   admissibility exponent  $\gamma\approx2$. (2)
%We next address the a nest property of the intensity  set
%$I_{\Xi_{\Omega,N}}$ in \eqref{kkJJJ}.
By \eqref{fft12345}, $\Lambda_{\Omega,N}\subseteq\Lambda_{\Omega,N+1}$. Then it follows from
\eqref{xdclingwaijihe} that $\Xi_{\Omega,N}\subseteq\Xi_{\Omega, N+1}.$
Consequently, besides the above mentioned  stability the sequence in \eqref{KJKKK} of plane reference waves
% we can also design
%$\{g_{N}\}^{\infty}_{N=1}$ such that Approach  \ref{ppp4}
also enjoys the  inheritance:
$ I_{\Xi_{\Omega,N}}
\subseteq I_{\Xi_{\Omega,N+1}}.
$
%with $\epsilon$ sufficiently small.
Such a  nested property facilitates  updating the  intensity  set from the  lower level to the  higher level.
(3) By Note \ref{yilaixing}, the choice of coordinate  $j_{0}$ probably affects the admissibility $\gamma$
and consequently affects the error estimate in \eqref{90876}.
\end{rem}

\begin{note}
The construction of the required  refinable function $\phi$ in Theorem \ref{firstmaintheorem}
has been addressed in section \ref{gouzaojiaxihanshu}.
\end{note}

%\begin{rem}
%%(1)
%%By \cite[section 1.1]{LiSCI}, any function   $f\in H^{s}(\mathbb{R}^{d})$ with $s>D/2$ is  continuous.
%%Since $\Omega\subseteq \mathbb{R}^{d}$
%%is a bounded  ROI, then $f$ is uniformly continuous on
%%$\Omega$. From this aspect, the requirement in Theorem \ref{firstmaintheorem} that $f$ being $(1+\delta)$-Lipschitz continuous
%%is weak.
%Theorem \ref{firstmaintheorem} \eqref{90876} implies that the interference intensity-based phase retrieval
%can be achieved without involving the common  ambiguities (such as  a global phase, conjugation, reflection et.al) which are
%commonly unavoidable   for    the traditional phase retrieval (essentially conducted by   the non-interference intensities).
%\end{rem}

%%%%b  άÊýÔÖÄÑ

\begin{rem}
 For a refinable function $\phi$, its integer shifts commonly constitute a Riesz basis,
%  (c.f.\cite[(4.4.6)]{waveletBOOK1}),
  namely, there exist constants
 $0<C_{1}\leq C_{2}$ such that for any square integrable sequence $\{c_{\emph{\textbf{k}}}\}_{\emph{\textbf{k}}\in \mathbb{Z}^d}$
 it holds that
 \begin{align}\label{beiyudushu}
 2^{-dN/2}C_{1}\big(\sum_{\emph{\textbf{k}}\in \mathbb{Z}^d}|c_{\emph{\textbf{k}}}|^{2}\big)^{1/2}\leq\Big\|\sum_{\emph{\textbf{k}}\in \mathbb{Z}^d}c_{\emph{\textbf{k}}}\phi(2^{N}\cdot-\emph{\textbf{k}})\Big\|_{L^2(\mathbb{R}^d)}\leq2^{-dN/2}C_{2}\big(\sum_{\emph{\textbf{k}}\in \mathbb{Z}^d}|c_{\emph{\textbf{k}}}|^{2}\big)^{1/2}.
 \end{align}
 In \eqref{yyy345}, however,  we   does not use   the   Riesz basis property  to estimate
\begin{align}\label{kjing}\big\|\big(\sum_{\emph{\textbf{k}}\in\Lambda_{\Omega,N}}(f(2^{-N}\emph{\textbf{k}})-\mathring{f}(2^{-N}\emph{\textbf{k}}))\phi(2^{N}\cdot-\emph{\textbf{k}}\big)\big|_{\Omega}\big\|_{L^2(\mathbb{R}^d)}.\end{align}
 We next explain this.
 If the inequality on the right-hand side of \eqref{beiyudushu} is used to estimate \eqref{kjing} then
 \begin{align}\label{jkkjing}\begin{array}{lllll}
 \displaystyle \big\|\big(\sum_{\emph{\textbf{k}}\in\Lambda_{\Omega,N}}(f(2^{-N}\emph{\textbf{k}})-\mathring{f}(2^{-N}\emph{\textbf{k}}))\phi(2^{N}\cdot-\emph{\textbf{k}}\big)\big|_{\Omega}\big\|_{L^2(\mathbb{R}^d)}\\
  \displaystyle \leq 2^{-dN/2}C_{2}\big(\sum_{\emph{\textbf{k}}\in\Lambda_{\Omega,N}}|f(2^{-N}\emph{\textbf{k}})-\mathring{f}(2^{-N}\emph{\textbf{k}})|^{2}\big)^{1/2}\\
 =\hbox{O}(2^{-dN/2}2^{dN}2^{-(N+1)\zeta})\quad (\ref{jkkjing} A)\\
 =\hbox{O}(2^{N(d/2-\zeta)}), \quad (\ref{jkkjing} B)
 \end{array}\end{align}
where we  use the cardinality  $\#\Lambda_{\Omega,N}=\hbox{O}(2^{dN})$
and \eqref{qiujie4edf35} in $(\ref{jkkjing} A)$.
On the other hand,  it follows from   Theorem \ref{firstmaintheorem}
that $\zeta=\min\{1, \varsigma-s\}$. If $d\geq2$
then the estimation in $(\ref{jkkjing} B)$
does not provide the decay information as $N\rightarrow\infty$.
%If $d\leq2$ and $\zeta=1$ then it
%does not also  provide the decay information.
This   is due to $\#\Lambda_{\Omega,N}=\hbox{O}(2^{dN})$, an exponential growth w.r.t $N$.
Instead,  (\ref{yyy345}\hbox{A}) and (\ref{yyy345}\hbox{B}) which are derived from  the nonnegativity  property of $\phi$ and
the unit decomposition  \eqref{yydnew} imply  that the corresponding estimate therein  is independent of such a cardinality  growth.
%(2) The refinable function $\phi$ in Theorem \ref{firstmaintheorem} is required to be nonnegative.
\end{rem}

\subsection{The second main  result: the spherical reference wave method}\label{qiumianbocase}
%We start with some necessary denotations. As in Theorem \ref{firstmaintheorem}, we assume that
%$\hbox{supp}(\phi)\subseteq\mathcal{M}=[0, M_{1}]\times \cdots\times [0, M_{D}]$.
%Associated with $\phi$, a ROI $\Omega\subseteq \mathbb{R}^{d}$ and the    level $N$, as in \eqref{fft1} let
% \begin{align}\label{fft23}\begin{array}{lll}
% \Lambda_{\Omega,N}
% =\{(x_{1},\ldots,  x_{D})\in \mathbb{Z}^{d}: 2^{N}L_{l,\min}-M_{l}\leq x_{l}\leq2^{N}L_{l,\max}, l=1, \ldots, D\}.
%\end{array}\end{align}
%% Throughout this section we require that $\|\Omega\|_{2,\inf}>0$.
In what follows, we establish a  recovery  approach   for the  $N$-level  data $\{f(2^{-N}\emph{\textbf{k}}): \emph{\textbf{k}}\in  \Lambda_{\Omega,N}\}$   in  Scheme \ref{xinsuanfa}  \textbf{step $(1')$} from
the single-shot interference (with a single spherical reference  wave) intensity set:
\begin{align}\label{nengliangqiumianbo123}\begin{array}{lllllll} I_{\Xi_{\Omega,N}}=\{I_{N, \emph{\textbf{k}}}, I_{N, \emph{\textbf{k}}'},
I_{N, \emph{\textbf{k}}''}: \emph{\textbf{k}}\in  \Lambda_{\Omega,N},
\emph{\textbf{k}}'=(1+2^{-N})\emph{\textbf{k}},
\emph{\textbf{k}}''=(1-2^{-N})\emph{\textbf{k}}\}.
\end{array}\end{align}
It is the counterpart of Approach \ref{ppp4} for the case of  plane reference wave.

\begin{appr}\label{ppp55}
\textbf{Input}:
%a refinable  function  $\phi$ such that $\hbox{supp}(\phi)\subseteq \mathcal{M}=
%[0, M_{1}]\times\cdots\times  [0, M_{D}]$,
level $N$, ROI $\Omega$,
the two sets
$ \Lambda_{\Omega,N}$ and $ \Xi_{\Omega,N}$  defined in \eqref{fft1}
and \eqref{lingwaijihehewenjie}
such that $\textbf{0}\notin\Lambda_{\Omega,N}$,
a  spherical  reference  wave $g_{N}(\emph{\textbf{x}})=\frac{a_{N}e^{\textbf{i}\nu_{N} \|\emph{\textbf{x}}\|_{2}}}{\|\emph{\textbf{x}}\|_{2}}$
 being admissible on $ \Xi_{\Omega,N}$ w.r.t level   $N$,
%a bounded  ROI $\Omega\subseteq \mathbb{R}^{d}$,
%the two finite  sets
%$\Lambda_{\Omega,N}$ and $\Theta_{\mathcal{M},\Omega,N}$ respectively defined in \eqref{fft1} and \eqref{jianch56},
interference intensity set $I_{\Xi_{\Omega,N}}$ in \eqref{nengliangqiumianbo123}.
%(Or $\{I_{N,K}=|F(2^{-N}K)+G(2^{-N}K)|: K\In \Mathring{\Lambda}_{\Phi,\Omega,N}\}$).

\textbf{Step 1}: For any $\emph{\textbf{k}}\in  \Lambda_{\Omega,N}$, compute $c=\Re(g_{N}(2^{-N}\emph{\textbf{k}})-g_{N}(2^{-N}\emph{\textbf{k}}')),$
$d=\Im(g_{N}(2^{-N}\emph{\textbf{k}})$
$-g_{N}(2^{-N}\emph{\textbf{k}}'))$,  $h=\Re(g_{N}(2^{-N}\emph{\textbf{k}})-g_{N}(2^{-N}\emph{\textbf{k}}'')) $ and
$e=\Im(g_{N}(2^{-N}\emph{\textbf{k}})-g_{N}(2^{-N}\emph{\textbf{k}}''))$,
 where $\emph{\textbf{k}}'=(1+2^{-N})\emph{\textbf{k}}$ and $ \emph{\textbf{k}}''=(1-2^{-N})\emph{\textbf{k}}$.

\textbf{Step 2}: Compute
\begin{gather}\label{qiujie2551}
\left[\begin{array}{cccccccccc}\mathring{\Re}(f(2^{-N}\emph{\textbf{k}})) \\
\mathring{\Im}(f(2^{-N}\emph{\textbf{k}}))\end{array}\right]=\frac{1}{2\mu_{N;\emph{\textbf{k}}, \emph{\textbf{k}}', \emph{\textbf{k}}''}}\left[\begin{array}{cccccccccc}
 e &-d \\-h &c\end{array}\right]\left[\begin{array}{cccccccccc}I_{N,\emph{\textbf{k}}}^{2}-I^{2}_{N,\emph{\textbf{k}}'}-\frac{a^{2}_{N}}{\|\emph{\textbf{k}}\|^{2}}(1-\frac{1}{(1+2^{-N})^{2}}) \\ I_{N,\emph{\textbf{k}}}^{2}-I^{2}_{N,\emph{\textbf{k}}''}+\frac{a^{2}_{N}}{\|\emph{\textbf{k}}\|^{2}}(\frac{1}{(1-2^{-N})^{2}}-1) \end{array}\right],
\end{gather}
where
\begin{align}\label{123cvxxcba1}\mu_{N;\emph{\textbf{k}}, \emph{\textbf{k}}', \emph{\textbf{k}}''}=-\Im[(g_{N}(2^{-N}\emph{\textbf{k}})-g_{N}(2^{-N}\emph{\textbf{k}}^{'}))(\overline{g_{N}}(2^{-N}\emph{\textbf{k}})-\overline{g_{N}}(2^{-N}\emph{\textbf{k}}^{''}))].\end{align}

\textbf{Output}: $\mathring{f}(2^{-N}\emph{\textbf{k}}):=\mathring{\Re}(f(2^{-N}\emph{\textbf{k}}))+\textbf{i}\mathring{\Im}(f(2^{-N}\emph{\textbf{k}}))$.
\end{appr}

\begin{note}
(1) Formula \eqref{qiujie2551} is derived from \eqref{BVC1234} and  $\emph{\textbf{k}}'=(1+2^{-N})\emph{\textbf{k}}$, $ \emph{\textbf{k}}''=(1-2^{-N})\emph{\textbf{k}}$.
(2) It follows from the admissibility of   $g_{N}$ and  Proposition \ref{rongxuxing} that  $\mu_{N;\emph{\textbf{k}}, \emph{\textbf{k}}', \emph{\textbf{k}}''}\neq0$.
\end{note}

\begin{theo}\label{secondmaintheorem}
Suppose that $\phi\in H^{s}(\mathbb{R}^{d})$   is a nonnegative refinable  function such that $\nu_{2}(\phi)>\varsigma>s>d/2$,  $\hbox{supp}(\phi)\subseteq [0, M_{1}]\times \cdots\times [0, M_{d}]$ and its mask symbol has
$\kappa+1$ sum rules such that $\kappa+1>\varsigma$. Additionally,  $\Omega\subseteq \mathbb{R}^{d}$
is a bounded  ROI and  the two sets
$ \Lambda_{\Omega,N}$ and $ \Xi_{\Omega,N}$  are as in Approach \ref{ppp55}
such that $\textbf{0}\notin\Lambda_{\Omega,N}$.
Moreover,  $\{g_{N}\}^{\infty}_{N=1}$ is
a sequence of   uniformly $\gamma$-admissible spherical reference waves on $ \{\Xi_{\Omega,N}\}^{\infty}_{N=1}$.
%a      reference wave  such that $ \frac{2^{N}\varepsilon}{\|\Omega\|_{\inf}}\leq \nu_{N}\leq
%\frac{2^{N}\cdot\frac{\pi}{2}}{\|\Omega\|_{\sup}}$ with $\varepsilon\in (0, \frac{\pi\|\Omega\|_{\inf}}{2\|\Omega\|_{\sup}}]$.
Then there exists $C(s,
\varsigma), \widehat{C}_{1}(s,\varsigma,\Omega)>0$ such that  for any   $f\in H^{s}(\mathbb{R}^{d})$ satisfying  $\nu_{2}(f)>\varsigma>s$,
% such  that it is $(1+\delta)$-Lipschitz continuous  (i.e. $|f(x_{1})-f(x_{2})|\leq C_{f}\|x_{1}-x_{2}\|^{1+\delta}_{2}$, $\delta>0, \forall x_{1},
%x_{2}\in \mathbb{R}^{d}$) and $\nu_{2}(f)\geq\varsigma$
its restriction
on $\Omega$ can be approximated by $\big(\sum_{\emph{\textbf{k}}\in  \Lambda_{\Omega,N}}\mathring{f}(2^{-N}\emph{\textbf{k}})\phi(2^{N}\cdot-\emph{\textbf{k}})\big)|_{\Omega}$
as follows,
\begin{align}\label{908716}\begin{array}{lll}
\Big\|f|_{\Omega}-\big(\sum_{\emph{\textbf{k}}\in  \Lambda_{\Omega,N}}\mathring{f}(2^{-N}\emph{\textbf{k}})\phi(2^{N}\cdot-\emph{\textbf{k}})\big)|_{\Omega}\Big\|_{L^2{(\mathbb{R}^d)}}\\
\leq
C(s,\varsigma)\|f\|_{H^{\varsigma}(\mathbb{R}^{d})} 2^{-N(\varsigma-s)/2}
+\gamma\widehat{C}_{1}(s,\varsigma,\Omega)\|f\|_{H^{\varsigma}(\mathbb{R}^{d})}\big(\|f\|_{H^{\varsigma}(\mathbb{R}^{d})}+\sup_{k\geq1}\sup_{\emph{\textbf{x}}\in \Omega}|g_{k}(\emph{\textbf{x}})|\big)2^{-N\zeta},
\end{array}\end{align}
where the level $N\geq1$,  $\zeta=\min\{1, \varsigma-s\}$,  the data set  $\{\mathring{f}(2^{-N}\emph{\textbf{k}}): \emph{\textbf{k}}\in  \Lambda_{\Omega,N}\}$ is derived from   Approach   \ref{ppp55}.
%$\|f\|_{H^{\varsigma}(\mathbb{R}^{d})}\big(\|f\|_{H^{\varsigma}(\mathbb{R}^{d})}+\sup_{l\geq1}\sup_{x\in \Omega}|g_{l}(x)|\big).$
\end{theo}

\begin{proof}
For any level $N$ and $(\emph{\textbf{k}}, \emph{\textbf{k}}', \emph{\textbf{k}}'')\in \Xi_{\Omega,N}$,
as in \eqref{fangshenengliang} the quasi-interference intensities $A_{N, \emph{\textbf{k}}, \emph{\textbf{k}}'}$ and  $A_{N, \emph{\textbf{k}}, \emph{\textbf{k}}''}$ are defined by
  \begin{align}\label{123fangshenengliang} A_{N,\emph{\textbf{k}}, \emph{\textbf{k}}'}=|f(2^{-N}\emph{\textbf{k}})+g_{N}(2^{-N}\emph{\textbf{k}}')|^{2}, A_{N,\emph{\textbf{k}}, \emph{\textbf{k}}''}=|f(2^{-N}\emph{\textbf{k}})+g_{N}(2^{-N}\emph{\textbf{k}}'')|^{2}.\end{align}
By \eqref{lingwaijihehewenjie},   $\emph{\textbf{k}}'=(1+2^{-N})\emph{\textbf{k}}$ and  $\emph{\textbf{k}}''=(1-2^{-N})\emph{\textbf{k}}$.
%
%and $K_{1}\in \Lambda_{\Omega,N}$,
%let  the
% be defined in \eqref{qiumianbodexienengliang} for the spherical case.
We next   construct  the approximation to $A_{N, \emph{\textbf{k}}, \emph{\textbf{k}}'}$ and  $A_{N, \emph{\textbf{k}}, \emph{\textbf{k}}''}$ by the interference intensities
$I_{N,\emph{\textbf{k}}'}$ and $I_{N,\emph{\textbf{k}}''}$.
For any $\emph{\textbf{k}}=(k_{1}, \ldots, k_{d})\in \Lambda_{\Omega, N}$ we have
\begin{align}\label{CVXZC}
\|2^{-N}\emph{\textbf{k}}\|_{2}\leq2^{-N}\sqrt{d}\|\emph{\textbf{k}}\|_{\infty}\underbrace{\leq}_{\tiny{\hbox{from}} \ \eqref{fft1}}2^{-N}\sqrt{d}(2^{N}\|\Omega\|_{2,\sup}+M)\leq
\sqrt{d}(\|\Omega\|_{2,\sup}+M)
\end{align}
where  $M=\max\{M_{1}, \ldots, M_{d}\}$.
We first estimate $\sqrt{A_{N,\emph{\textbf{k}}, \emph{\textbf{k}}'}}$ as follows,
\begin{align}\label{C1234jian}
\begin{array}{lll} \sqrt{A_{N,\emph{\textbf{k}}, \emph{\textbf{k}}'}}&=|f(2^{-N}\emph{\textbf{k}}')+g_{N}(2^{-N}\emph{\textbf{k}}')+f(2^{-N}\emph{\textbf{k}})-f(2^{-N}\emph{\textbf{k}}')|\\
&\leq \sqrt{I_{N,\emph{\textbf{k}}'}}+|f(2^{-N}\emph{\textbf{k}})-f(2^{-N}\emph{\textbf{k}}')|\\
&\leq \sqrt{I_{N,\emph{\textbf{k}}'}}+\widehat{C}(s,\varsigma)2^{1-\zeta}\|f\|_{H^{\varsigma}(\mathbb{R}^{d})}\|2^{-N}(\emph{\textbf{k}}-\emph{\textbf{k}}')\|_{2}^{\zeta}
\quad (\ref{C1234jian} A)\\
&= \sqrt{I_{N,\emph{\textbf{k}}'}}+\widehat{C}(s,\varsigma)2^{1-\zeta}\|f\|_{H^{\varsigma}(\mathbb{R}^{d})}\|2^{-2N}\emph{\textbf{k}}\|_{2}^{\zeta}
\\
&\leq\sqrt{I_{N,\emph{\textbf{k}}'}}+\widehat{C}(s,\varsigma)2^{1-\zeta}\|f\|_{H^{\varsigma}(\mathbb{R}^{d})}
(\sqrt{d}(\|\Omega\|_{2,\sup}+M))^{\zeta}2^{-N\zeta}, \quad (\ref{C1234jian} B)
%&:=I_{N,K_{2}}+\textcolor[rgb]{1.00,0.00,0.00}{\widehat{C}(s, \varsigma)}\|f\|_{H^{\varsigma}(\mathbb{R}^{d})}2^{-(N+1)\zeta}.
\end{array}\end{align}
where $(\ref{C1234jian} A)$ (with $\widehat{C}(s,\varsigma)\geq1$) and $(\ref{C1234jian} B)$ are derived  from Proposition  \ref{1stlemma} \eqref{ghgh1}
and \eqref{CVXZC}, respectively.
Similarly, we can prove that \begin{align}\label{niluohe} \sqrt{A_{N,\emph{\textbf{k}}, \emph{\textbf{k}}'}}\geq\sqrt{I_{N,\emph{\textbf{k}}'}}-\widehat{C}(s,\varsigma)2^{1-\zeta}\|f\|_{H^{\varsigma}(\mathbb{R}^{d})}
(\sqrt{d}(\|\Omega\|_{2,\sup}+M))^{\zeta}2^{-N\zeta}.
\end{align}
Combining \eqref{C1234jian} and  \eqref{niluohe} we have
\begin{align}
|\sqrt{A_{N,\emph{\textbf{k}}, \emph{\textbf{k}}'}}-\sqrt{I_{N,\emph{\textbf{k}}'}}|\leq\widehat{C}(s,\varsigma)2^{1-\zeta}\|f\|_{H^{\varsigma}(\mathbb{R}^{d})}
(\sqrt{d}(\|\Omega\|_{2,\sup}+M))^{\zeta}2^{-N\zeta}.
\end{align}
%First, by the  triangle inequality  we have
% \begin{align}\label{nengliangbijin}\begin{array}{lll} |I_{N,K_{i}}-\mathring{A}_{N,K_{1},K_{i}}|&\leq|f(2^{-N}K_{1})-f(2^{-N}K_{i})|\\
% &\leq
% 2^{1-\zeta}\frac{\|2^{-2N}K_{1}\|_{2}^{\zeta}}{(2\pi)^{D/2}}\|f\|_{H^{\varsigma}(\mathbb{R}^{d})}\\
% &\leq
% \frac{2^{1-\zeta}}{(2\pi)^{D/2}}2^{-N\zeta}\|\Omega\|^{\zeta}_{2, \sup}\|f\|_{H^{\varsigma}(\mathbb{R}^{d})},
% \end{array}\end{align}
% where the second inequality is derived from Lemma \ref{1stlemma} \eqref{ghgh1} and the third one from  $2^{-N}K_{1}\in \Omega$.
 Therefore,
  \begin{align}\label{HJJ}\begin{array}{lll} |I_{N,\emph{\textbf{k}}}-A_{N,\emph{\textbf{k}},\emph{\textbf{k}}'}|&=|\sqrt{A_{N,\emph{\textbf{k}}, \emph{\textbf{k}}'}}-\sqrt{I_{N,\emph{\textbf{k}}'}}||\sqrt{A_{N,\emph{\textbf{k}}, \emph{\textbf{k}}'}}+\sqrt{I_{N,\emph{\textbf{k}}'}}|\\
  %&|f(2^{-N}K_{1})-f(2^{-N}K_{i})|\\
 %&\leq
% 2^{1-\zeta}\frac{\|2^{-2N}K_{1}\|_{2}^{\zeta}}{(2\pi)^{D/2}}\|f\|_{H^{\varsigma}(\mathbb{R}^{d})}\\
 &\leq
 \widehat{C}(s,\varsigma)2^{1-\zeta}\|f\|_{H^{\varsigma}(\mathbb{R}^{d})}
(\sqrt{d}(\|\Omega\|_{2,\sup}+M))^{\zeta}2^{-N\zeta}\\
&\times 2(\sup_{\emph{\textbf{x}}\in \mathbb{R}^{d}}|f(\emph{\textbf{x}})|+\sup_{\emph{\textbf{x}}\in \Omega}|g_{N}(\emph{\textbf{x}})|)\\
 &\leq
 \widehat{C}(s,\varsigma)2^{2-\zeta}\|f\|_{H^{\varsigma}(\mathbb{R}^{d})}
(\sqrt{d}(\|\Omega\|_{2,\sup}+M))^{\zeta}2^{-N\zeta}\\
&\times\big(\widehat{C}(s,
\varsigma)\|f\|_{H^{\varsigma}(\mathbb{R}^{d})}+\sup_{\emph{\textbf{x}}\in \Omega}|g_{N}(\emph{\textbf{x}})|\big),
 \end{array}\end{align}
where the last inequality is derived from Proposition  \ref{1stlemma} \eqref{ghgh}.
Similarly, we can prove that
 \begin{align}\label{H5617JJ}\begin{array}{lll} |I_{N,\emph{\textbf{k}}}-A_{N,\emph{\textbf{k}},\emph{\textbf{k}}''}|
 %&=|\sqrt{A_{N,\emph{\textbf{k}}, \emph{\textbf{k}}'}}-\sqrt{I_{N,\emph{\textbf{k}}'}}||\sqrt{A_{N,\emph{\textbf{k}}, \emph{\textbf{k}}'}}+\sqrt{I_{N,\emph{\textbf{k}}'}}|\\
  %&|f(2^{-N}K_{1})-f(2^{-N}K_{i})|\\
 %&\leq
% 2^{1-\zeta}\frac{\|2^{-2N}K_{1}\|_{2}^{\zeta}}{(2\pi)^{D/2}}\|f\|_{H^{\varsigma}(\mathbb{R}^{d})}\\
 %&\leq
% \widehat{C}(s,\varsigma)2^{1-\zeta}\|f\|_{H^{\varsigma}(\mathbb{R}^{d})}
%(\sqrt{d}(\|\Omega\|_{2,\sup}+M))^{\zeta}2^{-N\zeta}\\
%&\times 2(\sup_{x\in \mathbb{R}^{d}}|f(x)|+\sup_{x\in \Omega}|g_{N}(x)|)\\
 &\leq
 \widehat{C}(s,\varsigma)2^{2-\zeta}\|f\|_{H^{\varsigma}(\mathbb{R}^{d})}
(\sqrt{d}(\|\Omega\|_{2,\sup}+M))^{\zeta}2^{-N\zeta}\\
&\times\big(\widehat{C}(s,
\varsigma)\|f\|_{H^{\varsigma}(\mathbb{R}^{d})}+\sup_{\emph{\textbf{x}}\in \Omega}|g_{N}(\emph{\textbf{x}})|\big).
 \end{array}\end{align}

Next we  estimate the error $\|\{\mathring{f}(2^{-N}\emph{\textbf{k}})-f(2^{-N}\emph{\textbf{k}}): \emph{\textbf{k}}\in  \Lambda_{\Omega,N}\}\|_{2}$
where $\mathring{f}(2^{-N}\emph{\textbf{k}})$ is derived from Approach   \ref{ppp55},
 \begin{align}\label{YDG}\begin{array}{lllllll} \Big\|\left[\begin{array}{cccccccccc}\Re(f(2^{-N}K_{1})) \\
\Im(f(2^{-N}K_{1}))\end{array}\right]-\left[\begin{array}{cccccccccc}\mathring{\Re}(f(2^{-N}K_{1})) \\
\mathring{\Im}(f(2^{-N}K_{1}))\end{array}\right]\Big\|_{2}\\
\leq\frac{\max\{|g_{N}(2^{-N}\emph{\textbf{k}})-g_{N}(2^{-N}\emph{\textbf{k}}')|, |\overline{g_{N}}(2^{-N}\emph{\textbf{k}})-\overline{g_{N}}(2^{-N}\emph{\textbf{k}}'')|\}}{|\Im[(g_{N}(2^{-N}\emph{\textbf{k}})-g_{N}(2^{-N}\emph{\textbf{k}}^{'}))(\overline{g_{N}}(2^{-N}\emph{\textbf{k}})-\overline{g_{N}}(2^{-N}\emph{\textbf{k}}^{''}))]|}
\sqrt{(I_{N,\emph{\textbf{k}}'}-A_{N,\emph{\textbf{k}},\emph{\textbf{k}}'})^{2}+(I_{N,\emph{\textbf{k}}''}-A_{N,\emph{\textbf{k}},\emph{\textbf{k}}''})^{2}} \quad  (\ref{YDG} A)\\
\leq\gamma\sqrt{(I_{N,\emph{\textbf{k}}'}-A_{N,\emph{\textbf{k}},\emph{\textbf{k}}'})^{2}+(I_{N,\emph{\textbf{k}}''}-A_{N,\emph{\textbf{k}},\emph{\textbf{k}}''})^{2}}\quad (\ref{YDG} B)\\
%\leq
%\gamma(\sqrt{d}(\|\Omega\|_{2,\sup}+M))^{\zeta}\widehat{C}(s,\varsigma)\\
%\quad \times\|f\|_{H^{\varsigma}(\mathbb{R}^{d})}\big(\widehat{C}(s,
%\varsigma)\|f\|_{H^{\varsigma}(\mathbb{R}^{d})}+\sup_{x\in \Omega}|g_{N}(x)|\big)2^{-N\zeta}\\
\leq
\gamma\widehat{C}(s,\varsigma)2^{2-\zeta}\sqrt{2}
(\sqrt{d}(\|\Omega\|_{2,\sup}+M))^{\zeta}\\
\quad \times\|f\|_{H^{\varsigma}(\mathbb{R}^{d})}\big(\widehat{C}(s,
\varsigma)\|f\|_{H^{\varsigma}(\mathbb{R}^{d})}+\sup_{k\geq1}\sup_{\emph{\textbf{x}}\in \Omega}|g_{k}(\emph{\textbf{x}})|\big)2^{-N\zeta},
(\ref{YDG} C)\end{array}\end{align}
where $(\ref{YDG} A)$, $(\ref{YDG} B)$ are derived from
\eqref{p45tgy} and  \eqref{bzxsd}, respctively, and $(\ref{YDG} C)$  is from  \eqref{HJJ} and \eqref{H5617JJ}.
Now we  choose $\widehat{C}_{0}(s,\varsigma,\Omega):=\widehat{C}^{2}(s,\varsigma)2^{5/2-\zeta}
(\sqrt{d}(\|\Omega\|_{2,\sup}+M))^{\zeta}<\infty$. Then it follows from $(\ref{YDG} C)$ and $\widehat{C}(s,\varsigma)\geq1$
that
\begin{align}\label{YD1G}\begin{array}{lllllll} \Big\|\left[\begin{array}{cccccccccc}\Re(f(2^{-N}K_{1})) \\
\Im(f(2^{-N}K_{1}))\end{array}\right]-\left[\begin{array}{cccccccccc}\mathring{\Re}(f(2^{-N}K_{1})) \\
\mathring{\Im}(f(2^{-N}K_{1}))\end{array}\right]\Big\|_{2}\\
\leq
\gamma\widehat{C}_{0}(s,\varsigma,\Omega)\|f\|_{H^{\varsigma}(\mathbb{R}^{d})}\big(\|f\|_{H^{\varsigma}(\mathbb{R}^{d})}+\sup_{k\geq1}\sup_{\emph{\textbf{x}}\in \Omega}|g_{k}(\emph{\textbf{x}})|\big)2^{-N\zeta}.\end{array}\end{align}
As in \eqref{8765} we have
\begin{align}\label{128765}\begin{array}{llllll}
 \|\big(\sum_{\emph{\textbf{k}}\in\mathbb{Z}^{d}}\mathring{f}(2^{-N}\emph{\textbf{k}})\phi(2^{N}\cdot-\emph{\textbf{k}})-f\big)\big|_{\Omega}\|_{L^2(\mathbb{R}^d)}\\
 %=\|\big(\sum_{\emph{\textbf{k}}\in \Lambda_{\Omega,N}}\mathring{f}(2^{-N}\emph{\textbf{k}})\phi(2^{N}\cdot-\emph{\textbf{k}})-f\big)\big|_{\Omega}\|_{L^2(\mathbb{R}^d)}\\
% \leq \|\big(\sum_{\emph{\textbf{k}}\in \Lambda_{\Omega,N}}f(2^{-N}\emph{\textbf{k}})\phi(2^{N}\cdot-\emph{\textbf{k}})-f\big)\big|_{\Omega}\|_{L^2(\mathbb{R}^d)}\\
%\ +\|\big(\sum_{\emph{\textbf{k}}\in \Lambda_{\Omega,N}}(f(2^{-N}\emph{\textbf{k}})-\mathring{f}(2^{-N}\emph{\textbf{k}}))\phi(2^{N}\cdot-\emph{\textbf{k}})\big)\big|_{\Omega}\|_{L^2(\mathbb{R}^d)}\\
 \leq \|\sum_{\emph{\textbf{k}}\in \mathbb{Z}^d}f(2^{-N}\emph{\textbf{k}})\phi(2^{N}\cdot-\emph{\textbf{k}})-f\|_{L^2(\mathbb{R}^d)}\quad (\ref{128765} A)\\
\ +\|\big(\sum_{\emph{\textbf{k}}\in \Lambda_{\Omega,N}}(f(2^{-N}\emph{\textbf{k}})-\mathring{f}(2^{-N}\emph{\textbf{k}}))\phi(2^{N}\cdot-\emph{\textbf{k}})\big)\big|_{\Omega}\|_{L^2(\mathbb{R}^d)}. \quad (\ref{128765} B)\end{array}\end{align}
By \eqref{bound67890}, $ (\ref{128765} A)$ is estimated as follows,
\begin{align}\begin{array}{lll} \label{1234bound67890123}
\displaystyle
\big\|\sum_{\emph{\textbf{k}}\in \mathbb{Z}^d}f(2^{-N}\emph{\textbf{k}})\phi(2^{N}\cdot-\emph{\textbf{k}})-f\big\|_{L^2(\mathbb{R}^d)} \leq C(s,
\varsigma)\|f\|_{H^{\varsigma}(\mathbb{R}^{d})} 2^{-N(\varsigma-s)/2}.
\end{array}
\end{align}
Additionally, $(\ref{128765} B)$ is estimated by
\begin{align}\label{KVBXCZA}\begin{array}{lll}
\|\big(\sum_{\emph{\textbf{k}}\in \Lambda_{\Omega,N}}(f(2^{-N}\emph{\textbf{k}})-\mathring{f}(2^{-N}\emph{\textbf{k}}))\phi(2^{N}\cdot-\emph{\textbf{k}})\big)\big|_{\Omega}\|_{L^2(\mathbb{R}^d)}\\
\leq
\|\big(\sum_{\emph{\textbf{k}}\in \Lambda_{\Omega,N}}|(f(2^{-N}\emph{\textbf{k}})-\mathring{f}(2^{-N}\emph{\textbf{k}}))|\phi(2^{N}\cdot-\emph{\textbf{k}})\big)\big|_{\Omega}\|_{L^2(\mathbb{R}^d)} \quad (\ref{KVBXCZA} A)\\
\leq\sup_{\emph{\textbf{k}}\in\Lambda_{\Omega,N}}\{|f(2^{-N}\emph{\textbf{k}})-\mathring{f}(2^{-N}\emph{\textbf{k}})|\}\|\big(\sum_{\emph{\textbf{k}}\in\Lambda_{\Omega,N}}\phi(2^{N}\cdot-\emph{\textbf{k}})\big)\big|_{\Omega}\|_{L^2(\mathbb{R}^d)}\\
\leq \gamma\widehat{C}_{0}(s,\varsigma,\Omega)\hbox{Vol}^{1/2}(\Omega)\|f\|_{H^{\varsigma}(\mathbb{R}^{d})}\big(\|f\|_{H^{\varsigma}(\mathbb{R}^{d})}+\sup_{k\geq1}\sup_{\emph{\textbf{x}}\in \Omega}|g_{k}(\emph{\textbf{x}})|\big)2^{-N\zeta}, \quad  (\ref{KVBXCZA} B)\\
=\gamma\widehat{C}_{1}(s,\varsigma,\Omega)\|f\|_{H^{\varsigma}(\mathbb{R}^{d})}\big(\|f\|_{H^{\varsigma}(\mathbb{R}^{d})}+\sup_{k\geq1}\sup_{\emph{\textbf{x}}\in \Omega}|g_{k}(\emph{\textbf{x}})|\big)2^{-N\zeta},
\end{array}
\end{align}
where $\widehat{C}_{1}(s,\varsigma,\Omega)=\widehat{C}_{0}(s,\varsigma,\Omega)\hbox{Vol}^{1/2}(\Omega)$, $(\ref{KVBXCZA} A)$ is from the nonegativity of $\phi$ and $(\ref{KVBXCZA} B)$ is from \eqref{YD1G} and
the unit decomposition \ref{yydnew}.
Combining \eqref{1234bound67890123} and \eqref{KVBXCZA}, the proof is completed.
\end{proof}

%\begin{rem} (1)
%By Proposition \ref{pppp675}, we can choose
% $ \nu_{N}\approx
%\frac{2^{N}\cdot\frac{\pi}{2}}{\|\Omega\|_{\inf}}$ such that the approximation in \eqref{908716}  is stable. (2)
%If $\nu_{N}=\frac{2^{N}\cdot\frac{\pi}{2}\vartheta}{\|\Omega\|_{\inf}}$ for a fixed $\vartheta\in (0,1]$, then the  measurement set
% in Approach \ref{ppp55} also enjoys the inheritance property  as that mentioned  in Remark \ref{herit}.
%\end{rem}

\subsection{A final remark: interpretation of Theorems \ref{firstmaintheorem} and \ref{secondmaintheorem} from the perspective of intensity difference}\label{interpretation}
 Although the inputs for  Approaches \ref{ppp4} and \ref{ppp55} are the interference intensities in \eqref{kkJJJ} and \eqref{nengliangqiumianbo123}, as implied in \eqref{qiujie255} and \eqref{qiujie2551}
what we need indeed  for computation is the intensity  difference. Particularly, in \eqref{qiujie255} both
$$I_{N,\emph{\textbf{k}}}-I_{N,\emph{\textbf{k}}'}=|f(2^{-N}\emph{\textbf{k}})+g_{N}(2^{-N}\emph{\textbf{k}})|^{2}-|f(2^{-N}\emph{\textbf{k}}+2^{-N}\emph{\textbf{e}}_{j})+g_{N}(2^{-N}\emph{\textbf{k}}+2^{-N}\emph{\textbf{e}}_{j})|^{2}$$
and
$$I_{N,\emph{\textbf{k}}}-I_{N,\emph{\textbf{k}}''}=|f(2^{-N}\emph{\textbf{k}})+g_{N}(2^{-N}\emph{\textbf{k}})|^{2}-|f(2^{-N}\emph{\textbf{k}}-2^{-N}\emph{\textbf{e}}_{j})+g_{N}(2^{-N}\emph{\textbf{k}}-2^{-N}\emph{\textbf{e}}_{j})|^{2}$$
can be interpreted mathematically as the difference of the intensity function $|f(\emph{\textbf{x}})+g_{N}(\emph{\textbf{x}})|^{2}$
at $2^{-N}\emph{\textbf{k}}$ along the coordinate  direction $\emph{\textbf{e}}_{j}$. In this sense, Theorems \ref{firstmaintheorem} \eqref{90876} and \ref{secondmaintheorem}
\eqref{908716} can be also interpreted as the recovery from intensity difference.
Such an   interpretation from  intensity difference
resembles  that of transport of intensity equation (TIE, a classical   noninterference method, c.f. \cite{TIE1,TIE2,TIE3}) which holds for PR problems associated with the
paraxial monochromatic coherent beam propagation. Particularly,
the target field is  $U(x,y, z)=\sqrt{I(x,y, z)}e^{\textbf{i}\phi(x, y)}, (x,y, z)\in \mathbb{R}^{3}$.
Then  the corresponding   TIE is expressed as
$$-\kappa\frac{\partial I(x, y, z)}{\partial z}=\nabla\cdot[I(x, y, z)\nabla\phi(x, y)],$$
where $\kappa$ is the wave number of $U$ and $\nabla$ is the  gradient.

\bibliographystyle{splncs03}

\bibliography{pls}

\begin{thebibliography}{99}

 \bibitem{Alaifari}{\ R. Alaifari,  P. Grohs, Phase Retrieval in the general setting of continuous frames for Banach spaces, {\it   SIAM Journal on Mathematical Analysis}, 49, 1895-1911, 2017.}

%49(3):1895¨C1911, 2017.
%
% Vol. 49, Iss. 3 (2017)10.1137/16M1071481



%


\bibitem{Ba1}{\ R. Balan, P. Casazza and   D. Edidin, On signal reconstruction without noisy phase, {\it Applied and Computational Harmonic Analysis}, 20, 345-356, 2006.}

\bibitem{Ba2}{\ R. Balan, Y. Wang, Invertibility and robustness of phaseless reconstruction, {\it Applied and Computational Harmonic Analysis}, 38(3), 469-488, 2015.}

%\bibitem{Ba3}{\ R. Balan, D. Zou,  On Lipschitz analysis and Lipschitz synthesis for the phase retrieval problem, {\it Linear Algebra and its Applications}, 496, 152-181, 2016.}
 \bibitem{Robert1}{\
R. Beinert, One-dimensional phase retrieval with additional interference measurements, {\it Results in Mathematics}, 72(1-2), 1-24, 2017.}

% \bibitem{Robert2}{\
%R. Beinert, G. Plonka, Ambiguities in one-dimensional discrete phase retrieval from Fourier magnitudes, {\it Journal of Fourier Analysis and Applications}, 21, 1169-1198, 2015.}

% \bibitem{Eldar}{\ T. Bendory, D. Edidin  and Y.C. Eldar, On signal reconstruction from FROG measurements,
% {\it Applied and Computational Harmonic Analysis,} 48(3), 1030-1044, 2020.}

 \bibitem{Ron}{\ C. de Boor, A. Ron,  Box splines revisited: Convergence and acceleration methods for the subdivision and the cascade, {\it Journal of Approximation Theory}, 150, 1-23, 2008.}


 \bibitem{FAAlight} T. E. Chambers, M. W. Hamilton, and I.  M. Reid, A low-cost digital holographic imager for calibration and validation of cloud microphysics remote sensing, {\it  Proc. SPIE 10001, Remote Sensing of Clouds and the Atmosphere XXI, 100010P,} 2016.

%\bibitem{Boashash}{\ B. Boashash, Time-frequency signal analysis, the choice of a method
%and its applications, {\it in Advances in Spectrum Estimation and Array
%Processing (S. Haykin, Ed.). Englewood Cliffs, NJ: Prentice-Hall,} 1, 418-517,
%1991.}

%\bibitem{CHEN} E. J. Cand$\grave{e}$s, P. R. Charlton and    H. Helgason,  Detecting highly oscillatory signals by chirplet path pursuit,  {\it  Applied and Computational Harmonic Analysis,}  24(1), 14-40, 2008.

 \bibitem{QiyuPR}{\ Y. Chen, C. Cheng, Q. Sun  and  H. Wang,
 Phase retrieval of real-valued signals in a shift-invariant space,
 {\it  Applied and Computational Harmonic Analysis}, 49, 56-73,  2020.}


 \bibitem{QiyuPR1}{\ C. Cheng, J. Jiang and   Q. Sun,
 Phaseless sampling and reconstruction of real-valued signals in shift-invariant spaces,
 {\it  Journal of Fourier Analysis and Applications,}  25, 1361-1394,  2019.}



   \bibitem{waveletBOOK2}{C. K. Chui,  An Introduction to Wavelets,
 {\it   Academic Press}, 1992. }

%   \bibitem{Cohen}{\ L. Cohen, Time-Frequency Analysis,
% {\it  Prentice Hall,}  1995.}

     \bibitem{waveletBOOK3}{I. Daubechies,  Ten Lectures on Wavelets.
 {\it CBMS-NSF Series in Applied Mathematics, SIAM, Philadelphia}, 1992. }

\bibitem{Enciso} A. Enciso, N.  Kamran, A singular initial-boundary value problem for nonlinear wave
equations and holography in asymptotically anti-de Sitter spaces, {\it  Journal de Math\'{e}matiques Pures et Appliqu\'{e}es,}   103, 1053-1091, 2015.

\bibitem{Fannjiang} D. A Barmherzig, J. Sun, P. N  Li, T J Lane,  and E. J Cand\`{e}s, Holographic phase retrieval and reference design, {\it  Inverse Problems,}    35, 094001 (30pp),  2019.


 \bibitem{Fienup1}{\ J. R. Fienup, Phase retrieval algorithms: A comparison, {\it   Applied Optics}, 21(15), 2758-2769,  1982.}

\bibitem{Fienup2}{\ J. R. Fienup, Reconstruction of an object from the modulus of its
Fourier transform, {\it   Optics Letters}, 3(1), 27-29,  1978.}



\bibitem{D.Gabor}{\ D. Gabor, A new microscopic principle, {\it Nature}, 161, 777-778, 1948.}

\bibitem{Xuzhiqiang}{\ B. Gao, Q. Sun, Y. Wang, Z. Xu, Phase retrieval from the magnitudes of affine linear measurements,
 {\it   Advances in Applied Mathematics}, 93,  121-141,  2018.}

 \bibitem{Dynamicboject1}{\ G. T Gullberg, B. W Reutter
A. Sitek, J. S Maltz, Dynamic single photon emission computed tomography-basic principles and cardiac applications, {\it   Physics in Medicine and Biology}, 55,  111-191, 2010.}

\bibitem{Gaussian}{\ K. Gr\"{o}chenig,  Phase-Retrieval in Shift-Invariant Spaces with Gaussian Generator, {\it  Journal of Fourier Analysis and Applications,} 26(3), 52,
2020.}


\bibitem{Yamaguchi1}{\
P. Guo, A. J. Deva, Digital microscopy using phase-shifting digital holography with two reference waves, {\it Optics Letters}, 29, 857-859, 2004.}

 \bibitem{Fourieroptics}{\ J.W. Goodman, Introduction to Fourier optics,
3rd Edition, Roberts \& Co., Greenwood Village, 2005. }

\bibitem{waveletBOOK1}{\ B. Han, Framelets and wavelets: Algorithms, analysis, and applications, {\it Applied
and Numerical Harmonic Analysis,} Birkh\"{a}user/Springer, Cham, 2017. xxxiii +724
pp.}

\bibitem{Hanbin1}{B. Han, Z. Shen, Dual wavelet frames and Riesz bases
in Sobolev spaces, {\it Constructive Approximation,} 29, 369-406, 2009.}

 \bibitem{Heinosaarri} T. Heinosaarri, L. Mazzarella and    M.  M.Wolf, Quantum tomography under
prior information, {\it  Communications in Mathematical Physics,}   318, 355-374, 2013.

 \bibitem{Huangmeng} M. Huang, Z. Xu, Phase retrieval from the norms of affine transformations, {\it  Advances in Applied Mathematics,}   130, 102243, 2021.







\bibitem{Roadmap}{\ B. Javidi, A. Carnicer, A. Anand, G. Barbastathis, W. Chen, P. Ferraro, J. W. Goodman, R. Horisaki, K. Khare, M. Kujawinska, R.
A. Leitgeb, P. Marquet, T. Nomura, A. Ozcan, Y. Park, G. Pedrini, P. Picart, J. Rosen, G. Saavedra, N. T. Shaked, A. Stern, E.
Tajahuerce, L. Tian, G. Wetzstein, and M. Yamaguchi, Roadmap on digital holography [Invited], {\it Optics Express}, 29, 35078-
35118, 2021.}

%\bibitem{phasederivative}{\ J. Jeong, G.S. Cunningham, W.J. Williams,  The discrete-time phase derivative as a definition of discrete instantaneous frequency and its relation to discrete time-frequency distributions, {\it
%    IEEE Transactions on Signal Processing,} 43(1), 341-344,
%1995.}

 \bibitem{Jiarongqing}{\ R. Jia, Approximation properties of multivariate wavelets,
 {\it Mathematics of Computation,} 67, 647-665,  1998.}



 \bibitem{holographic}{\ M. K. Kim, Principles and techniques of digital holographic microscopy,
 {\it SPIE Reviews,} 1, 018005, 2010.}

\bibitem{Lai}{\ C.K. Lai, F. Littmann,  E. Weber, Conjugate phase retrieval in Paley-Wiener space, {\it Journal of Fourier Analysis and Applications}, 27, Article number: 89, 2021.}

\bibitem{LISUN}
Y. Li,  W. Sun, Random phaseless sampling for causal signals in
  shift-invariant spaces: a zero distribution perspective,  \emph{IEEE
  Transactions on Signal Processing}, 68,  5473-5486, 2020.

\bibitem{YMH}{\ Y. Li, Y. Ma, D. Han, FROG-measurement based phase retrieval for analytic signals, {\it Applied and Computational Harmonic Analysis,} 55, 199-222, 2021.}

\bibitem{fagep}{Y. Li,  Sampling
approximation by framelets in Sobolev space  and its application in
modifying interpolating error, {\it Journal of  Approximation  Theory,} 175, 43-63,
2013.}

\bibitem{LiSCI}{\ Y. Li, D. Han, S. Yang, G. Huang,  Nonuniform sampling and  approximation in Sobolev space from the perturbation of framelet system, {\it Science China Mathematics,} 64,  351-372, 2021.}

\bibitem{Maxfield} H. Maxfield, S. F Ross,  and B. Way, Holographic partition functions and phases
for higher genus Riemann surfaces, {\it Classical and Quantum Gravity,} 33,  125018, 2016.

 \bibitem{crystallography}{\ J. Miao, T. Ishikawa, I.K. Robinson and  M.M. Murnane, Beyond crystallography: Diffractive imaging using coherent x-ray light sources, {\it   Science}, 348(6234), 530-535,  2015.}

 \bibitem{Micchelli}{\
 C. A. Micchelli, Mathematical aspects of geometric modeling, {\it SIAM, Philadelphia, PA},
 1995.}

 \bibitem{Checkboard}{\
T. Nobukawa, T. Muroi, Y. Katano, N. Kinoshita,  N. Ishii, Single-shot phase-shifting incoherent digital holography with multiplexed checkerboard phase gratings, {\it Optics Letters}, 43, 1698-1701, 2018.}

\bibitem{Pohl}{\ V. Pohl, F. Yang,  H. Boche, Phaseless signal recovery in infinite dimensional spaces using structured modulations, {\it   Journal of Fourier Analysis and Applications}, 20, 1213-1233, 2014.}

  \bibitem{reasonforPR}{\ Y. Shechtman, Y. C. Eldar, O. Cohen, H. N. Chapman, J. Miao  and M. Segev,
  Phase retrieval with application to optical imaging, {\it   IEEE Signal Processing Magazine}, 32(3), 87-109,  2015.}

  \bibitem{Dynamicboject}{\ Y. Shi, W. C. Karl,
  Tomographic reconstruction of dynamic objects, {\it   Proceedings of SPIE-The International Society for Optical Engineering}, 151-160,  2003.}



  \bibitem{Matla}{M. St$\acute{e}$phane, A Wavelet Tour of Signal Processing, {\it Elsevier Inc.,}   2009.}

   \bibitem{Wenchangsun}{\ W.  Sun,  Local  and  global  phaseless  sampling  in  real  spline  spaces,
 {\it Mathematics of Computation,} 90,  1899-1929, 2021.}

 \bibitem{TIE1}{\
M. R. Teague,  Irradiance moments: their propagation and use for unique retrieval of phase, {\it Journal of the Optical Society of America}, 72(9), 1199-1209, 1982.}

 \bibitem{TIE2}{\ M. R. Teague,  Deterministic phase retrieval: a Green's function solution, {\it Journal of the Optical Society of America}, 73(11), 1434-1441, 1983.}

\bibitem{Iserles}{\ G. Thakur,  Reconstruction of bandlimited functions
from unsigned samples, {\it  Journal of Fourier Analysis and Applications,} 17(4), 720-732,
2011.}


%\bibitem{Zhangxy}{ I. Yamaguchi, J. Kato, S. Ohta, J. Mizuno, Image formation in phase-shifting digital holography and applications to microscopy,
%{\it Applied Optics}, 40, 6177¨C6186, 2001.}

  \bibitem{ Johan De Villiers}{\
 J. De Villiers, On refinable functions and subdivision with positive masks, {\it Advances in Computational Mathematics},
 24, 281-295, 2006.}


\bibitem{WYZ}{ Y. Wang, Y. Zhen, H. Zhang,
Y. Zhang, Study on digital holography with single phase-shifting operation,
{\it Chinese Optics Letters}, 2, 141-143, 2004.}

 \bibitem{Yangwang}{\
Y. Wang, Subdivision schemes and refinement equations with
nonnegative masks, {\it Journal of Approximation Theory},
 113, 207-220, 2001.}

  \bibitem{Xu2}{\ Y. Xia,   Z. Xu, The recovery of complex sparse signals from few phaseless measurements, {\it  Applied and Computational Harmonic Analysis,}   50, 1-15, 2021.}





\bibitem{Zhang}{\ I. Yamaguchi, T. Zhang, Phase-shifting digital holography, {\it Optics Letters}, 22, 1268-1270, 1997.}







\bibitem{Yamaguchi0}{\
I. Yamaguchi, J. Kato, S. Ohta, J. Mizuno, Image formation in phase-shifting digital holography and applications to microscopy, {\it
Applied Optics}, 40,  6177-6186, 2001.}






 \bibitem{Liguifang}{\
J. Zhang, Y. Xie, G. Li, Y. Ye, B. E. A. Saleha, Single-shot phase-shifting digital holography, {\it Optical Engineering}, 53(11), 112316, 2014.}

 \bibitem{Zhouxinlong}{\
X. Zhou, Positivity of refinable functions defined by nonnegative finite masks, {\it Applied and Computational Harmonic Analysis},
27(2),  133-156, 2009.}









 \bibitem{TIE3}{\ C. Zuo,   J.  Li, J.  Sun, Y. Fan, J. Zhang, L. Lu, R. Zhang, B. Wang, L. Huang, Q. Chen,  Transport of intensity equation: a tutorial, {\it Optics and Lasers in Engineering}, 135, 106187, 2020.}





%
%
%
%
\end{thebibliography}

\end{document}